\documentclass[aos,preprint]{imsart}
\setattribute{journal}{name}{}

\RequirePackage[round]{natbib}
\RequirePackage[colorlinks,citecolor=blue,urlcolor=blue, linkcolor=blue]{hyperref}

\usepackage{bbm}
\usepackage{comment}
\usepackage{amsmath, amsthm, amssymb}
\usepackage{algorithm, algorithmicx}
\usepackage{graphicx}
\usepackage{subfigure}
\usepackage{lineno}
\usepackage[left=2.5cm,right=2.5cm, bottom=3cm]{geometry}
\usepackage[usenames,dvipsnames]{color}
\usepackage{multirow}
\usepackage{booktabs}

\let\hat\widehat
\newtheorem{thm}{Theorem}
\newtheorem{lem}[thm]{Lemma}
\newtheorem{cor}[thm]{Corollary}

\newtheorem{thm3}{Theorem}
\newtheorem{example}[thm3]{Example}

\newtheorem{thm4}{Theorem}
\newtheorem{definition}[thm4]{Definition}

\theoremstyle{remark}
\newtheorem{remark}{Remark}

\newcommand\K{\mathbb{K}}
\newcommand\R{\mathbb{R}}

\newcommand\E{\mathbb{E}}

\newcommand\cL{{\cal L}}

\catcode`@=11
\newskip\beforeproofvskip
\newskip\afterproofvskip
\beforeproofvskip=\medskipamount
\afterproofvskip=\bigskipamount

\def\prooftag{Proof}
\def\proofskip{\enspace}

\def\proof{\@ifnextchar[{\@@proof}{\@proof}}  
\def\@startproof{\par\vskip\beforeproofvskip\leavevmode}
\def\@proof{\@startproof{\scshape\prooftag.}\proofskip}
\def\@@proof[#1]{\@startproof {\scshape\prooftag #1.}\proofskip}

\catcode`@=12

\let\hat\widehat
\let\tilde\widetilde

\begin{document}

\begin{frontmatter}

\title{Generalized Cluster Trees and Singular Measures}
\runtitle{Generalized Cluster Trees}

\begin{aug}
  \author{\fnms{Yen-Chi}
    \snm{Chen}\ead[label=e1]{yenchic@uw.edu}}
  \affiliation{Department of Statistics\\University of Washington}
  \runauthor{Y.-C. Chen}

  \address{Department of Statistics\\University of Washington\\
Box 354322\\	Seattle, WA 98195 \\
          \printead{e1}}
        \today
\end{aug}

\begin{abstract}
In this paper, we study the $\alpha$-cluster tree ($\alpha$-tree) under both singular and nonsingular measures.
The $\alpha$-tree uses probability contents within a set created by the ordering of points
to construct
a cluster tree so that it is well-defined even for singular measures.
We first derive the convergence rate for a density level set around critical points,
which leads to the convergence rate for estimating an $\alpha$-tree under nonsingular measures.
For singular measures, we study how the kernel density estimator (KDE) behaves
and prove that the KDE is not uniformly consistent but pointwise consistent after rescaling. 
We further prove that the estimated $\alpha$-tree fails to converge in the $L_\infty$ metric but is still consistent
under the integrated distance.
We also observe a new type of critical points--the dimensional critical points (DCPs)--of a singular measure.
DCPs are points that contribute to cluster tree topology but 
cannot be defined using density gradient. 
Building on the analysis of the KDE and DCPs, 
we prove the topological consistency of an estimated $\alpha$-tree.
\end{abstract}



\end{frontmatter}

\section{Introduction}

Given a function $f$ defined on a smooth manifold $\mathcal{M}$, the cluster tree
of $f$ is a tree structure representing the creation and merging of connected components
of a level set $\{x: f(x)\geq f_0\}$ when we move down the level $f_0$ \citep{Stuezle03,klemela2004visualization}. 
Because cluster trees keep track of the connected components of level sets,
the shape of a cluster tree contains topological information about the underlying function $f$.
Moreover, a cluster tree can be displayed on a two-dimensional plane regardless of the dimension of $\mathcal{M}$,
which makes it an attractive approach for visualizing $f$.
Figure~\ref{fig::exTree} provides an example illustrating the construction of a cluster tree in a $1D$ Euclidean space.
In this paper, we focus on the case where $f \equiv f_P$ is some function of the underlying distribution $P$.
In this context, the cluster tree of $f$ reveals information about $P$.

\begin{figure}
\includegraphics[width=1.6in]{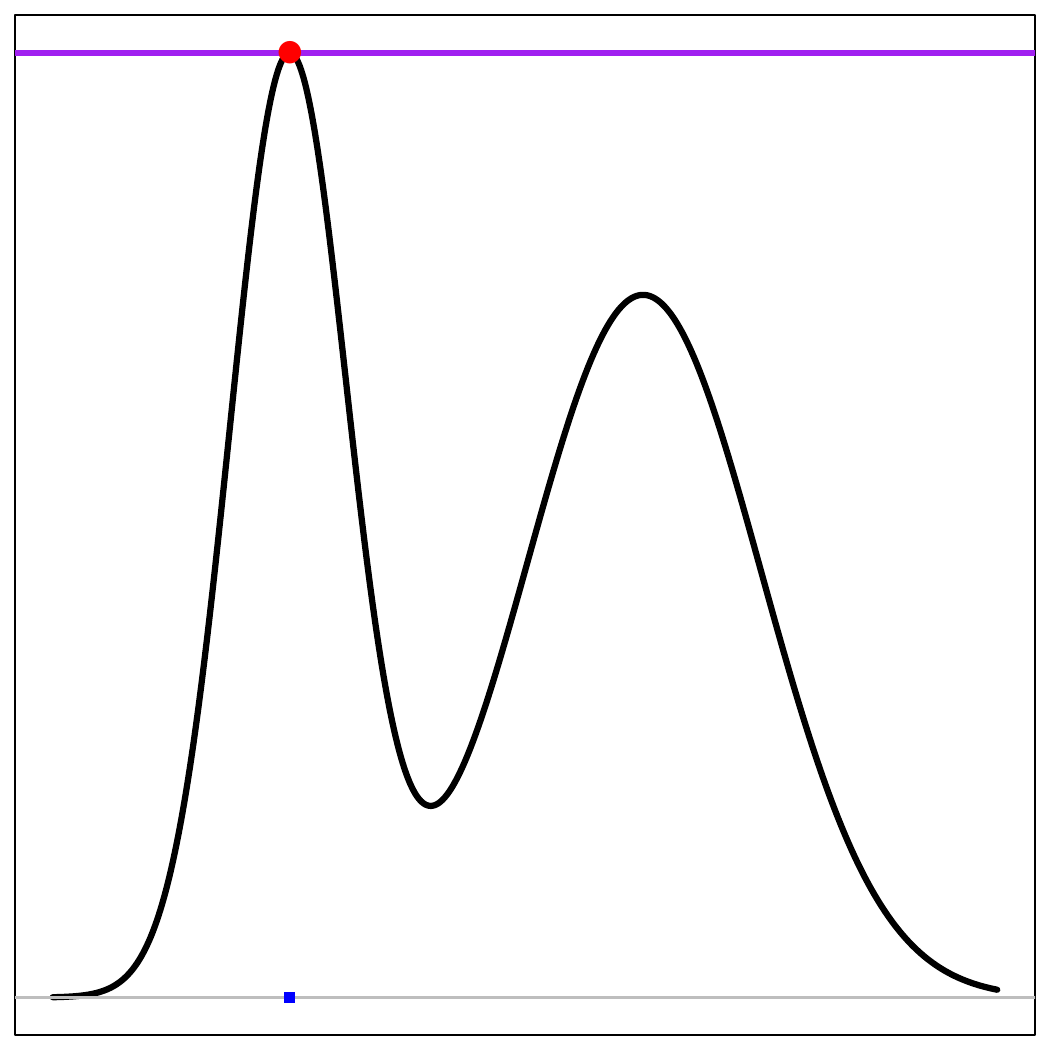}
\includegraphics[width=1.6in]{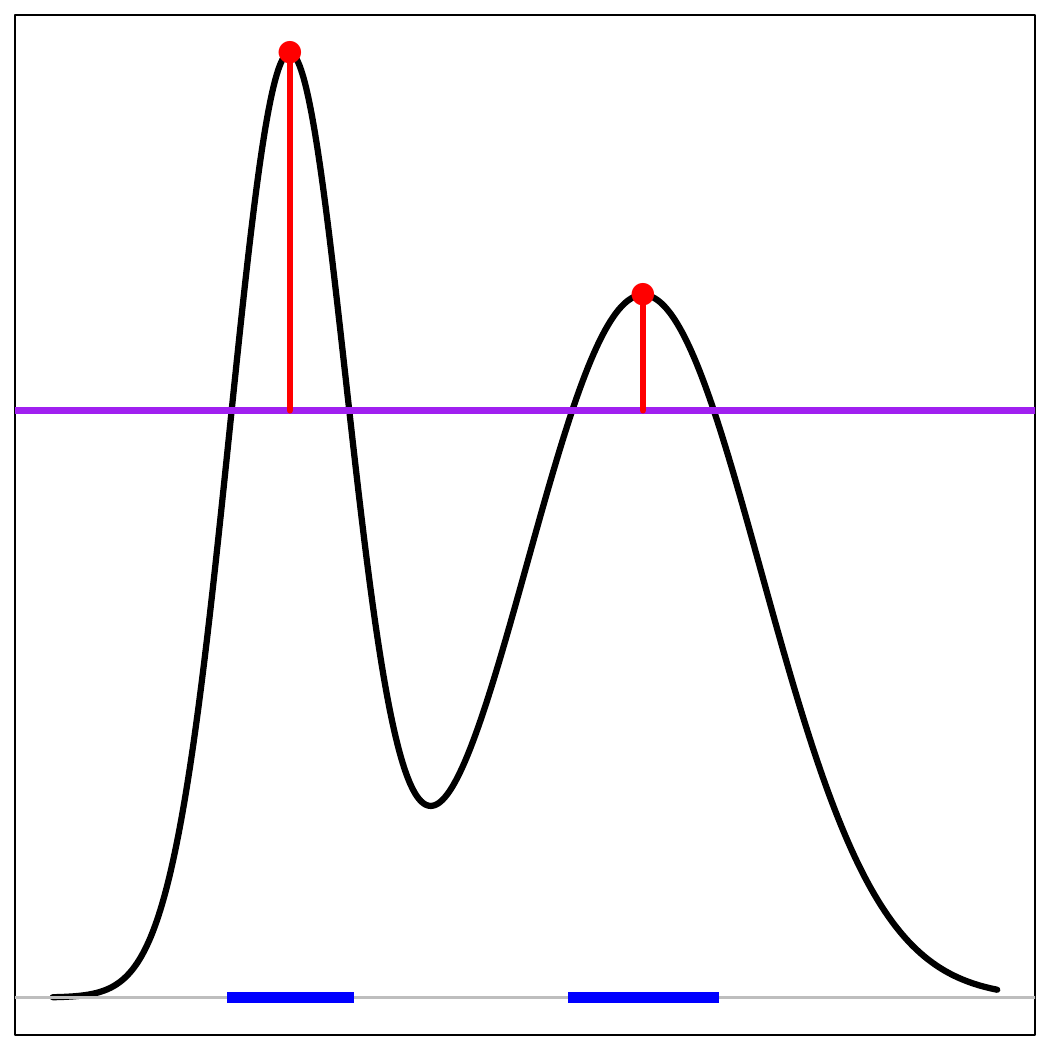}
\includegraphics[width=1.6in]{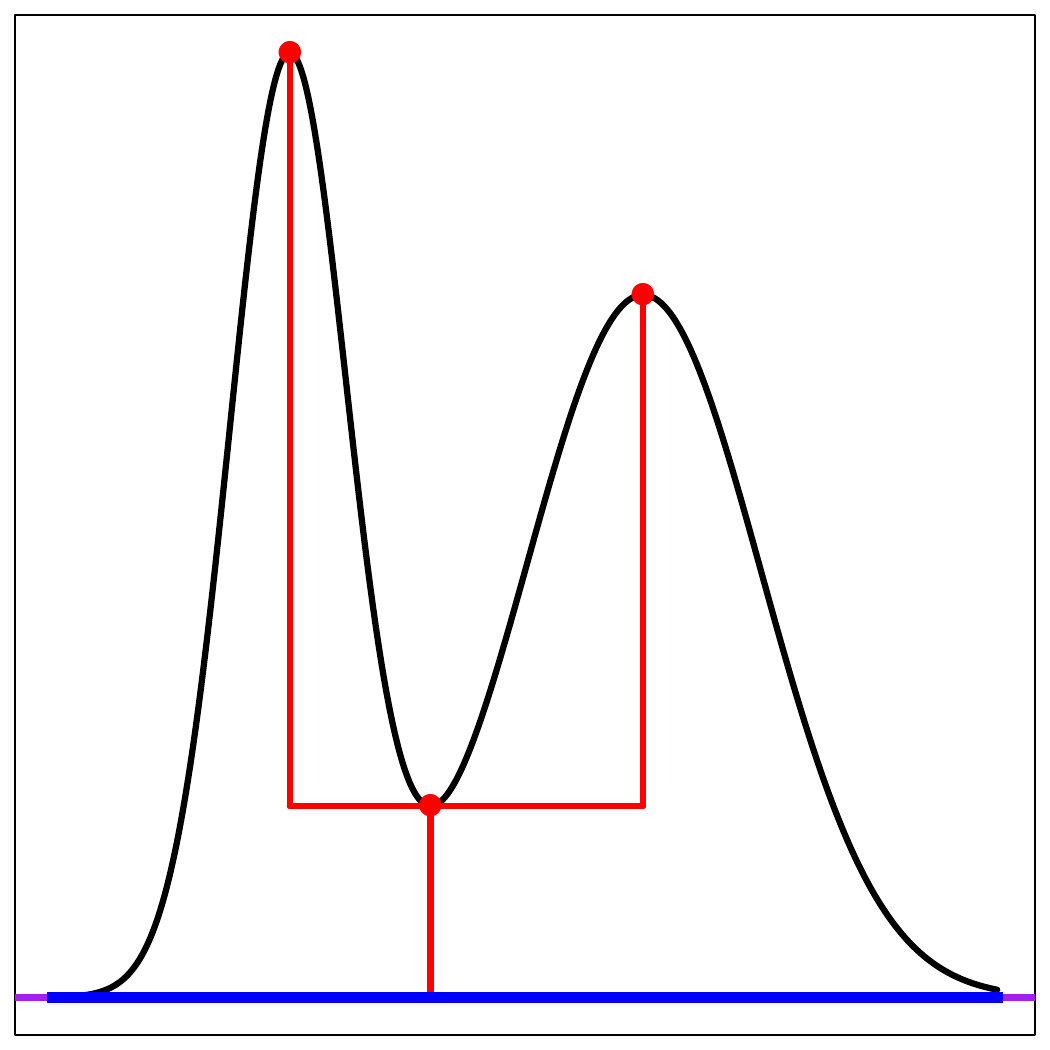}
\caption{
An example of constructing a cluster tree in $d=1$ case.
The purple horizontal lines indicates the level $f_0$ we are using in each panel. 
The blue region at the bottom indicates the corresponding level set $\{x: f(x)\geq f_0\}$. 
From left to right, we gradually decrease the level $f_0$ and use red lines to indicate
how the connected components evolve. 
The resulting red tree in the right panel is a cluster tree.
}
\label{fig::exTree}
\end{figure}

Most of cluster trees being studied in the literature
are the $\lambda$-tree of a distribution 
\citep{balakrishnan2012,chaudhuri2010rates,chaudhuri2014consistent,chen2016statistical,kpotufe2011pruning,Stuezle03}. 
The $\lambda$-tree of a distribution is the cluster tree of the density function $p$ of
that distribution.
In this case, the tree structure contains the topological information of $p$
and we can use the $\lambda$-tree to visualize a multivariate density function.
When we use an $\lambda$-tree for visualization purposes, it is also called a density tree
\citep{klemela2004visualization,klemela2006visualization,klemela2009smoothing}.

\cite{Kent2013} proposed a new type of cluster tree of a distribution--the 
$\alpha$-tree. The $\alpha$-tree uses the function $\alpha(x) = P(\{y: p(y)\leq p(x)\})$ to construct a cluster tree.
Note that such a function is also called \emph{density ranking} in our following paper \citep{chen2017measuring}
and it shows great potential in analyzing GPS datasets.
When the distribution is nonsingular and smooth, the $\alpha$-tree and the $\lambda$-tree
are topologically equivalent (Lemma~\ref{lem::tree}),
so they both provide similar topological information for the underlying distribution.
To estimate an $\alpha$-tree, we use the cluster tree of the function estimator $\hat{\alpha}_n(x) = \hat{P}_n(\{y: \hat{p}_n(y)\leq \hat{p}_n(x)\})$
where $\hat{P}_n$ is the empirical measure and $\hat{p}_n$ is the kernel density estimator (KDE).
Namely, 
we first use the KDE to estimate the density of each data point
and count the number of data points with a density below the density of that given point.

When a distribution is singular, the $\lambda$-tree is ill-defined because of the lack of a probability density function,
but the $\alpha$-tree is still well-defined under a mild modification.
For an illustrating example, see Figure~\ref{fig::Tree}.
These are
random samples from a distribution mixed with a point mass at $x=2$ with a probability of $0.3$
and a standard normal distribution with a probability of $0.7$. 
Thus, these samples are from a singular distribution.
We generate $n=5\times 10^3$ (left), $5\times 10^5$ (middle), and $5\times 10^7$ (right) data points
and estimate the density using the KDE with the smoothing bandwidth selected by the default rule in R.
The estimated density and $\lambda$-trees (red trees) are displayed in the top row.
It can be seen that when the sample size increases, $\lambda$-trees become degenerated. 
This is because there is no population $\lambda$-tree for this distribution. 
However, the $\alpha$-trees are stable in all three panels (see the bottom row of Figure~\ref{fig::Tree}).

\begin{figure}
\includegraphics[width=1.6in]{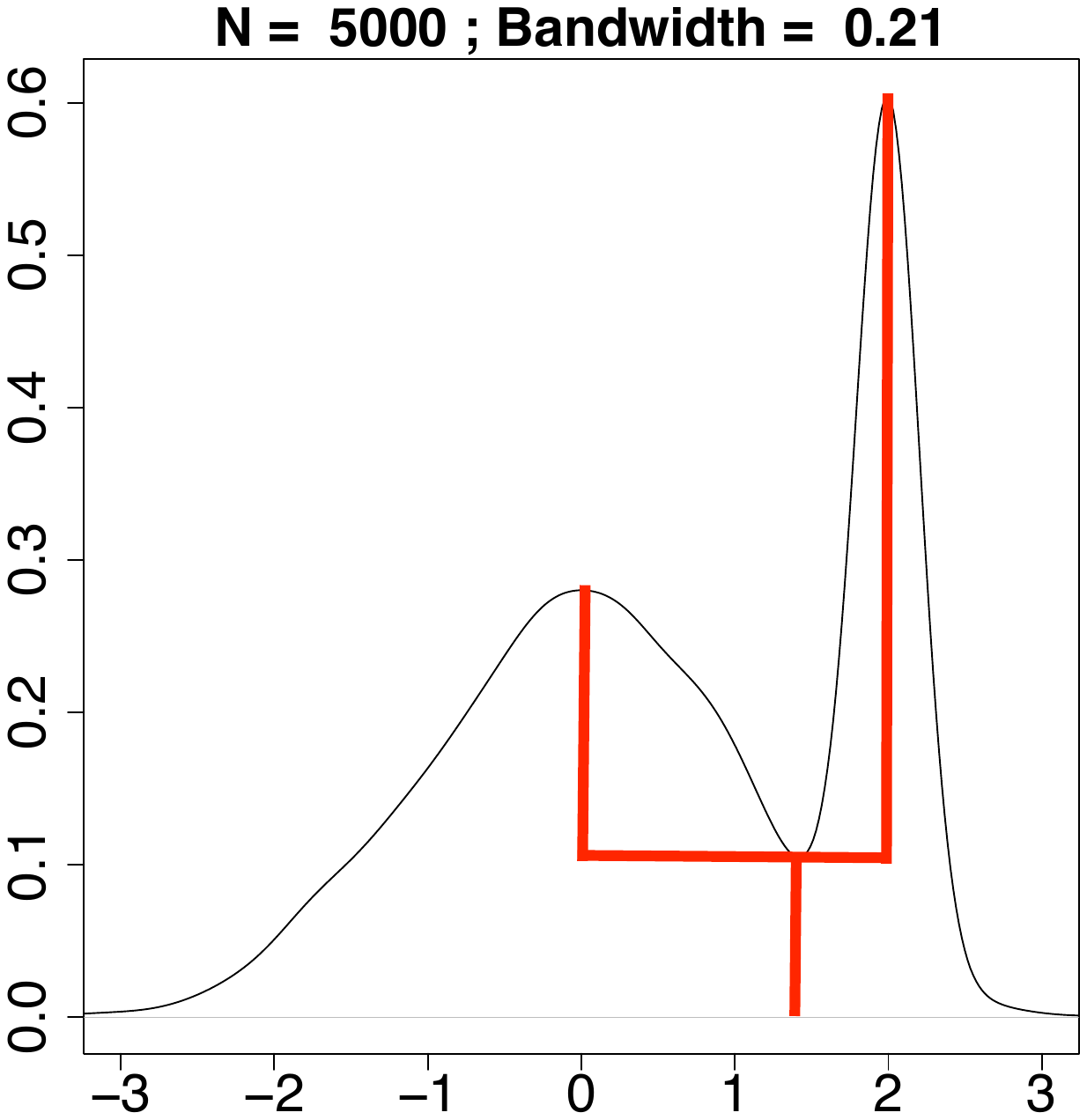}
\includegraphics[width=1.6in]{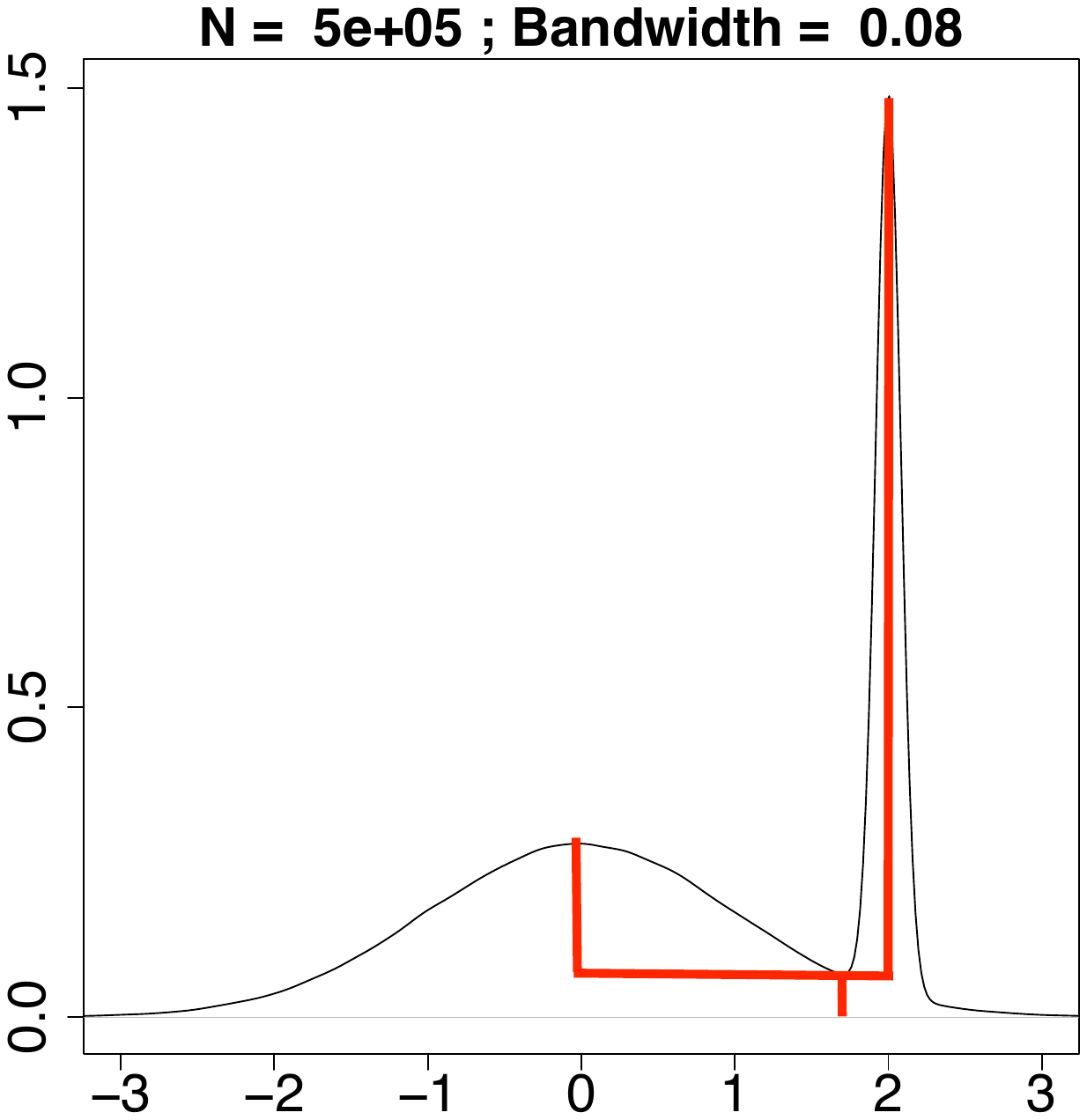}
\includegraphics[width=1.6in]{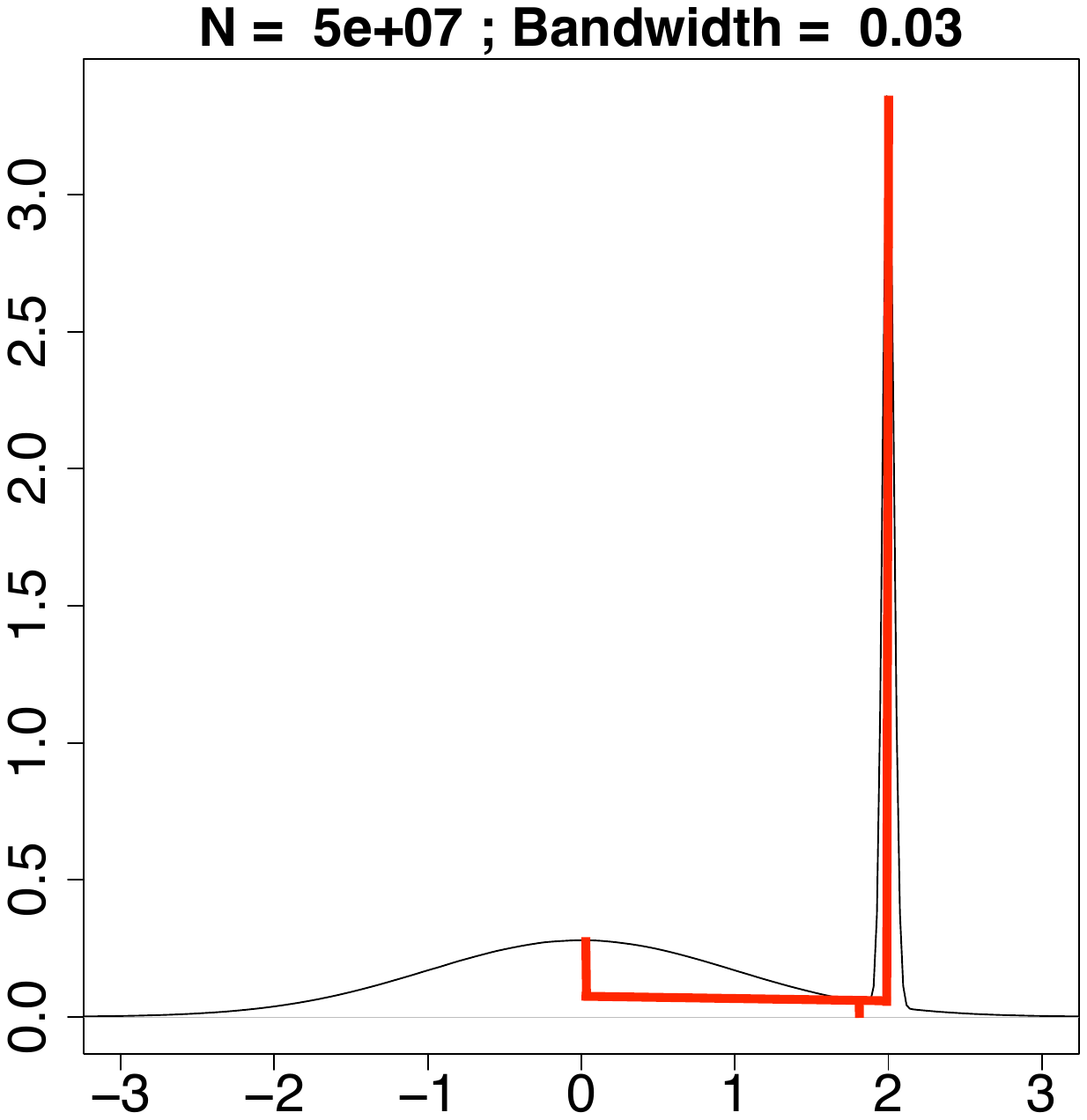}\\
\includegraphics[width=1.6in]{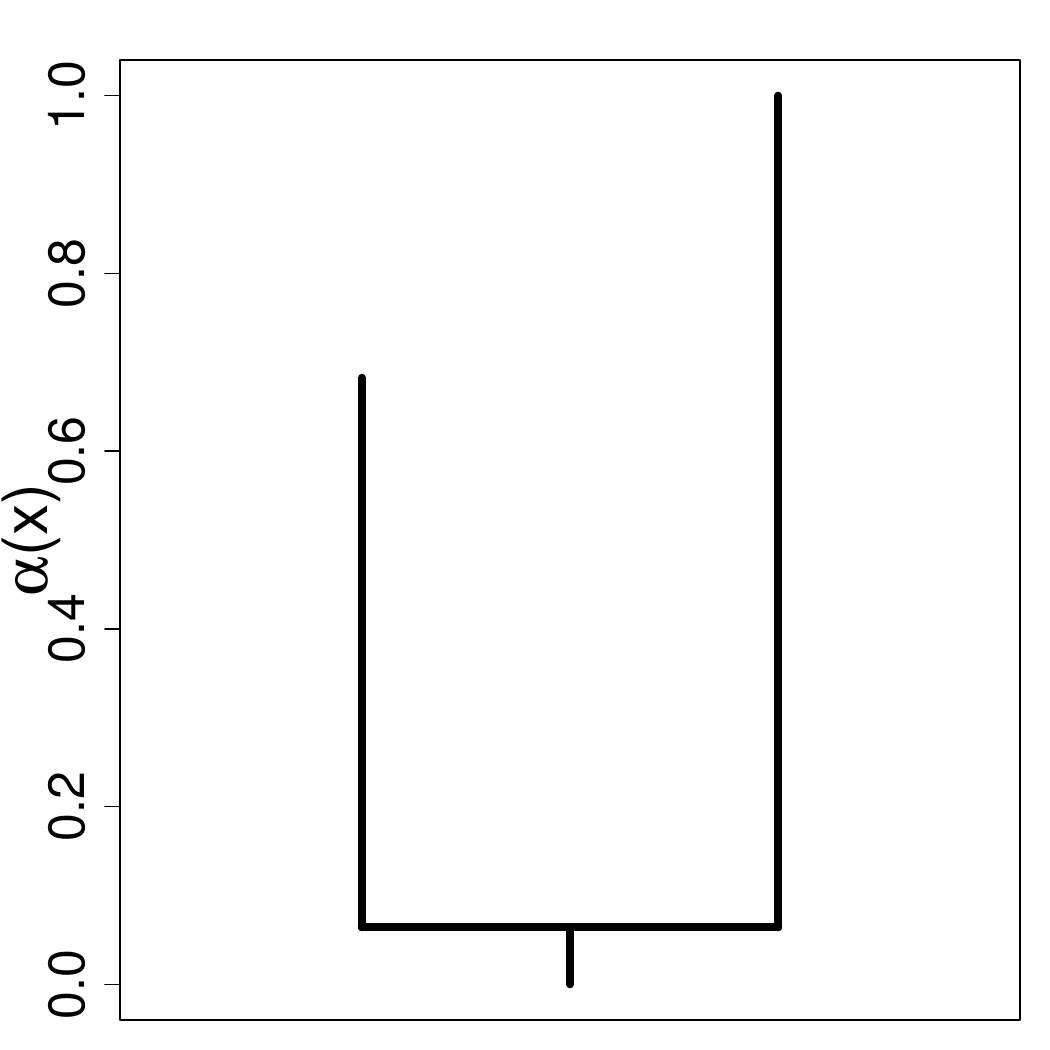}
\includegraphics[width=1.6in]{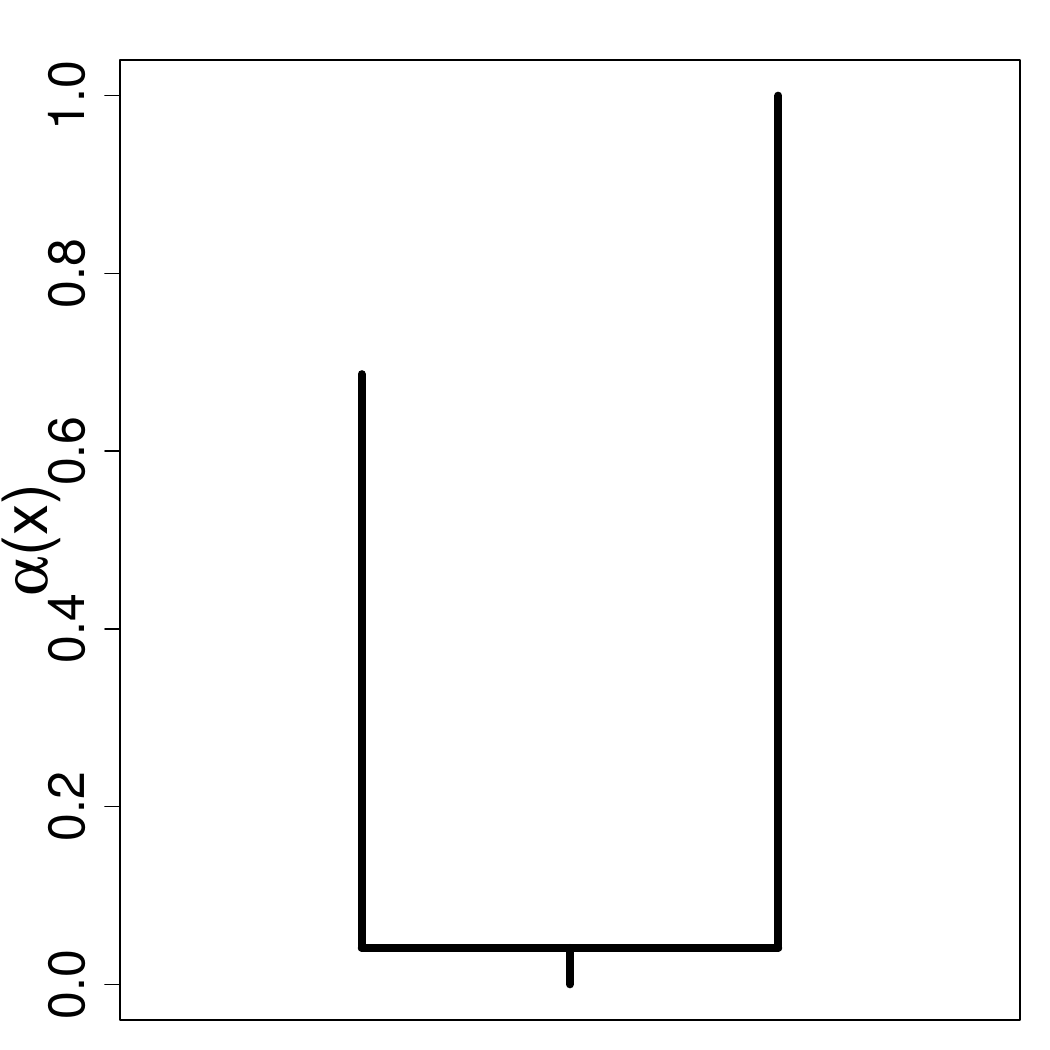}
\includegraphics[width=1.6in]{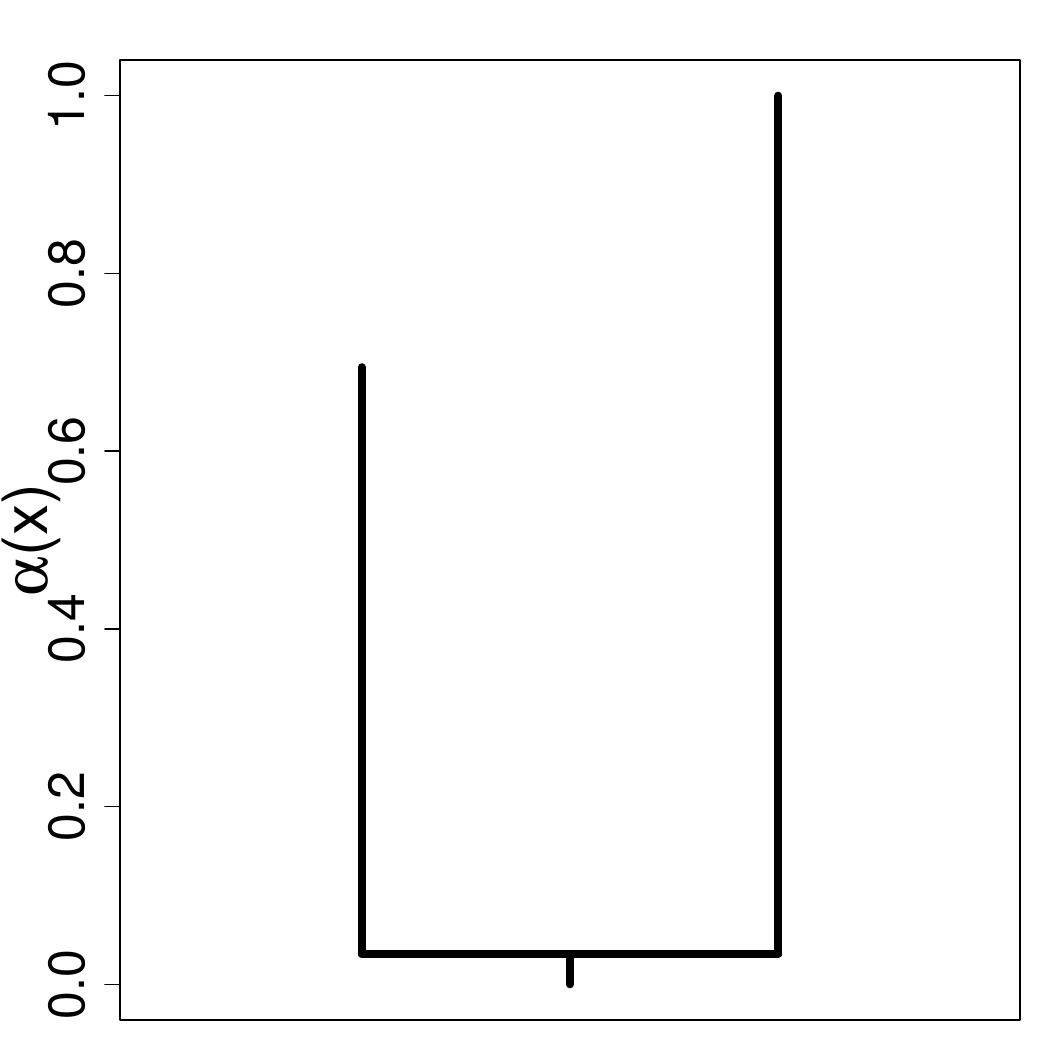}
\caption{
Example of the estimated $\lambda$-tree and $\alpha$-tree of a singular distribution.
This is a random sample from a singular distribution where with a probability of $0.3$, it puts a point mass at $x=2$,
and with a probability of $0.7$, we sample from a standard normal.
The top panel shows the density estimated by the KDE
and the red tree structure corresponds to the estimated $\lambda$-tree.
The bottom panel displays the estimated $\alpha$-tree.
From left to right, we increase the sample size from $5\times 10^3$, $5\times 10^5$, to $5\times 10^7$.
Because the distribution is singular, there is no population $\lambda$-tree so 
when the smoothing bandwidth decreases (when the sample size increases), the estimated $\lambda$-tree
is getting degenerated. 
On the other hand, the estimated $\alpha$-trees remain stable regardless of the smoothing bandwidth.
Note that every level set in a $\lambda$-tree corresponds to a level set in the $\alpha$-tree
but the value of the corresponding level ($Y$-axis) will be different. 
The function used in constructing a $\lambda$-tree may have an unbounded range when the sample size goes to infinity,
whereas the function for building an $\alpha$-tree always has a range of $[0,1]$. 
}
\label{fig::Tree}
\end{figure}

{\em Main Results.}
The main results of this paper are summarized as follows:
\begin{itemize}
\item When the distribution is nonsingular,
	\begin{itemize}
	\item[1.] We derive the convergence rate for the estimated level set
	when the level equals the density value of a critical point (Theorem~\ref{thm::critical}).
	\item[2.] We derived the convergence rate of $\hat{\alpha}_n$ (Theorem~\ref{thm::alpha_non}).
	\end{itemize}
\item When the distribution is singular,
	\begin{itemize}
	\item[3.] We propose a framework that generalizes $\alpha(x)$
	to define the $\alpha$-tree (Section~\ref{sec::singular}).
	\item[4.] We show that after rescaling, the KDE is pointwise, but 
	not uniformly, consistent (Theorem~\ref{thm::KDE}).
	\item[5.] We prove that $\hat{\alpha}_n$ is inconsistent under the $L_\infty$ metric (Corollary~\ref{cor::alpha_infty_s})
	but consistent under the $L_1$ and $L_1(P)$ distance (Theorem~\ref{thm::alpha::Pro}).
	\item[6.] We identify a new type of critical points, the dimensional critical points (DCPs), which also contribute
	to the change in cluster tree topology. We analyze their properties in Lemma~\ref{lem::DCP::p}, \ref{lem::DCP}, and \ref{lem::DCP_level}.
	\item[7.] We demonstrate that the estimated $\alpha$-tree $T_{\hat{\alpha}_n}$
	is topologically equivalent to the population $\alpha$-tree with probability exponentially converging to $1$ (Theorem~\ref{thm::top::p}).
	\end{itemize}
\end{itemize}

{\em Related Work.}
There is extensive literature on theoretical aspects of the $\lambda$-tree.
Notions of consistency are analyzed in \cite{Hartigan81,chaudhuri2010rates,chaudhuri2014consistent,eldridge2015beyond}.
The convergence rate and
the minimax theory are studied in \cite{chaudhuri2010rates,balakrishnan2012,chaudhuri2014consistent}.
\cite{chen2016statistical} study how to perform statistical inference for a $\lambda$-tree.
The cluster tree is also related to the topological data analysis \citep{carlsson2009topology,edelsbrunner2012persistent,wasserman2016topological}.
In particular, a cluster tree contains information about the zeroth order homology groups 
\citep{cohen2007stability,fasy2014confidence,bobrowski2017topological,bubenik2015statistical}. 
In our analysis, we generalize the Morse theory to a non-smooth and even discontinuous function.
\cite{baryshnikov2013min} also generalizes the Morse theory to a non-smooth function using the concept of
configuration space (the collection of $n$ points in a bounded area in $R^d$).
Note that their setting is different from us because we are working on a probability density function
whereas the function in \cite{baryshnikov2013min} is related to pairwise distance between points and distance to the boundary of certain area. 
The theory of estimating a cluster tree is closely related to the theory of estimating a level set;
an incomplete list of literature is as follows:
\cite{polonik1995measuring,Tsybakov1997,Walther1997,mason2009asymptotic,singh2009adaptive,Rinaldo2010,steinwart2011adaptive}.



{\em Outline.}
We begin with an introduction of cluster trees and the geometric concepts used in this paper in Section~\ref{sec::background}.
In Section~\ref{sec::nonsingular}, we derive the convergence rate for the $\alpha$-tree estimator under nonsingular measures.
In Section~\ref{sec::singular}, we study the behavior of the KDE
and the stability of the estimated $\alpha$-tree under singular measures.
In Section~\ref{sec::topology}, we investigate critical points of singular measures and derive the topological consistency
of the estimated $\alpha$-tree.
We summarize this paper and discuss possible future directions in Section~\ref{sec::discussion}.
We leave all proofs in the supplementary materials \citep{chen2016supp}.


\section{Backgrounds}	\label{sec::background}

\subsection{Cluster Trees}		\label{sec::cluster_tree}
Here we recall the definition of cluster trees in \cite{chen2016statistical}.
Let $\K\subset \R^d$ and $f:\K\mapsto [0,\infty)$ be a function with support $\K$.
The cluster tree of function $f$ is defined as follows.

\begin{definition}[Definition 1 in \cite{chen2016statistical}]
For any $f: \K \mapsto [0,\infty)$
we define
$T_f: \mathbb{R} \mapsto 2^{2^{\K}}$, where
$2^{\K}$ denotes the set of all subsets of $\K$,
$2^{2^{\K}}$ denotes the collection of all sets of subsets of $\K$,
and $T_f(\lambda)$ is the set of connected components
of the upper level set $\{x \in \K: f(x) \geq \lambda\}$.
We define 
the collection of connected components $\{T_f\}$, 
as $\{T_f\}=\underset{\lambda}{\bigcup}~T_f(\lambda)$. Thus,  $\{T_f\}$
is a collection of subsets of $\K$. 
We called $\{T_f\}$ the cluster tree of $f$.
\label{def::tree}
\end{definition}

Clearly, the cluster tree $\{T_f\}$ has a tree structure,
because for each pair $C_{1},C_{2}\in \{T_f\}$, either 
$C_{1}\subset C_{2}$, $C_{2}\subset C_{1}$, or $C_{1}\cap C_{2}=\phi$ holds.




\begin{figure}
\includegraphics[height=1.5in]{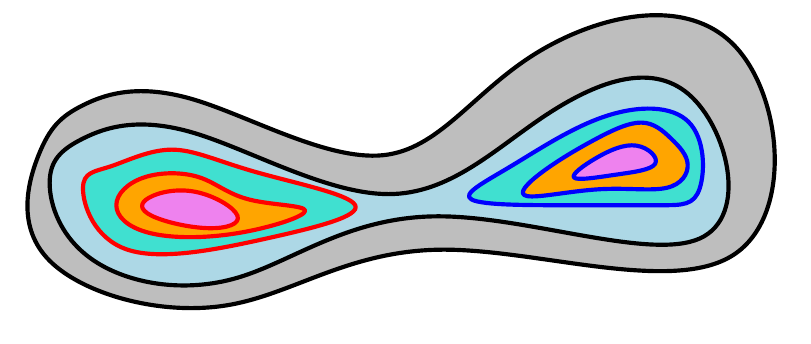}
\includegraphics[height=1.5in]{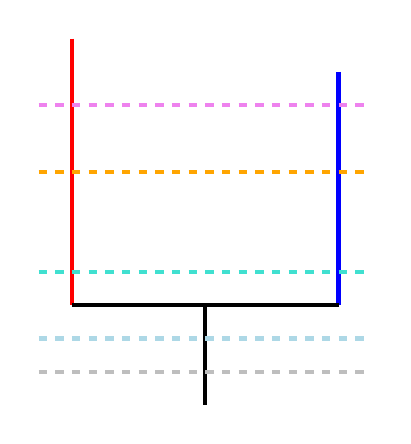}
\caption{
Connected components, edges, and edge set of a cluster tree.
Left: We display connected components of level sets under five different levels 
(indicated by the colors: magenta, yellow, sea green, sky blue, and gray).
The color of boundaries of each connected component denotes the edge they correspond to.
Right: The cluster tree. We color the three edges (vertical lines) red, blue, and black.
The edge set $E(T_f) = \{\mathbb{C}_{\sf red}, \mathbb{C}_{\sf blue}, \mathbb{C}_{\sf black}\}$
and we have the ordering $\mathbb{C}_{\sf red}\leq \mathbb{C}_{\sf black}$ and $\mathbb{C}_{\sf blue}\leq \mathbb{C}_{\sf black}$.
Note that the solid black horizontal line is not an edge set; it is a visual 
representation of connecting the blue and red edges to the black edge.
The horizontal dashed lines indicate the five levels corresponding to the left panel.
In the left panel, the three connected components with red boundaries are elements of the edge $\mathbb{C}_{\sf red}.$
}
\label{fig::equiv}
\end{figure}

To get a geometric understanding of the cluster tree in Definition
\ref{def::tree}, we identify edges that constitute the cluster
tree. Intuitively, edges correspond to either leaves or internal
branches. An edge is roughly defined as a set of clusters whose
inclusion relation with respect to clusters outside an edge is
equivalent. So when the collection of connected components is
divided into edges, we observe the same inclusion relation between
representative clusters whenever any cluster is selected as
representative for each edge.

To formally define edges, we define an interval in the cluster
tree, and the equivalence relation in the cluster tree. For any two
clusters $A,B\in\{T_f\}$, the interval $[A,B]\subset\{T_f\}$ is
defined as a set clusters that contain $A$ and are contained in $B$,
i.e.
\[
[A,B]:=\left\{ C\in\{T_f\}:\,A\subset C\subset B\right\} ,
\]
We define the equivalence relation $\sim$ such that
$A\sim B$ if and only if
\begin{align*}
\forall C \in \{T_f\} \text{ such that }&C\notin[A,B]\cup[B,A],\\
&C\subset A \Leftrightarrow C\subset B, \quad A\subset C \Leftrightarrow B\subset C.
\end{align*}
It is easy to see that the relation $\sim$ is reflexive ($A\sim
A$), symmetric $(A\sim B$ implies $B\sim A$), and transitive ($A\sim B$
and $B\sim C$ implies $A\sim C$). Hence the relation $\sim$ is indeed
an equivalence relation, and we can consider the set of equivalence
classes $\{T_f\}/_{\sim}$. We define the edge set (the collection of edges) $E(T_{f})$ as
$E(T_{f}):=\{T_f\}/_{\sim}$.
Each element in the edge set $\mathbb{C}\in E(T_f)$ is called an edge, 
which contains many nested connected components of the cluster tree $\{T_f\}$ (i.e. if $C_1,C_2\in\mathbb{C}$, then 
either $C_1\subset C_2$ or $C_2\subset C_1$).
Note that every element in an edge corresponds to a connected component of an upper level set of function $f$.

To associate the edge set $E(T_{f})$ to a tree structure,
we define a partial order on the edge set as follows:
let $\mathbb{C}_1,\mathbb{C}_2\in E(T_f)$ be two edges, we write
$\mathbb{C}_1\leq \mathbb{C}_2$ if and only if $A\subset B$ for all $A\in \mathbb{C}_1$ and $B\in \mathbb{C}_2$. 
Then the topology of the cluster tree (the shape of the cluster tree) is completely determined by the
edge set $E(T_f)$ and the partial order among the edge set.
Figure~\ref{fig::equiv} provides an example of the connected components, edges, and edge set of a cluster tree
along with a tree representation.

Based on the above definitions,
we define the topological equivalence between two cluster trees.
\begin{definition}
For two functions $f: \K \mapsto [0,\infty)$ and $g: \K \mapsto [0,\infty)$,
we say $T_f$ and $T_g$ are topological equivalent, denoted as $T_f\overset{T}{\approx} T_g$, if
there exists a bijective mapping $S: E(T_{f})\mapsto E(T_{g})$ such that for any $\mathbb{C}_1, \mathbb{C}_2\in E(T_f)$,
$$
\mathbb{C}_1 \leq \mathbb{C}_2 \Longleftrightarrow S(\mathbb{C}_1) \leq S(\mathbb{C}_2).
$$
\label{def::topology}
\end{definition}

For each $\mathbb{C}\in E(T_f)$, we define
$$
U(\mathbb{C}) = \sup\{\lambda: \exists C\in T_f(\lambda), C\in \mathbb{C}\} 
$$
to be the maximal level of an edge $\mathbb{C}$.
We define the \emph{critical tree-levels} of function $f$ as 
\begin{equation}
\mathcal{A}_f = \{U(\mathbb{C}): \mathbb{C}\in E(T_f)\}.
\label{eq::A_f}
\end{equation}
It is easy to see that $\mathcal{A}_f$ is the collection of levels of $f$ where 
the creation of a new connected component or the merging of two connected components occurs.

In most of the cluster tree literature \citep{balakrishnan2012,chaudhuri2010rates,chaudhuri2014consistent,chen2016statistical,eldridge2015beyond}, 
the cluster tree is referred to as the $\lambda$-tree,
which uses the probability density function $p$ to build a cluster tree.
Namely, the $\lambda$-tree is $T_p$.

In this paper, we focus on the $\alpha$-tree \citep{Kent2013}
that uses the function
\begin{equation}
\alpha(x) = P(\{y:p(y)\leq p(x)\}) =  1- P(\{y: p(y)> p(x)\}) = 1- P(L_{p(x)})
\label{eq::alpha}
\end{equation}
to build the cluster tree $T_\alpha$ ($T_\alpha$ is called the \emph{$\alpha$-tree}).
The function $\alpha(x)$ is also called \emph{density ranking} in \cite{chen2017measuring}. 
The set $L_\lambda = \{x: p(x)\geq \lambda\}$ is the upper level set of $p$ 
(note that $P(\{y: p(y)> p(x)\}) = P(\{y: p(y)\geq p(x)\})$ when the density function $p$ is bounded). 
A feature of the $\alpha$-tree is that the function $\alpha(x)$ depends only on the ordering of points within $\K$.
Namely, any function that assigns the same ordering of points within $\K$ as the density function $p$ 
can be used to construct the function $\alpha(x)$.
Specifically, we write
\begin{align*}
x_1\succ_p x_2 &\Leftrightarrow p(x_1)>p(x_2),\\ 
x_1\prec_p x_2 &\Leftrightarrow p(x_1)<p(x_2),\\ 
x_1\simeq_p x_2 &\Leftrightarrow p(x_1)=p(x_2).
\end{align*}
Then 
\begin{equation}
\alpha(x) = P(\{y: p(y)\leq p(x)\}) = P(\{y: y\preceq_p x\}),
\label{eq::alpha0}
\end{equation}
where $y\preceq_p x$ means either $y\prec_p x$ or $y\simeq_p x$.
Note that if we replace the function $p(x)$ by $2p(x)$ or $\log p(x)$, the ordering remains the same
(i.e., $x\succ_p y\Leftrightarrow x \succ_{2p} y\Leftrightarrow x \succ_{\log p} y$).
We will use this feature later to generalize equation \eqref{eq::alpha} to singular measures. 

One feature of the $\alpha$-tree is that it is topological equivalent to the $\lambda$-tree.
\begin{lem}
Assume the distribution $P$ has a bounded density function $p$ and $p$ is a Morse function
with a compact support.
Then the $\lambda$-tree and $\alpha$-tree are topological equivalent.
Namely,
$$
T_p \overset{T}{\approx} T_\alpha .
$$
\label{lem::tree}
\end{lem}
The proof of this lemma follows from the argument
at the beginning of Section 4.1 of \cite{cadre2009clustering}
so we ignore the proof.
The main idea is that
by equation \eqref{eq::alpha} and the fact that $p$ is Morse, 
$\alpha$ is a strictly monotonic transformation of the density $p$ so
the topology is preserved.

When we use the $\alpha$-tree, the induced upper level set 
$$
\mathbb{A}_\varpi = \left\{x: \alpha(x)\geq \varpi\right\}
$$
is called an \emph{$\alpha$-level set}.


\begin{remark}[$\kappa$-tree]
\cite{Kent2013} also proposed another cluster tree--the $\kappa$-tree--which 
uses the probability content within each edge set defined by an $\alpha$-tree (or a $\lambda$-tree)
to compute the function $\kappa(x)$.
Because it is just a rescaling from the $\alpha$-tree,
the theory of $\alpha$-tree also works for the $\kappa$-tree.
For simplicity, we only study the theory of the $\alpha$-tree in this paper.
\end{remark}

\subsection{Singular Measure}	\label{sec::singular}
When the probability measure is singular, the $\lambda$-tree is no longer well-defined
because there is no density function.
However, 
the $\alpha$-tree
can still be defined.

A key feature for constructing the $\alpha$-tree is the ordering. 
Here we will use a generalized density function, the Hausdorff density \citep{preiss1987geometry,mattila1999geometry}, 
to define the $\alpha$-tree
under singular measures.
Given a probability measure $P$, the $s$-density ($s$ dimensional Hausdorff density) is
$$
\mathcal{H}_s (x) = \lim_{r\rightarrow 0}\frac{P(B(x,r))}{C_s\cdot r^s}
$$
provided the limit exists.
Note that $B(x,r) = \{y: \|y-x\|\leq r\}$ and $C_s$ is the $s$-dimensional volume of an $s$-dimensional unit ball for $s\geq 1$ and $C_0=1$.

For a given point $x$, we define the notion of generalized density using two quantities $\tau(x)$ and $\rho(x)$:
\begin{align*}
\tau(x) & = \max \{s\leq d: \mathcal{H}_s(x)<\infty\}\\
\rho(x) & = \mathcal{H}_{\tau(x)}(x).
\end{align*}
Namely, $\tau(x)$ is the `dimension' of the probability measure at $x$ and $\rho(x)$ 
is the corresponding Hausdorff density at that dimension.
Note that the function $\rho(x)$ is well-defined for every $x$. 
For any two points $x_1,x_2\in\K$, we define an ordering 
such that $x_1\succ_{\tau,\rho} x_2$ if 
$$
\tau(x_1)<\tau(x_2), \qquad \mbox{or}\quad \tau(x_1)=\tau(x_2), \quad \rho(x_1)>\rho(x_2).
$$
That is, for any pair of points,
we first compare their dimensions $\tau(x)$.
The point with the lower dimensional value $\tau$ will be ranked higher than the other point.
If two points have the same dimension, then we compare their
corresponding Hausdorff density.
When the distribution is nonsingular, $\tau(x)=d$ for every $x\in\K$ 
and $\rho(x) = p(x)$ is the usual density function. 
Thus, the ordering is determined by the density function $p(x)$.

To define the $\alpha$-tree, 
we use equation \eqref{eq::alpha0}:
\begin{equation}
\alpha(x) = P\left(\{y: y\preceq_{\tau,\rho} x\}\right). 
\label{eq::alpha1}
\end{equation}
Namely, $\alpha(x)$ is the probability content of regions 
of points whose ordering is lower than or equal to $x$.
As shown in equation \eqref{eq::alpha0}, equation \eqref{eq::alpha1} is the same as equation \eqref{eq::alpha} when $P$ is nonsingular.
Note that by equation \eqref{eq::alpha1}, 
the $\alpha$-level set $\mathbb{A}_\varpi = \left\{x: \alpha(x)\geq \varpi\right\}$ is 
well-defined in a singular measure.


\subsection{Geometric Concepts} \label{sec::geometrics}

We first define some notations for sets.
For a set $A$, define $\overline{A}$ to be the closure of $A$, $\mathring{A}$ to be the
interior of $A$,
$\partial A $ to be the boundary of set $A$,
and $A^C= \K\backslash A$ to be the complement of set $A$ restricted to the support $\K$.
When $A$ is a manifold, $\partial A $ and $\mathring{A}$ will be the boundary and interior of the manifold respectively. 
We define $A\triangle B = (A\backslash B)\cup (B\backslash A)$ to be the symmetric difference for sets $A$ and $B$


Based on the definition of $\tau(x)$, we decompose the support $\K$ into
\begin{equation}
\K = \K_d\bigcup \K_{d-1}\bigcup \cdots \bigcup \K_0,
\label{eq::support}
\end{equation}
where $\K_s = \{x: \tau(x)=s\}$.
Thus, $\{\K_0,\cdots,\K_d\}$ forms a partition of the entire support $\K$.
We call each $\K_s$ an $s$-dimensional support.
When we analyze the support $\K_s$,
any $\K_{s'}$ with $s'>s$ is called a higher dimensional support (with respect to $\K_s$) and
$s'<s$ will be called a lower dimensional support.

\begin{figure}
\center
\includegraphics[height=2in]{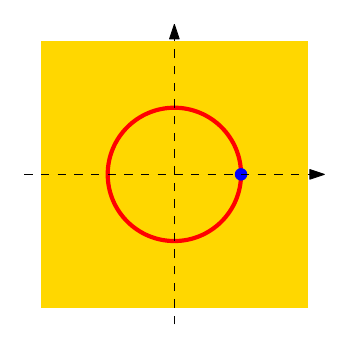}
\caption{
The supports $\K_2,\K_1,$ and $\K_0$ for the bivariate random variable in Example \ref{ex::ex02}.
The yellow region is $\K_2 = [-1,1]^2 \backslash  R_{0.5}$, where $R_{0.5} = \{(x,y): x^2+y^2 = 0.5^2\}$.
The red ring is $\K_1 = R_{0.5}\backslash \{(0.5,0)\}$.
The blue dot is the location of $(0.5,0)$.
}
\label{fig::ex02}
\end{figure}

\begin{example}
Let $X$ be a bivariate random variable such that (i) with a probability of $0.7$, it is from a uniform distribution on $[-1,1]^2$,
(ii) with a probability of $0.25$, it is from a uniform distribution on the ring $R_{0.5} = \{(x,y): x^2+y^2 = 0.5^2\}$,
(iii) with a probability of $0.05$, it is equal to $(0.5,0)$. 
Apparently, the distribution of $X$ is singular and the support $\K = [-1,1]^2$.
In this case, $\K_2 = [-1,1]^2 \backslash  R_{0.5}$, $\K_1 = R_{0.5}\backslash \{(0.5,0)\}$, and $\K_0 = \{(0.5,0)\}$.
Figure~\ref{fig::ex02} shows these supports under different colors.
The yellow rectangular region is $\K_2$, the red ring area is $\K_1$, and the blue dot denotes the location of $\K_0$.
\label{ex::ex02}
\end{example}


To regularize the behavior of $\rho(x)$ on each support $\K_s$, 
we assume that the closure of the support, 
$\overline{\K_s}$,
is an $s$-dimensional smooth manifold 
(properties of a smooth manifold can be found in 
\citealt{lee2012introduction,tu2010introduction}).

For an $s$-dimensional smooth manifold $\mathcal{M}$, 
the tangent space on each point of $\mathcal{M}$ changes smoothly \citep{tu2010introduction, lee2012introduction}. 
Namely, for $x\in\mathcal{M}$, 
we can find an orthonormal basis $\{v_1(x),\cdots, v_s(x): v_\ell(x)\in\R^d, \ell=1,\cdots, s\}$ such that
$\{v_1(x),\cdots, v_s(x)\}$ spans
the tangent space of $\mathcal{M}$ at $x$ 
and each $v_\ell(x)$ is a smooth (multivalued) function on $\mathcal{M}_s$.
For simplicity, for $x\in \K_s$, we denote $T_s(x)$ as the tangent space of $\K_s$ at $x$,
$N_s(x)$ as the normal space of $\K_s$ at $x$,
and $\nabla_{T_s(x)}$ to be taking the derivative in the tangent space.

For a function $f:\mathcal{M}\mapsto\R$ defined on a smooth manifold $\mathcal{M}$,
the function $f$ is a \emph{Morse function} \citep{milnor1963morse,morse1925relations,morse1930foundations} 
if all its critical points are non-degenerate. 
Namely, the eigenvalues of the Hessian matrix of $f$ at each critical point are 
non-zero.
When the function $f$ is a Morse function, its $\lambda$-tree is stable \citep{chazal2014robust,chen2016statistical}
in the $L_\infty$ metric under a small perturbation of $f$.

To link the concept of the Morse function to the Hausdorff density $\rho(x)$, 
we introduce a generalized density
$$
\rho^\dagger_s:\overline{\K_s} \mapsto [0,\infty)
$$
such that $\rho^\dagger_s(x) = \lim_{x_n\in \K_s: x_n\rightarrow x} \rho(x_n)$
provided the limit exists and does not depend on the choice of the sequence $x_n$.
It is easy to see that $\rho^\dagger_s(x) = \rho(x)$ when $x\in\K_s$ 
but now it is defined on a smooth manifold $\overline{\K_s}$.
We say $\rho(x)$ is a \emph{generalized Morse function} if 
the corresponding $\rho^\dagger_s(x)$ is a Morse function for $s=1,\cdots, d$.
Later we will show that this generalization leads to a stable $\alpha$-tree for a singular measure.

For $\overline{\K_s}$,
let $\mathcal{C}_s =\{x\in\overline{\K_s}: \nabla_{T_s(x)}\rho^\dagger_s(x)=0\}$ be the collection of its critical points.
Then the fact that $\rho^\dagger_s(x)$ is a Morse function 
implies that the eigenvalues of the Hessian matrix $\nabla_{T_s(c)}\nabla_{T_s(c)}\rho^\dagger_s(c)$ are non-zero
for every $c\in\mathcal{C}_s$.
We call $g_s(x)=\nabla_{T_s(x)}\rho^\dagger_s(x)$ the generalized gradient and
$H_s(x)=\nabla_{T_s(x)}\nabla_{T_s(x)}\rho^\dagger_s(x)$ the generalized Hessian.
For the case $s=0$ (point mass), we define $ \mathcal{C}_0=\K_0 $.
The collection
$\mathcal{C} = \bigcup_{s=1,\cdots,d} \mathcal{C}_s$ 
is called the collection of\emph{ generalized critical points} of $\rho(x)$.
Each element $c\in\mathcal{C}$ is called a generalized critical point.


Finally, we introduce the concept of \emph{reach} \citep{Federer1959,chen2015density} for a smooth manifold $\mathcal{M}$.
The reach of $\mathcal{M}$ is defined as 
$$
{\sf reach}(\mathcal{M}) = \sup\{r\geq0: \mbox{every point in }\mathcal{M}\oplus r\mbox{ has a unique projection onto $\mathcal{M}$}\},
$$
where $A\oplus r = \{x: d(x,A)\leq r\}$.
One can view reach as the radius of the largest ball that can roll freely outside $\mathcal{M}$. 
More details about reach can be found in \cite{Federer1959,chen2015density}.
Reach plays a key role in the stability of a level set; see \cite{chen2015density} for more details.

\begin{remark}
\label{rm::ss}
If we further assume that the supports satisfy
\begin{equation}
\K\supset\overline{\K}_d\supset \overline{\K}_{d-1}\supset \cdots \supset \overline{\K}_0,
\label{eq::lowD}
\end{equation}
then the support $\K$ forms a stratified space
\citep{goresky1980intersection,goresky2012stratified}.
Roughly speaking, a stratified space is a topological space $\mathbb{W}$ such that
there exists a decomposition (called stratification) $\mathbb{W}_{0},\cdots, \mathbb{W}_d$ of $\mathbb{W}$
with the properties that (i) each $\mathbb{W}_k$ is a $k$-dimensional smooth manifold,
(ii) $\mathbb{W} = \bigcup_{\omega=0}^d \mathbb{W}_{\omega}$, and (iii) for any $k\leq \ell$,
$$
\mathbb{W}_k\cap \overline{\mathbb{W}}_\ell \neq \emptyset \Leftrightarrow \mathbb{W}_k\subset \overline{\mathbb{W}}_\ell.
$$
From the properties of a stratified space, one can see how equation \eqref{eq::lowD} is related to a stratified space. 
Note that for a stratified space, if we consider a probability measure that is a mixture of probability measures 
defined on each stratum ($\mathbb{W}_k$), this defines a singular measure as the one being considered in this paper. 
The topology of a stratified space can be defined using
the intersection homology \citep{goresky1980intersection,edelsbrunner2008persistent,friedman2014singular}. 
The notion of intersection homology and stratified space will
be particularly useful if we want to work on higher-order homology groups.
\end{remark}

\subsection{Estimating the $\alpha$ function and the $\alpha$-tree}
In this paper, we focus on estimating $\alpha$-trees via
the KDE:
$$
\hat{p}_n(x) = \frac{1}{nh^d}\sum_{i=1}^nK\left(\frac{\|x-X_i\|}{h}\right).
$$
Specifically, we first estimate the density by $\hat{p}_n$ and then
construct the estimator $\hat{\alpha}_n$:
\begin{equation}
\hat{\alpha}_n(x) = \hat{P}_n\left(\{y: \hat{p}_n(y)\leq \hat{p}_n(x)\}\right) 
\label{eq::alpha_estimate}
\end{equation}
where $\hat{P}_n$ is the empirical measure and $\hat{L}_\lambda = \{x: \hat{p}_n(x)\geq \lambda\}$.
Note that when $x$ does not contain any point mass of $P$, 
$\hat{\alpha}_n(x)= 1-\hat{P}_n\left(\hat{L}_{\hat{p}_n(x)}\right).$

To quantify the uncertainty in the estimator $\hat{\alpha}_n$, we consider three error measurements.
The first error measurement is 
the \emph{$L_\infty$ error}, which is defined as
$$
\|\hat{\alpha}_n-\alpha\|_\infty=\sup_x|\hat{\alpha}_n(x)-\alpha(x)|.
$$
The $L_\infty$ error has been used in several cluster tree literature; see, e.g., \cite{eldridge2015beyond,chen2016statistical}.
An appealing feature of the $L_\infty$ error is that this quantity
is the same (up to some constant) as some other tree error metrics such as the merge distortion metric \citep{eldridge2015beyond}.
Convergence in the merge distortion metric implies the Hartigan consistency \citep{eldridge2015beyond}, 
a notion of consistency of a cluster tree estimator described in \cite{Hartigan81,chaudhuri2010rates,chaudhuri2014consistent}.
Thus, because of the equivalence between the $L_\infty$ error and the merge distortion metric, convergence in $L_\infty$ implies 
the Hartigan consistency of an estimated cluster tree.

The other two errors are the integrated error and the probability error (probability-weighted integrated error). 
Both are common error measurements for evaluating the quality of a function estimator \citep{wasserman2006all,scott2015multivariate}.
The \emph{integrated error} is 
$$
\|\hat{\alpha}_n-\alpha\|_{\mu} = \int |\hat{\alpha}_n(x)-\alpha(x)| dx,
$$
which is also known as the integrated distance or $L_1$ distance. 
The \emph{probability error (probability-weighted integrated error)} is
$$
\|\hat{\alpha}_n-\alpha\|_{P}=\int |\hat{\alpha}_n(x)-\alpha(x)| dP(x),
$$
which
is the integrated distance weighted by the probability measure, which is also known
as $L_1(P)$ distance.
The integrated error and the probability error 
are more robust than the $L_\infty$ error--a large difference in a small region will not 
affect on these errors much.

To quantify the uncertainty in the topology of $\alpha$-tree,
we introduce the notion of \emph{topological error}, which is defined as
$$
P\left(T_{\hat{\alpha}_n}\not\overset{T}{\approx} T_\alpha\right) = 1- P\left(T_{\hat{\alpha}_n}\overset{T}{\approx} T_\alpha\right). 
$$
Namely, the topological error is the probability that the estimated $\alpha$-tree
is not topological equivalent to the population $\alpha$-tree.


Finally, we define the following notations.
For a smooth function $p$, we define $\|p\|_{\ell,\infty}$ as the supremum maximal norm of $\ell$-th derivative of $p$.
For instance,
$$
\|p\|_{0,\infty}=\sup_{x\in\K}p(x),\quad\|p\|_{1,\infty}=\sup_{x\in\K}\|g(x)\|_{\max},\quad\|p\|_{2,\infty}=\sup_{x\in\K}\|H(x)\|_{\max},
$$
where $g(x) = \nabla p(x)$ and $H(x) = \nabla \nabla p(x)$ are the gradient and Hessian matrix of $p(x)$, respectively.
A vector $\beta =
(\beta_1,\ldots,\beta_d)$ of non-negative integers is called a
multi-index with $|\beta| = \beta_1 + \beta_2 + \cdots + \beta_d$
and the corresponding derivative operator is
\begin{equation}
D^\beta = \frac{\partial^{\beta_1}}{\partial x_1^{\beta_1}} \cdots 
\frac{\partial^{\beta_d}}{\partial x_d^{\beta_d}},
\label{eq::d1}
\end{equation}
where $D^\beta f$ is often written as $f^{(\beta)}$.

\section{Theory for Nonsingular Measures}	\label{sec::nonsingular}

To study the theory for nonsingular measures, we 
make the following assumptions.\\
{\bf Assumptions.}
\begin{itemize}
\item[(P1)] $p$ has a compact support $\K$ and is a Morse function and is four times differentiable with $\|p\|_{\ell,\infty}<\infty$ for $\ell=0,\cdots, 4.$
\item[(K1)] $K(x)$ has compact support and is non-increasing on $[0,1]$, and has at least fourth-order bounded derivative
and
$$
\int \|x\|^2 K^{(\beta)}(\|x\|) dx <\infty, \qquad \int  \left(K^{(\beta)}(\|x\|)\right)^2 dx <\infty
$$
for $|\beta| \leq 2$ and $K^{(2)}(0) <0$.

\item[(K2)] Let 
\begin{align*}
\mathcal{K}_r &= \left\{y\mapsto K^{(\beta)}\left(\frac{\|x-y\|}{h}\right): x\in\mathbb{R}^d, |\beta|=r, \bar{h}>h>0\right\},
\end{align*}
where $K^{(\beta)}$ is defined in equation \eqref{eq::d1}
and $\mathcal{K}^*_l = \bigcup_{r=0}^l \mathcal{K}_r$
and $\bar{h}$ is some positive number. 
We assume that $\mathcal{K}^*_2$ is a VC-type class. i.e. 
there exists constants $A,v$ and a constant envelope $b_0$ such that
\begin{equation}
\sup_{Q} N(\mathcal{K}^*_2, \cL^2(Q), b_0\epsilon)\leq \left(\frac{A}{\epsilon}\right)^v,
\label{eq::VC}
\end{equation}
where $N(T,d_T,\epsilon)$ is the $\epsilon$-covering number for a
semi-metric set $T$ with metric $d_T$ and $\cL^2(Q)$ is the $L_2$ norm
with respect to the probability measure $Q$.

\end{itemize}

Assumption (P1) is a common condition to guarantee that critical points are well-separated
and will not move too far away
under a small perturbation on the gradient and Hessian of the density function \citep{chazal2014robust,chen2016statistical}.
We need the fourth-order derivative to ensure the estimated density Hessian matrix
converges to the population density Hessian matrix 
(the bias in estimating the Hessian matrix depends on fourth-order derivatives).
Assumption (K1) is a standard condition on kernel function \citep{wasserman2006all,scott2015multivariate}.
Assumption (K2) regularizes the complexity of kernel functions so  
we have uniform bounds on density, gradient, and Hessian estimation.
It was first proposed by \cite{Gine2002} and \cite{Einmahl2005}
and was later used in various studies such as \cite{genovese2009path,genovese2014nonparametric,chen2015asymptotic}.

We first study the error rates under nonsingular measures.
In the case of the $\lambda$-tree, error rates are well-studied, and 
we summarize them in the following theorem.

\begin{thm}
Assume (P1, K1--2).
Then when $h\rightarrow 0,\frac{nh^{d+4}}{\log n} \rightarrow\infty$,
\begin{align*}
\|\hat{p}_n-p\|_\infty & =  O(h^2) + O_P\left(\sqrt{\frac{\log n}{nh^d}}\right)\\
\|\hat{p}_n-p\|_\mu & = O(h^2) + O_P\left(\sqrt{\frac{1}{nh^d}}\right)\\
\|\hat{p}_n-p\|_P & = O(h^2) + O_P\left(\sqrt{\frac{1}{nh^d}}\right)\\
P\left(T_{\hat{p}_n}\overset{T}{\approx} T_{p}\right)& \geq 1-c_0\cdot e^{-c_1\cdot nh^{d+4}},
\end{align*}
for some $c_0, c_1>0$.
\label{thm::lambda::non}
\end{thm}
The rate of consistency under the $L_\infty$ error can be found in \cite{chen2015density,Gine2002,Einmahl2005}.
The integrated error and probability error can be seen in \cite{scott2015multivariate}.
And the topological error bound follows Lemma 2 in \cite{chen2016statistical} and the concentration of $L_\infty$ metric
for the estimated Hessian matrix.

The requirement of $h$ in Theorem~\ref{thm::lambda::non} enforces
the uniform convergence of the KDE as well as its first and second derivative. 
Uniform convergence of derivatives of the KDE implies the convergence of some geometric structures
of the density function, such as the ridges \citep{chen2015asymptotic,genovese2014nonparametric}, 
critical points \citep{chazal2014robust,chen2016comprehensive}, 
and persistent diagrams \citep{cohen2007stability,fasy2014confidence,chen2017tutorial}.

Now we turn to the consistency for $\alpha$-tree.
To derive the rate for the $\alpha$-tree,
we need to study the convergence rate of an estimated level set when the level is 
the density value of a critical point (also known as a critical level). 
The reason is that the quantity $\alpha(x) = 1-P(L_{p(x)})$ 
is the probability content of upper level set $L_{p(x)} = \{y: p(y)\geq p(x)\}$. 
When $p(x)=p(c)$ for some critical point $c$ of $p$, we face the problem of analyzing the stability of 
level sets at a critical level.

\begin{thm}[Level set error at a critical value]
Assume (P1) and (K1--2) and $d\geq 2$.
Let $\lambda$ be a density level corresponding to the density of a critical point.
When $h\rightarrow 0,\frac{\log n}{nh^{d+4}}\rightarrow 0$,
$$
\mu\left(\hat{L}_\lambda\triangle L_{\lambda}\right) = O_P\left(\|\hat{p}_n - p\|_{\infty}^{\frac{d}{d+1}}\right).
$$
\label{thm::critical}
\end{thm}
The rate in Theorem~\ref{thm::critical} is slower than the usual density estimation rate.
This is because the boundary of $L_\lambda$ hits a critical point when $\lambda$ equals the density of a critical point. 
The regions around a critical point have a very low gradient,
which leads to a slower convergence rate. 
It is well-known \citep{Gine2002,Einmahl2005,genovese2014nonparametric} that under assumption (P) and (K1--2), 
$$
\|\hat{p}_n - p\|_{\infty} = O(h^2) + O_P\left(\sqrt{\frac{\log n}{nh^d}}\right).
$$

In Theorem~\ref{thm::critical}, we see that when $d$ is large, the quantity $\frac{d}{d+1}\rightarrow 1$
so the error rate is similar to $\|\hat{p}_n - p\|_{\infty}$.
This is because the regions that slow down the error rate are areas around the critical points.
These areas occupy a small volume when $d$ is large, which decreases the difference in the rate.


\begin{remark}
Theorem~\ref{thm::critical} complements many existing level set estimation theories. 
To our knowledge, no literature has worked on the situation where $\lambda$ equals the density of a critical point.
Level set theories mostly focus on one of the following three cases: 
(i) the gradient on the level set boundary $\partial L_{\lambda} = \{x: p(x)=\lambda\}$
is bounded away from $0$ \citep{Molchanov1990,Tsybakov1997,Walther1997,Cadre2006,Laloe2012,mammen2013confidence,chen2015density},
(ii) a lower bound on the density changing rate around level $\lambda$ \citep{singh2009adaptive, rinaldo2012stability}, 
(iii) an $(\epsilon,\sigma)$ condition for density \citep{chaudhuri2010rates,chaudhuri2014consistent}. 
When $\lambda$ equals a critical level, none of these assumptions hold. 

\end{remark}

Based on Theorem~\ref{thm::critical}, we derive the convergence rate of $\hat{\alpha}_n$.
\begin{thm}
Assume (P1) and (K1--2) and $d\geq 2$ and the smoothing bandwidth satisfies $h\rightarrow0,\frac{\log n}{nh^{d+4}}\rightarrow 0$.
Let $\mathcal{C} = \{x: \nabla p(x)=0\}$ be the collection of critical points
and let $a_n$ be a sequence of $n$ such that 
$\|\hat{p}_n-p\|_\infty = o(a_n)$. 
Then uniformly for all $x$,
$$
\hat{\alpha}_n(x) - \alpha(x) = 
\begin{cases}
O_P\left(\|\hat{p}_n-p\|_\infty\right) \quad&\mbox{, if}\quad |p(x)-p(c)|>a_n\mbox{ for all $c\in\mathcal{C}$,}\\
O_P\left(\|\hat{p}_n-p\|_\infty^{\frac{d}{d+1}}\right) \quad&\mbox{, otherwise}.
\end{cases}
$$
\label{thm::alpha_non}
\end{thm}
Theorem~\ref{thm::alpha_non} shows uniform error rates for $\hat{\alpha}_n$.
When the density of a given point is away from critical levels, 
the rate follows the usual density estimation rate.
When the given point has a density value close to some critical points,
the rate is slowed down by the low gradient areas around critical points. 
Note that the sequence $a_n$ is to make the bound uniform for all $x$.
To obtain an integrated error rate (and the probability error rate) of $\hat{\alpha}_n$, 
we can choose $a_n = \frac{1}{\log n}\left(O(h^2) + O_P\left(\sqrt{\frac{1}{nh^d}}\right)\right)$
which leads to the following result.


\begin{cor}
Assume (P1) and (K1--2) and $d\geq 2$.
Then when $h\rightarrow 0,\frac{nh^{d+4}}{\log n} \rightarrow\infty$,
\begin{align*}
\|\hat{\alpha}_n-\alpha\|_\infty & =  O\left(h^{\frac{2d}{d+1}}\right)+O_P\left(\left(\frac{\log n}{nh^d}\right)^{\frac{d}{2(d+1)}}\right)\\
\|\hat{\alpha}_n-\alpha\|_\mu & = O(h^2) + O_P\left(\sqrt{\frac{\log n}{nh^d}}\right)\\
\|\hat{\alpha}_n-\alpha\|_P & = O(h^2) + O_P\left(\sqrt{\frac{\log n}{nh^d}}\right)\\
P\left(T_{\hat{\alpha}_n}\overset{T}{\approx} T_\alpha\right)& \geq 1-c_0\cdot e^{-c_1\cdot nh^{d+4}},
\end{align*}
for some $c_0,c_1>0$.
\label{cor::alpha_non}
\end{cor}
Comparing Corollary \ref{cor::alpha_non} to Theorem~\ref{thm::lambda::non},
we see that
only the $L_\infty$ error rate has a major difference
and the other two errors differ by a $\sqrt{\log n}$ factor.
This is because Theorem~\ref{thm::alpha_non} proves that, only at the level of a critical point, we will have 
a slower convergence rate. 
Thus, the $L_\infty$ error will be slowed down by these points.
However, the collection of points $\{x: p(x)=p(c)\mbox{ for some }c\in\mathcal{C}\}$ has Lebesgue measure zero
so the slow convergence rate does not translate to the integrated error and the probability error.
The topological error follows from Theorem~\ref{thm::lambda::non} and Lemma~\ref{lem::tree}:
$
T(\hat{p}_n)\overset{T}{\approx} T_{\hat{\alpha}_n},
T(p)\overset{T}{\approx} T_\alpha.
$


\section{Singular Measures: Error Rates}	\label{sec::singular}

Now we study error rates under singular measures.
When a measure is singular, the usual (Radon-Nikodym) density cannot be defined. 
Thus, we cannot define the $\lambda$-tree.
However, as we discussed in Section~\ref{sec::singular},
we are still able to define the $\alpha$-tree.
Thus, in this section, we will focus on error rates for the $\alpha$-tree.

\subsection{Analysis of the KDE under Singular Measures}

To study the convergence rate, we first investigate the bias of smoothing in the singular measure.
Let $p_h(x) = \mathbb{E}(\hat{p}_n)$, which is also known as the smoothed density.


{\bf Assumption.}
\begin{itemize}
\item[(S)] For all $s<d$, ${\overline{\K_s}}$ 
is a smooth manifold with positive reach and $\K$ is a compact set.
\item[(P2)] $\rho(x)$ is a generalized Morse function and there exists some $\rho_{\min},\rho_{\max}>0$ such that 
$0<\rho_{\min}\leq \rho(x)\leq\rho_{\max}<\infty$ for all $x$.
Moreover, for any $s>0$, $\rho^\dagger_s$ is unique and has bounded continuous derivatives up to the fourth order.
\end{itemize}

Assumption (S) ensures that
$\K_s$ is smooth and every connected component
of $\K_s$ is separated for each $s$.
Assumption (P2) is a generalization of (P1) to singular distributions.

\begin{lem}[Bias of the smoothed density]
Assume (S, P2).
Let $x\in\mathring{\K}_s$ and define
$m(x) = \min\{\ell> s: x\in\overline{\K_\ell}\}-s $.
Let $C^\dagger_\ell = (\int_{B_\ell}K(\|y\|)dy)^{-1}$, where $B_\ell = \{y: \|y\|\leq1,y_{\ell+1}=y_{\ell+2}=\cdots=y_d=0\}$
for $\ell=1,\cdots, d$ and $dy$ is integrating with respect to $\ell$-dimensional area and $C^\dagger_0 = 1/K(0)$.
Then 
for a fixed $x$, when $h\rightarrow 0$ and $m(x)>0$,
$$
C^\dagger_{\tau(x)}h^{d-\tau(x)}\cdot p_h(x) = \rho(x)+ 
\begin{cases}
O(h^2) + O(h^{m(x)}),\quad &{ if }\quad m(x)>0\\
O(h^2),\quad &{ if }\quad m(x)= 0\\
\end{cases}.
$$
Moreover, if $\overline{\K_\ell}\cap \K_s\neq \phi$, for some $s<\ell$,
then there exists $\epsilon>0$ such that 
$$
\lim_{h\rightarrow 0}\sup_{x\in\K}|C_{\tau(x)}h^{d-\tau(x)}\cdot  p_h(x) - \rho(x)|>\epsilon>0.
$$
\label{lem::smooth::bias}
\end{lem}
Lemma~\ref{lem::smooth::bias} describes the bias of the KDE.
The scaling factor $C^\dagger_{\tau(x)}h^{d-\tau(x)}$ rescales the smoothed density 
to make it comparable to the generalized density.
The first assertion is a pointwise convergence of smoothed density.
In the case of $m(x)>0$, the bias contains two components: $O(h^2)$, the usual smoothing bias,
and $O(h^{m(x)})$, the bias from a higher dimensional support.
This is because the KDE is isotropic, so the probability content outside $\K_s$ will also be 
included, which causes additional bias.
The second assertion states that the smoothed density does not uniformly converge to the generalized density $\rho(x)$,
so together with the first assertion, we conclude that the smoothing bias converges pointwisely but not 
uniformly.
Next, we provide an example showing the failure of uniform convergence of a singular measure.

\begin{example}[Failure of uniform convergence]
\label{ex::point}
We consider $X$ from the same distribution as in Figure~\ref{fig::Tree}:
with a probability of $0.3$, $X=2$, and with a probability of $0.7$, $X$ follows a standard normal.
For simplicity, we assume that the kernel function is the spherical kernel $K(x) = \frac{1}{2}I(0\leq x\leq 1)$
and consider the smoothing bandwidth $h\rightarrow 0$.
This choice of kernel yields $C^\dagger_1 = 1$.
Now consider a sequence of points $x_h = 2+\frac{h}{2}$. 
Then the smoothed density at each $x_h$ is
\begin{align*}
p_h(x_h)
&= \frac{1}{h}P(x_h-h<X<x_h+h) \\
&= \frac{1}{h}P\left(2-\frac{h}{2}<X<2+\frac{3h}{2}\right) \\
&\geq \frac{1}{h}P\left(X=2\right)  = \frac{3}{10h},
\end{align*}
which diverges when $h\rightarrow 0$.
However,
it is easy to see that $\tau(x_h)=1$ and $\rho(x_h) =  \frac{7}{10}\phi(x_h)\rightarrow \frac{7}{10}\phi(1)$ which is a finite number.
Thus, $|\mathbb{E}(p_h(x_h)-\rho(x_h)|$ does not converge.
\end{example}

\begin{remark}
The scaling factor in Lemma~\ref{lem::smooth::bias} $C^\dagger_{\tau(x)}h^{d-\tau(x)}$
depends on the support $\K_s$ where $x$ resides.
In practice, we do not know $\tau(x)$ so we cannot properly rescale $\hat{p}_n(x)$ to estimate $\rho(x)$. 
However, we are still able to rank pairs of data points based on Lemma~\ref{lem::smooth::bias}.
To see this,
let $x_1$ and $x_2$ be the two points that we want to compare their orderings
(i.e., we want to know $x_1\prec_{\tau,\rho} x_2$ or $x_1\succ_{\tau,\rho} x_2$ or $x_1\simeq_{\tau,\rho} x_2$). 
When $x_1$ and $x_2$ are both in $\K_s$ for some $s$,
the scaling does not affect the ranking between them
so the sign of $\rho(x_1)-\rho(x_2) $ is the same as that of $p_h(x_1)-p_h(x_2)$.
When $x_1$ and $x_2$ are in different supports (i.e. $x_1\in\K_{s_1}, x_2\in\K_{s_2}$, where $s_1\neq s_2$),
$p_h(x_1)$ and $p_h(x_2)$ diverge at different rates, meaning that we can eventually distinguish them.
Thus, ordering for most points can still be recovered under singular measures.
This is an important property that leads to the consistency of $\hat{\alpha}_n$ under other error measurements.
\end{remark}


Due to the failure of uniform convergence in the bias, the $L_\infty$ error of $\hat{\alpha}_n$
does not converge under singular measures.
\begin{cor}[$L_\infty$ error for singular measures]
Assume (S, P2).
When $\overline{\K_d}\cap \overline{\K_s}\neq \phi$, for some $s<d$,
$\|\hat{\alpha}_n-\alpha\|_\infty$ does not converge to $0$.
Namely, there exists $\epsilon>0$ such that 
$$
\liminf_{n,h} P(\|\hat{\alpha}_n-\alpha\|_\infty>\epsilon)>0.
$$
\label{cor::alpha_infty_s}
\end{cor}
The proof of Corollary~\ref{cor::alpha_infty_s} is a direct application of the failure of uniform convergence in
smoothing bias shown in Lemma~\ref{lem::smooth::bias}.
This corollary shows that for a singular measure, 
the $L_\infty$ error of the estimator $\hat{\alpha}_n$ does not converge in general.
Thus, there is no guarantee for the Hartigan consistency of the estimated $\alpha$-tree.


\subsection{Error Measurements}

Although Corollary~\ref{cor::alpha_infty_s} presents a negative result on estimating the $\alpha$-tree,
in this section, we show that the estimator $\hat{\alpha}_n$ is still consistent under other error measurements.
A key observation is that
there is a good region where the scaled KDE converges uniformly.

Let $\K_s(h) = \mathring{\K}_s \backslash(\bigcup_{\ell<s} \K_\ell\oplus h)$ be the set 
that is in the interior of $\K_s$ and is away from lower dimensional 
supports for $s>0$; in the case of $s=0,$ we define $\K_0(h) = \K_0$.
We define $\K(h) = \bigcup_{s\leq d} \K_s(h)$, which is the union of each $\K_s(h)$.
Figure~\ref{fig::ex02-1} shows the good region of support $\K_s(h)$ and the original support $\K_s$ in Example~\ref{ex::ex02}.
Later we will show that the set $\K(h)$ is the good region.

\begin{figure}
\center
\includegraphics[height=1.5in]{figures/ex02-1}
\includegraphics[height=1.5in]{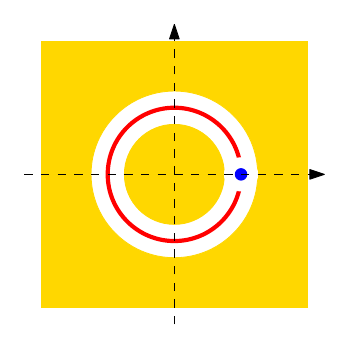}
\caption{
Good regions $\K_s(h)$ for Example \ref{ex::ex02}.
Left: the original $\K_2, \K_1$, and $\K_0$. 
Right: the corresponding $\K_2(h), \K_1(h)$, and $\K_0(h)$.
The yellow area is the set $\K_2(h)$, which are regions in $\K_2$ away from
lower-dimensional supports $\K_1$ and $\K_0$.
The red area is the set $\K_1(h)$, which are regions in $\K_1$ where regions close to $\K_0$ have been removed.
The blue dot is $\K_0(h)$, which is the same as $\K_0$ because there is no lower dimensional support.
Note that $\K(h)$ is the union of the three color regions in the right panel.
}
\label{fig::ex02-1}
\end{figure}

In Lemma~\ref{lem::smooth::bias}, the quantity 
$$
m(x) = \min\{\ell\geq \tau(x): x\in\overline{\K_\ell}\}-\tau(x)  
$$
plays a key role
in determining the rate of smoothing bias. 
Only when $m(x)=1$, do we have a slower rate for the bias. 
Thus, to obtain a uniform rate on the bias, we introduce the quantity 
\begin{equation}
m_{\min} = \inf_{x\in \K, m(x)>0} m(x).
\label{eq::mx}
\end{equation}
If $m(x)=0$ for all $x\in \K$, we define $m_{\min}=2$.

\begin{thm}[Consistency of the KDE under singular measures]
Assume (S, P2, K1--2).
Let $C^\dagger_\ell$ be the constant in Lemma~\ref{lem::smooth::bias}.
Define
\begin{align*}
\delta_{n,h,s}&=\sup_{x\in\K_s(h)}\|C^{\dagger}_{s}h^{d-s}\hat{p}_n(x)-\rho(x)\| \\
\delta^{(1)}_{n,h,s}&=\sup_{x\in\K_s(h)}\|C^{\dagger}_{s}h^{d-s} \nabla_{T_s(x)}\hat{p}_n(x)-\nabla_{T_s(x)}\rho(x)\|_{\max}\\
\delta^{(2)}_{n,h,s}&=\sup_{x\in\K_s(h)}\|C^{\dagger}_{s}h^{d-s} \nabla_{T_s(x)}\nabla_{T_s(x)}\hat{p}_n(x)-\nabla_{T_s(x)}\nabla_{T_s(x)}\rho(x)\|_{\max},
\end{align*}
where $\nabla_{T_s(x)}$ is taking gradient with respect to the tangent space of $\K_s$ at $x$.
Then
when $h\rightarrow 0, \frac{nh^{d+4}}{\log n}\rightarrow \infty$,
\begin{equation}
\begin{aligned}
\delta_{n,h,s} &= O(h^{2\bigwedge m_{\min}}) + O_P\left(\sqrt{\frac{\log n}{nh^s}}\right),\\
\delta^{(1)}_{n,h,s} &= O(h^{2\bigwedge m_{\min}}) + O_P\left(\sqrt{\frac{\log n}{nh^{s+2}}}\right),\\
\delta^{(2)}_{n,h,s} &= O(h^{2\bigwedge m_{\min}}) + O_P\left(\sqrt{\frac{\log n}{nh^{s+4}}}\right),
\end{aligned}
\label{eq::rate}
\end{equation}
where $a\bigwedge b = \min\{a,b\}$.
\label{thm::KDE}
\end{thm}
Theorem~\ref{thm::KDE} shows that after rescaling, the KDE is uniformly consistent within the good region $\K_s(h)$
for density, gradient, and Hessian estimation.
A more interesting result is that, after rescaling, the error rate is the same as the usual $L_\infty$ error rate in the $s$-dimensional case
with a modified bias term 
(bias is affected by a higher dimensional support). 

\begin{remark}(Non-convergence of the integrated distance of the KDE)
\label{rm::IE_KDE}
One may wonder if 
the scaled KDE ($C^\dagger_{\tau(x)}h^{d-\tau(x)}\cdot \hat{p}_n(x)$) converges to the generalized density $\rho(x)$
under the integrated distance.
There is no guarantee for such a convergence because
$$
\int\|C^\dagger_{\tau(x)}h^{d-\tau(x)}\cdot \hat{p}_n(x)-\rho(x)\| dx = O_P(1).
$$
To see this, consider a point $x\in\K_s$ and let $\K_\ell$ be a higher order support ($\ell>s$)
with $x\in\overline{\K_\ell}$.
Then the region $B(x,h)\cap \K_\ell$ has $\ell$-dimensional volume at rate $O(h^{\ell-s})$.
For any point $y\in B(x,h)\cap \K_\ell$, $\tau(y)=\ell$ but
the KDE $\hat{p}_n(y)$ is at rate $O_P(h^{s-d})$. 
Thus, the difference between the scaled KDE and the generalized density is
$$
C^\dagger_{\ell}h^{d-\ell}\cdot \hat{p}_n(y)-\rho(y) = O_P(h^{s-\ell}).
$$
Such $y$ has $\ell$-dimensional volume at rate $O(h^{\ell-s})$,
so the integrated error is at rate $O_P(h^{s-\ell})\times O(h^{\ell-s}) = O_P(1)$. 
\end{remark}

Based on Theorem~\ref{thm::KDE}, we can derive a nearly uniform convergence rate of $\hat{\alpha}_n$.
\begin{thm}[Nearly uniform consistency of $\alpha$-trees]
Assume (S, P2, K1--2).
Let $\mathcal{C}_s$ be the collection of generalized critical points of $\K_s$. 
Let $\delta_{n,h,s}$ be defined in equation \eqref{eq::rate}
and $r_{n,h,s}$ be a quantity such that $\frac{\delta_{n,h,s}}{r_{n,h,s}}=o_P(1)$.
Then when $h\rightarrow 0,\frac{nh^{d+2}}{\log n}\rightarrow \infty,$ uniformly for every $x\in\K_s(h)$,
$$
\hat{\alpha}_n(x)-\alpha(x) = 
\begin{cases}
O\left(\delta_{n,h,s}\right) \quad&\mbox{if}\quad \inf_{c\in\mathcal{C}_s}|\rho(x)-\rho(c)|>r_{n,h,s},\\
O\left(\left(\delta_{n,h,s}\right)^{\frac{s}{s+1}}\right) \quad&\mbox{otherwise}.\\
\end{cases}
$$
\label{thm::alpha::point}
\end{thm}
The convergence rate in Theorem~\ref{thm::alpha::point} is similar to
that in Theorem~\ref{thm::KDE}: for a given point $x$
when $\alpha(x)$ is away from the $\alpha$ value of a generalized critical point (a critical $\alpha$ level),
we have the usual convergence rate.
When $\alpha(x)$ is close to a critical $\alpha$ level,
the convergence rate is slower.
The quantity $r_{n,h,s}$ behaves like the quantity $\varpi_n$ in Theorem~\ref{thm::alpha_non},
which was introduced to guarantee the uniform convergence. 
To derive the consistency of $\hat{\alpha}_n$ under the integrated error and the probability error,
we choose $r_{n,h,s} = \frac{\delta_{n,h,s}}{\log n}$,
which leads to the following theorem.

\begin{thm}[Consistency of $\alpha$-trees]
Assume (S, P2 , K1--2).

Then 
\begin{align*}
\|\hat{\alpha}_n-\alpha\|_P &= O\left(\delta_{n,h,d}\right),\\
\|\hat{\alpha}_n-\alpha\|_\mu &= O\left(\delta_{n,h,d}\right).
\end{align*}
\label{thm::alpha::Pro}
\end{thm}
Theorem~\ref{thm::alpha::point} shows that 
the quantity $\alpha(x)$ can be consistently estimated by
$\hat{\alpha}_n(x)$ for the majority points.
This implies that 
the ordering of points using $\hat{p}_n$ is consistent with the ordering from $\tau,\rho$ in most areas of $\K$.



\section{Singular Measures: Critical Points and Topology}	\label{sec::topology}
Recall from Section~\ref{sec::cluster_tree} that
the topology of an $\alpha$-tree $T_\alpha$ is determined by its edge set $E(T_\alpha)$
and the relation among edges $\mathbb{C}\in E(T_\alpha)$.
The set $\mathcal{A}_\alpha$ (critical tree-levels) contains the levels where the upper level set 
$\mathbb{A}_\varpi=\{x: \alpha(x)\geq \varpi\}$
changes its shape.  
For nonsingular measures, $\mathcal{A}_\alpha$ corresponds
to the density value of some critical points. 
For singular measures, this is not true even when $\rho(x)$ is a generalized
Morse function. 

Consider the example in Figure~\ref{fig::DCP1}. 
The solid box in the left panel indicates a new type of critical points, where merges between elements in different edge sets
occur (change in the topology of level sets occurs).
By the definition of $\mathcal{A}_\alpha$, this corresponds to an element in $\mathcal{A}_\alpha$, but it is clearly not a generalized critical point.
We call this type of critical points the \emph{dimensional critical points (DCPs)}.
In Figure~\ref{fig::DCP1}, the dimension $d=2$ and we have a $2D$ spherical distribution mixed with a $1D$ singular measure that
is distributed on the red curves $\K_1$.
The bluish contours are density contours of the $2D$ spherical distribution,
the crosses are locations of local modes,
and
the solid box is the location of a DCP.
To see how the solid box changes the topology of level sets, 
we display two level sets in the middle and right panels. 
In the middle panel, the level is high and there are two connected components (the gray area and the solid curve).
In the right panel, we lower the level, and now, the two connected components merge at the location of the solid box.
Although the location of the solid box does not belong to the collection of generalized critical points $\mathcal{C}$,
this point does correspond to the merging of connected components in the level sets. 
Thus, this point corresponds to an element in $\mathcal{A}_\alpha$. 

\begin{figure}
\includegraphics[width=1.5in]{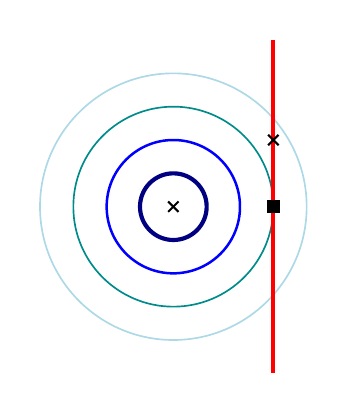}
\includegraphics[width=1.5in]{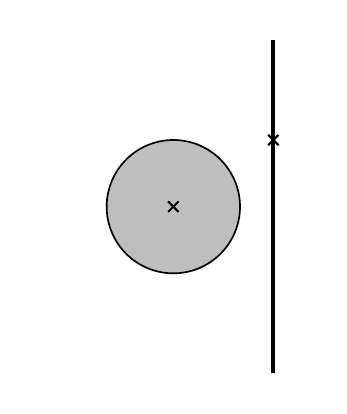}
\includegraphics[width=1.5in]{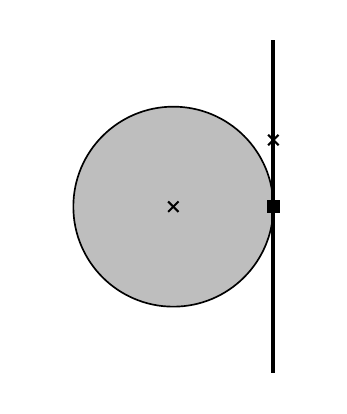}
\caption{Example of DCPs.
This is a $d=2$ case; there is a $2D$ spherical distribution mixed with
a $1D$ distribution on a line segment.
Left: the blue contours are density contours of the $2D$ spherical distribution
and the red line segment is $\K_1$, the support of the $1D$ singular distribution.
The two crosses are the density maxima at the $2D$ distribution and the $1D$ singular distribution.
The black square indicates a DCP.
To see how DCPs merge two connected components, 
we consider the middle and the right panel, which are level sets of $\alpha(x)$
at two different levels. 
Middle: the level set $\mathbb{A}_\varpi$ where the level $\varpi$ is high; 
we can see that there are two connected components (left gray-black disk and the right line segment).
Right: we move down the level a little bit; 
now the two connected components merge so there is only one connected component. 
The merging point is the square point, which is defined as a DCP.
}
\label{fig::DCP1}
\end{figure}

Here is the formal definition of the DCP.
Recall that $\mathcal{C}$ is the collection of generalized critical points of $\rho(x)$
and from equation \eqref{eq::A_f}, $\mathcal{A}_\alpha$ is the collection of levels of function $\alpha(x)$
such that the creation of a new connected component or merging of connected components occurs. 
For simplicity, we denote $\mathcal{A}=\mathcal{A}_\alpha$.
For any level $\varpi\in\mathcal{A}$,
we define
\begin{equation}
\xi(\varpi) = \begin{cases}
\max\{w\geq 0: \K_{w}\subset  \mathbb{A}_\varpi\}, \quad &\mbox{if } \K_{0}\subset \mathbb{A}_\varpi\\
-1,\quad &\mbox{if } \K_{0}\not\subset \mathbb{A}_\varpi
\end{cases}.
\end{equation}
Namely, $\xi(\varpi)$ is the highest dimensional structure that is covered by the level set $\mathbb{A}_\varpi$. 



\begin{definition}
For $\varpi\in\mathcal{A}$, we say $x$ is a DCP if the following holds
\begin{itemize}
\item[(1)] $x\in\K_\ell$ for some $\ell\leq\xi(\varpi)$.
\item[(2)] There exists an edge $\mathbb{C}\in E(T_\alpha)$ such that
	\begin{itemize}
	\item[(i)] $x\notin C$ for all $C\in\mathbb{C}$,
	\item[(ii)] $d(x,C_\epsilon)\rightarrow 0$ when $\epsilon\rightarrow 0$, where $C_\epsilon =\mathbb{C}\cap T_\alpha(\varpi+\epsilon)$.
	\end{itemize}
\end{itemize}
Note that $x$ may not exist. In such a case, there is no DCPs for level $\varpi$.
\end{definition}
The first requirement is to ensure that $x$ is on a lower dimensional support ($\K_\ell: \ell<\xi(\varpi)$).
The second requirement is to show that the DCP $x$ is not in the same edge $\mathbb{C}$, but 
its distance to the elements (connected components) of the edge is shrinking to $0$. 
By the definition of $\alpha(x)$, the first requirement implies that $x$ is contained in 
$\mathbb{A}_{\varpi+\epsilon}$ 
for a sufficiently small $\epsilon$.
Therefore, we can find $\mathbb{C}'\in E(T_\alpha)$ such that every element $C\in\mathbb{C}'$ contains $x$. 
Because $x$ is (i) in the elements of edge $\mathbb{C}'$, and (ii) not in any element of edge $\mathbb{C}$,
and because (iii) the distance from $x$ to the element of $\mathbb{C}$ converges to $0$
when the level decreases to the level $\varpi$,
$x$ is a merging point of edges $\mathbb{C}'$ and $\mathbb{C}$ 
and the level $\varpi$ is their merging level. 

Note that since DCPs occur when different dimensional regions
intersect each other, the topological structure of such an intersection
may be related to the intersection homology and stratified space \citep{edelsbrunner2008persistent,goresky1980intersection,goresky2012stratified}.

Let $\mathcal{C}^D$ be the collection of DCPs.
For a point $c\in\mathcal{C}^D$, we denote $\alpha^\dagger(c)$ as the level of $\alpha$ function
corresponding to the DCP at $c$. 
Note that $\alpha^\dagger(c)\neq \alpha(c)$ when $c$ is a DCP.
Since a DCP is generally in a lower dimensional support than the support that the merging happens,
$\alpha(c)>\alpha^\dagger(c)$.
Remark \ref{rm::ex1} provides an example of how $\alpha$ and $\alpha^\dagger$ differs.

\begin{remark}[Relation to the usual critical points]
The definition of DCPs is similar to that of saddle points or local minima, who contribute to the merging of level sets.
Saddle points (or local minima) that contribute to a merging of level sets can be defined as a point $x$ with the following properties:
\begin{itemize}
\item[(1)] $x\in\K_{\xi(\varpi)+1}$.
\item[(2)] There exists two different edges $\mathbb{C}_1,\mathbb{C}_2\in E(T_\alpha)$ such that
	\begin{itemize}
	\item[(i)] $x\notin C_1,x\notin C_2$ for all $C_1\in\mathbb{C}_1$ and $C_2\in\mathbb{C}_2$,
	\item[(ii)] $d(x,C_{1,\epsilon})\rightarrow 0$ when $\epsilon\rightarrow 0$, where $C_{1,\epsilon} =\mathbb{C}_1\cap T_\alpha(\varpi+\epsilon)$,
	\item[(iii)] $d(x,C_{2,\epsilon})\rightarrow 0$ when $\epsilon\rightarrow 0$, where $C_{2,\epsilon} =\mathbb{C}_2\cap T_\alpha(\varpi+\epsilon)$.
	\end{itemize}
\end{itemize}
It is easy to see that for a Morse function, a point $x$ with the above properties must be a saddle point or a local minimum. 
The main difference between this definition and that of DCPs is the support where $x$ lives--
if $x$ lives in a lower dimensional support $\K_\ell$, $\ell\leq \xi(\varpi)$, then it is a DCP, and
if $x$ lives in the support $\K_{\xi(\varpi)+1}$, then it is a saddle point or a local minimum.

\end{remark}

\begin{remark}	\label{rm::ex1}
Note that a DCP might be at the same position as a critical point. 
Consider the example in Figure~\ref{fig::Tree} and Example \ref{ex::point}.
In this case, 
\begin{align*}
\alpha(x) = \begin{cases}
0.7\cdot 2\cdot\Phi_0(-|x|), \quad &\mbox{if } x\neq 2\\
1, \quad &\mbox{if } x= 2
\end{cases},
\end{align*}
where $\Phi_0(x)$ is the cumulative distribution
of a standard normal.
Moreover, $\mathcal{A}_\alpha = \{1,0.7,0.0319\}$; the first element $\{1\}$ 
is the level of the point mass located at $x=2$, the second element $\{0.7\}$
is the level of the mode of the standard normal distribution,
the last element is the level where the connected components created
at levels $1$ and $0.7$ merged so it comes from a DCP.
This DCP also located at $x=2$, which coincides with a local mode. 
Note
that the number $0.0319 = 0.7\cdot 2\cdot\Phi_0(-2)$, which is the level where the merging occurred.
At the critical point $x=2$, $\alpha(2) = 1$ and $\alpha^\dagger(2) = 0.0319$.

\end{remark}

To analyze the properties of DCPs and their estimators,
we consider the following assumptions.

{\bf Assumptions.}
\begin{itemize}
\item[(A)] The elements in the collection $\mathcal{A}$ are distinct and each element corresponds 
to one critical point or one DCP, but not both. And all DCPs are distinct. 
\item[(B)] For every $x\in\partial \K_s$ ($s>0$) and $r>0$, 
there is $y\in B(x,r)\cap \K_s$ such that $\rho(y)>\rho(x)$.
\item[(C)] There exists $\eta_0>0$ such that
$$
\inf_{c\in\mathcal{C}_s}d(c, \K_\ell) \geq \eta_0,
$$
for all $\ell<s$ and $s = 1,2,\cdots,d$.
\end{itemize}

Assumption (A) is to ensure that 
no multiple topological changes will occur at the same level 
so each level corresponds to only a merging or a creation.
Assumption (B) is to guarantee that no new connected component
at the boundary of a lower dimensional manifold will be created. 
Thus, any creation of a new connected component of the level set $\mathbb{A}_\varpi$ 
occurs only at a (generalized) local mode.
Assumption (C) is to regularize (generalized) critical points so that 
they are away from lower dimensional supports.
This implies that when $h$ is sufficiently small, all critical points will be in the good region $\K(h)$.

\begin{lem}[Properties of DCPs]
Assume (S, P2, B).
The DCPs have the following properties
\begin{itemize}
\item If a new connected component of $\mathbb{A}_\varpi$ is created at $\varpi\in\mathcal{A}$, then there is a local mode $c$ of $\rho$ or an element in $\K_0$
such that $\varpi=\alpha(c).$ Namely, DCPs only merge connected components.
\item For any value $\varpi\in\mathcal{A}$, either $\varpi= \alpha(c)$ for some $c\in\mathcal{C}$ or there is a DCP associated with $\varpi$.
\end{itemize}
\label{lem::DCP::p}
\end{lem}

Lemma~\ref{lem::DCP::p} provides two basic properties of DCPs. 
First, DCPs only merge connected components. 
Moreover, when the topology of connected components of $\alpha$-level sets changes (when we decrease the level), 
either a critical point or a DCP must be responsible for this. 
Theorefore, as long as we control the stability of generalized critical points and DCPs, we 
control the topology of an $\alpha$-tree.
Thus, in what follows, we will study the stability of generalized critical points and DCPs.


\begin{lem}[Stability of generalized critical points]
Assume (S, P2, K1--2, C).
Let $c\in\K_s$ be a generalized critical point with $n(c)$ negative eigenvalues of its generalized Hessian matrix.
Let $\hat{\mathcal{C}}$ be the collection of critical points of $\hat{p}_n$.
Then when $h\rightarrow 0,\frac{nh^{d+4}}{\log n}\rightarrow \infty$,
there exists
a point $\hat{c}\in \hat{\mathcal{C}}$ such that
\begin{align*}
\|\hat{c} -c\|  &= O\left(h\right) + O_P\left(\sqrt{\frac{1}{nh^{s+2}}}\right)\\
\|\hat{\alpha}_n(\hat{c}) -\alpha(c)\|  &= O_P\left((\delta_{n,h,s})^{\frac{s}{s+1}}\right)
\end{align*}
and the estimated Hessian matrix at $\hat{c}$ has $n(c)+d-s$ negative eigenvalues.
The quantity $\delta_{n,h,s}$ is defined in equation \eqref{eq::rate}.

\label{lem::g_critical}
\end{lem}

Lemma~\ref{lem::g_critical} is a generalization of the stability theorem of 
critical points given in Lemma 16 of \cite{chazal2014robust}.
Note that the bias is now of the order $O(h)$; this is due to the 
smoothing effect from a higher dimensional support. 




\begin{lem}[Properties of estimated critical points]
Assume (S, P2, K1--2, A, B).
Assume there are $k$ DCPs.
Let $\hat{\mathcal{C}}$ be the critical points of $\hat{p}_n$.
Define $\hat{\mathcal{G}}\subset \hat{\mathcal{C}}$ as the collection of estimated critical points 
corresponding to the generalized critical points.
Let $\hat{\mathcal{D}} =\hat{\mathcal{C}}\backslash \hat{\mathcal{G}}$ be the remaining estimated critical points.
Then when $h\rightarrow 0, \frac{nh^{d+4}}{\log n}\rightarrow \infty$, 
\begin{itemize}
\item $\hat{\mathcal{D}}\subset \K^C(h)$,
\item $\left|\hat{\mathcal{D}}\right|\geq k$, where $|A|$ for a set $A$ is the cardinality,
\item $\hat{\mathcal{D}}$ contains no local mode of $\hat{p}_n$.
\end{itemize}
\label{lem::DCP}
\end{lem}

Lemma~\ref{lem::DCP} provides several useful properties of the estimated critical points 
(critical points of $\hat{p}_n$). 
First, the estimated critical points are all in the bad region $\K^C(h)$, except for those converging to generalized critical points. 
Second, the number of estimated critical points will (asymptotically) not be less than the total number of DCPs. 
Third, all estimated local modes are estimators of generalized critical points.


\begin{lem}[Stability of critical tree-levels from DCPs]
Assume (S, P2, K1--2, A, B).
Let $c$ be a DCP and $\alpha^\dagger(c)\in\mathcal{A}$ be the associated level.
Let $\hat{\mathcal{D}}$ be defined as Lemma~\ref{lem::DCP}.
Then when $h\rightarrow 0,\frac{nh^{d+2}}{\log n}\rightarrow \infty$,
there exists
a point $\hat{c}\in \hat{\mathcal{D}}$ such that
$$
\|\hat{\alpha}_n(\hat{c}) -\alpha^\dagger(c)\|  = O\left(\delta_{n,h, \xi(\alpha_0(c))+1}\right),
$$
where $\delta_{n,h,s}$ is defined in \eqref{eq::rate}.
Moreover, the $\hat{\mathbb{A}}_{\hat{\alpha}_n(\hat{c})+\epsilon}$ and $\hat{\mathbb{A}}_{\hat{\alpha}_n(\hat{c})}$
are not topological equivalent.
\label{lem::DCP_level}
\end{lem}

Lemma~\ref{lem::DCP_level} illustrates the stability of critical tree-levels from DCPs:
for every DCP, there will be an estimated critical point that corresponds to this DCP
and this estimated critical point also represents a merging of the estimated level sets.

In Lemma~\ref{lem::g_critical}, we derived the convergence rate of the estimated (generalized) critical points versus
the population critical points, but here we only derive the rate for
the critical tree-levels.
The reason is that critical points are solutions to a certain function (gradient equals to 0), 
so we can perform a Taylor expansion to obtain the convergence rate.
However, for the DCPs, they are not solutions to some functions, so it is unclear how to derive the convergence rate
for their locations.

\begin{example}[A DCP and its estimator]	\label{ex::DCP}
Consider again the example in Figure~\ref{fig::Tree} and Example \ref{ex::point}.
We have a singular distribution mixed with a point mass at $x=2$ with a probability of $0.3$ and a standard normal with a probability of $0.7$.
In this case, as indicated, a DCP is located at $x=2$ with level $0.0319$ (see Remark \ref{rm::ex1}). 
In every panel of the top row of Figure~\ref{fig::Tree}, there is a local minimum located in the region $x\in[1.5,2 ]$. 
Moreover, when we increase the sample size (from left to right), 
this local minimum is moving toward $x=2$. 
This local minimum is an estimated critical point $\hat{c}\in\hat{\mathcal{D}}$, as described in Lemma~\ref{lem::DCP_level},
whose estimated $\alpha$-level is approaching the $\alpha$-level of the DCP at $x=2$.
\end{example}




\begin{thm}[Topological error of $\alpha$-trees]
Assume (S, P2, K1--2, A, B, C).
Then when $h\rightarrow 0, \frac{nh^{d+4}}{\log n}\rightarrow \infty$,
$$
P\left(T_{\hat{\alpha}_n}\overset{T}{\approx}T_\alpha\right) \geq 1-c_0\cdot e^{-c_1\cdot nh^{d+4}},
$$
for some $c_0,c_1>0$.
\label{thm::top::p}
\end{thm}
Theorem~\ref{thm::top::p} quantifies the topological error of the estimated $\alpha$-tree under singular measures.
The error rate is the same as that in the nonsingular measures (Theorem~\ref{cor::alpha_non}). 
The topological error bound is similar to that in Corollary \ref{cor::alpha_non}. 
Both have exponential concentration bounds
with a factor of $nh^{d+4}$, which is the Hessian estimation error rate.
The two concentration bounds are similar, because as shown in Theorem~\ref{thm::KDE}, the main
difference between singular and nonsingular measures lies in the bias part,
which will not contribute to the concentration inequality as long as $h\rightarrow0$.
The Hessian error rate is because we need to make sure the signs of eigenvalues of Hessian matrices around critical points
remain unchanged.


\begin{remark}
Theorem~\ref{thm::top::p} also implies that, under singular measures, 
the cluster tree of the KDE $\hat{p}_n$ (estimated $\lambda$-tree) converges topologically
to a population cluster tree defined by the function $\alpha(x)$.
To see this, recall that by
Lemma~\ref{lem::tree}, $T_{\hat{p}_n}\overset{T}{\approx}T_{\hat{\alpha}_n}$. This, together with Theorem~\ref{thm::top::p},
implies
$$
P\left(T_{\hat{p}_n}\overset{T}{\approx}T_\alpha\right) \geq 1- c_0\cdot e^{-c_1\cdot nh^{d+4}} \rightarrow 1
$$
under a suitable choice of $h$.
This shows that even when the population distribution is singular,
the estimated $\lambda$-tree still converges topologically to the population $\alpha$-tree.
\end{remark}

\section{Discussion}	\label{sec::discussion}

In this paper, we study how the $\alpha$-tree behaves under singular and nonsingular measures.
In the nonsingular case,
the error rate under the $L_\infty$ metric is slower than other
metrics because of the slow rate of level set estimation around saddle points.
However, other error rates are the same as estimating the $\lambda$-tree.

When a distribution is singular, 
we obtain many fruitful results for both the KDE and the estimated $\alpha$-tree.
In terms of the KDE, we prove that 
\begin{itemize}
\item[1.] the KDE is a pointwise consistent estimator after rescaling;
\item[2.] the KDE is a uniformly consistent estimator after rescaling for the majority of the support; and
\item[3.] the cluster tree from the KDE (estimated $\lambda$-tree) converges topologically to a population cluster tree defined by $\alpha$.
\end{itemize}
For the estimator $\hat{\alpha}_n(x)$ and the estimated $\alpha$-tree, we show that
\begin{itemize}
\item[1.] $\hat{\alpha}_n$ is a pointwise consistent estimator of $\alpha$;
\item[2.] $\hat{\alpha}_n$ is a uniformly consistent estimator for the majority of the support;
\item[3.] $\hat{\alpha}_n$ is a consistent estimator of $\alpha$ under the integrated distance and probability distance; and
\item[4.] the estimated $\alpha$-tree converges topologically to the population $\alpha$-tree.
\end{itemize}
Moreover, we observe a new type of critical points--the DCPs--that also contribute to the merging of
level sets for singular measures.
We study the properties of DCPs and show that the estimated critical points from the KDE approximate these DCPs. 

Finally, we point out some possible future directions.
\begin{itemize}
\item {\bf Persistence Homology.}
The cluster tree is closely related to the persistent homology of level sets \citep{fasy2014confidence,bobrowski2017topological}.
In the persistent homology, a common metric for evaluating the quality of an estimator is the bottleneck distance 
\citep{cohen2007stability,edelsbrunner2012persistent}.
Because the bottleneck distance is bounded by the $L_\infty$ metric \citep{cohen2007stability, edelsbrunner2012persistent},
many bounds on the bottleneck distance are derived via bounding the $L_\infty$ metric
\citep{fasy2014confidence,bobrowski2017topological}.
However, 
for $\alpha$-trees under singular measures, 
the $L_\infty$ metric does not converge (Corollary~\ref{cor::alpha_infty_s}) but we do have topological consistency 
(Theorem~\ref{thm::top::p}), which implies convergence in the bottleneck distance.
This provides an example where we have consistency under the bottleneck distance and inconsistency in the $L_\infty$ metric. 
How this phenomenon affects the persistence homology is unclear and we leave that line of study
for future work.

\item {\bf Higher-order Homology Groups and Stratified Space.}
Our definition of DCPs is for connected components, which are zeroth-order homology groups 
\citep{cohen2007stability,fasy2014confidence,bubenik2015statistical}
and sufficient
for analyzing cluster trees.
However, critical points also contribute to the creation and elimination of higher-order homology groups 
such as loops and voids, which are not covered in this paper.
Thus, a future direction is to study whether the KDE is also consistent in recovering higher-order homology groups under singular measures.
Moreover, as is mentioned in Remark~\ref{rm::ss},
the supports we are analyzing are related to the stratified space
\citep{friedman2014singular,goresky1980intersection}
and the DCPs might be related to the intersection homology
\citep{friedman2014singular,goresky1980intersection,edelsbrunner2008persistent}.
The intersection homology extends the definition of homology group
to a stratified space so it provides a tool to analyze higher-order homology
groups in our setting. 
Thus, finding the connection between theories of stratified space 
and the higher-order homology groups in our settings will
be another future research direction.




\item {\bf Minimax theory.}
\cite{chaudhuri2010rates,chaudhuri2014consistent} derived the minimax
theory for estimating the $\lambda$-tree under nonsingular measures and 
proved that the the k-nearest neighbor estimator is minimax.
When the distribution is nonsingular, the $\alpha$-tree and $\lambda$-tree are very similar so
we expect
the minimax theory to be the same.
However, for singular measures,
it is unclear how to derive the minimax theory so we plan to investigate this in the future. 


\end{itemize}

\section*{Acknowledgements}
We thank reviewers for their very insightful comments.
We also thank members in the CMU Topstat group for useful discussion. 



\appendix

\section{proof}

We use $g$ to denote the gradient of $p$ (when it exists)
and $\hat{g}_n$ to denote the gradient of $\hat{p}_n$.

To prove Theorem~\ref{thm::critical}, 
we will use the following lemma.
\begin{lem}
Assume $p$ satisfies (P1). 
Let $\lambda$ be the density value of a critical point and $\mathcal{C}(\lambda)$
be the collection of all critical points with density value being $\lambda$.
Let $\partial L_\lambda = \{x: p(x)=\lambda\}$. 
Then there exists $\bar{R}>0$ and $0<C_{\min}\leq C_{\max}<\infty$
such that for any positive number $r<\bar{R} $,
\begin{align*}
\inf_{x\in (\partial L_\lambda\oplus \bar{R})\backslash(\mathcal{C}(\lambda)\oplus r)} \|g(x)\|&\geq C_{\min} r,\\
\sup_{x\in \mathcal{C}(\lambda)\oplus r} \|g(x)\|&\leq C_{\max} r.
\end{align*}
\label{lem::critical}
\end{lem}
The first assertion means that the low gradient region must be around critical points
and the gradient has to be linearly increasing as we move away from a critical point.
The second assertion states that 
when we move away from a critical point, the gradient cannot increase faster than a linear rate.
Essentially, these two properties are just natural outcomes
from the fact that the density $p$ is a Morse function (Hessian matrix at critical points have non-zero eigenvalues)
and has bounded third derivatives.

\begin{proof}
{\bf First assertion.}
By the continuity of the gradient of $p$ (assumption (P1) implies bounded Hessian matrix, which implies the continuity of gradient), 
regions with low gradient must be around a critical point.

Let $c_0$ be a critical point of $\mathcal{C}(\lambda)$.
When $\bar{R}$ is sufficiently small, 
for any unit vector $\nu\in\R^d$
and any two positive numbers $r_1<r_2<\bar{R}$,
we have
$$
0<\|g(c_0+r_1\nu) \|< \|g(c_0+r_2\nu)\|.
$$
Namely, if we are moving away from a critical point along a fixed orientation, the amount gradient will be increasing
as long as we do not move too far away.
This is because the density around a critical point behaves like a quadratic function due to the Morse lemma
(see, e.g., Lemma 3.11 of \citealt{banyaga2013lectures}). 
Thus,
when $\bar{R}$ is sufficiently small and $r<\bar{R}$, 
the infimum $\inf_{x\in (\partial L_\lambda\oplus \bar{R})\backslash(\mathcal{C}(\lambda)\oplus r)} \|g(x)\|$ 
occurs at the boundary $\partial (\mathcal{C}(\lambda)\oplus r)$.
Let $x_0\in \partial (\mathcal{C}(\lambda)\oplus r)$ be the point where the infimum occurs
and $c$ be the critical point of $\mathcal{C}(\lambda)$ such that $\|x_0-c\| = r$.

Because $p$ is a Morse function, all eigenvalues of $H^2(c)$ must be non-zero. 
Thus, the fact that $p$ has a bounded third derivative (assumption (P1))
implies that when $\bar{R}$ is sufficiently small, 
all eigenvalues of $H^2(x)$ must be greater than or equal to a positive constant $C_1$ uniformly for all $x\in B(c,\bar{R})$.
Define a line $L_0(t) = t\cdot x_0 + (1-t)\cdot c$ passing through $c$ and $x_0$.
Then the gradient
\begin{align*}
g(x_0) &= \underbrace{g(c)}_{=0\mbox{ ($c$ is a critical point)}} + \int_{t=0}^{t= 1}dg(L_0(t))\\
& = \int_{t=0}^{t= 1}H(L_0(t))\underbrace{L'_0(t)}_{= x_0-c}dt,
\end{align*}
which implies
$$
\|g(x_0)\| = \left\|\int_{t=0}^{t= 1}H(L_0(t))(x_0-c)dt\right\|  \geq  \sqrt{C_1} \|x_0-c\| = C_{\min}r,
$$
which proves the first assertion.



{\bf Second assertion.}
By assumption (P1), every element of $H(x)$ is uniformly bounded
so the eigenvalues of $H^2(x)$ is also bounded.
Let $C_{\max} = \sup_x \| H(x)\|_{*}$
be the uniform bound on the spectral norm of $H(x)$.
For any point $x\in B(c,r)$,
again we define the line $L_0(t) = t\cdot x_0 + (1-t)\cdot c$ passing through $c$ and $x_0$.
Then,
\begin{align*}
g(x_0) &= \underbrace{g(c)}_{=0\mbox{ ($c$ is a critical point)}} + \int_{t=0}^{t= 1}dg(L_0(t))\\
& = \int_{t=0}^{t= 1}H(L_0(t))\underbrace{L'_0(t)}_{= x_0-c}dt
\end{align*}
and
$$
\|g(x_0)\| = \left\|\int_{t=0}^{t= 1}H(L_0(t))(x_0-c)dt\right\| \leq C_{\max} \|x-c\| \leq C_{\max} r,
$$
which completes the proof.

\end{proof}

\begin{proof}[Proof of Theorem~\ref{thm::critical}]

Let $c$ be a critical point with $\lambda = p(c)$. We denote $\delta_\infty = \|\hat{p}_n - p\|_\infty$. 
For a point $y\in \hat{L}_\lambda$,
$$
\lambda \leq \hat{p}_n(y) \leq p(y) +\delta_\infty\quad \Rightarrow\quad p(y) \geq \lambda-\delta_\infty.
$$
Thus, for a point $y\in \hat{L}_\lambda\backslash L_\lambda$, we have
$$
\lambda > p(y) \geq \lambda-\delta_\infty.
$$
Similarly, for any point $y\in L_\lambda\backslash \hat L_\lambda$,
we have
$$
\lambda+\delta_\infty > p(y) \geq \lambda.
$$
Thus,
the set difference 
$$
(\hat{L}_\lambda\triangle L_\lambda)\subset D_{\lambda,\infty} = \{x: |p(x)-\lambda|\leq \delta_\infty\},
$$
which further implies 
$$
\mu\left(\hat{L}_\lambda\triangle L_\lambda\right) \leq \mu(D_{\lambda,\infty}).
$$
So we
can bound the quantity $\mu(D_{\lambda,\infty})$ to complete the proof.

Note that by the continuity of $p(x)$ and the fact that $p$ is a Morse function (implying no flat density region), 
$\delta_\infty\rightarrow0$
implies that ${\sf Haus}(D_{\lambda,\infty}, \partial L_\lambda) \rightarrow 0$,
where $\partial L_\lambda = \{x: p(x)=\lambda\}$ is the boundary of $L_\lambda$.
Thus, Lemma~\ref{lem::critical} holds when $\delta_{\infty}$ is sufficiently small.



We partition the set $D_{\lambda,\infty}$ into two regions:
\begin{align*}
A_n &= D_{\lambda,\infty} \bigcap  B^C(c, R_n),\\
B_n &= D_{\lambda,\infty} \bigcap B(c,R_n),
\end{align*}
where $R_n$ is a number that shrinks to $0$ when $\delta_\infty\rightarrow 0$.
Later we will provide an explicit expression for $R_n$.
We only consider the case where $c$ is unique and is a saddle point; 
one can easily generalize the proof to the case of multiple critical points 
because $p$ is a Morse function so every critical point is well-separated (see, e.g., Lemma 3.2 of \citealt{banyaga2013lectures}).

The proof requires two simple geometric facts:
\begin{itemize}
\item {\bf Fact 1:}
Let $\gamma_y(t)$ be a curve starting from $y$ passing through a point $x$ at time $t_x$.
Namely, $\gamma_y(0)=y$ and $\gamma_y(t_x)=x$.
Then
$$
\|x-y\|\leq \int_0^{t_x} \|\gamma'_y(t)\|dt,
$$
i.e, the distance between two points is less than or equal to the length of any curve connecting the two points. 

\item {\bf Fact 2:}
If $\partial L_\lambda$ has finite $(d-1)$-dimensional volume (i.e., finite surface area),
then when $r_n \rightarrow 0$, 
$$
\mu\left(\partial L_\lambda \oplus r_n\right) = O(r_n),
$$
where $\mu(\cdot)$ is the Lebesgue measure.
Note that this fact is because $p$ is a Morse function
so the level set $\partial L_\lambda = \{x: p(x)=\lambda\}$
is a $(d-1)$ rectifiable set and thus the Minkowski content ($\lim_{\epsilon\rightarrow0}\frac{\mu\left(\partial L_\lambda \oplus \epsilon\right)}{2\epsilon}$)
equals to the $(d-1)$-dimensional volume.


\end{itemize}

For simplicity, we first consider the points in $A_n$ with a density value below $\lambda$.
Let $x \in A_n\cap \{y: p(y)< \lambda\}$ be one such point.
This implies $p(x)<\lambda$.
Define a density gradient (ascent) flow $\pi_x$ starting at $x$ as follows:
$$
\pi_x(0) = 0, \quad \pi_x'(t) = g(\pi_x(t)).
$$
Namely, the gradient flow $\pi_x$ is a flow starting at $x$ that follows the gradient of the density function. 
Because this flow follows the density gradient, the density along this flow $p(\pi_x(t))$ is increasing with respect to $t$.
Let $t_0$ be the time such that $p(\pi_x(t_0))= \lambda$. 
Then we have
\begin{equation}
\begin{aligned}
\lambda -p(x) &= \int_0^{t_0} \left(\frac{d p(\pi_x(t))}{dt}\right)dt \\
&=  \int_0^{t_0} g^T(\pi_x(t))\pi_x'(t)dt = \int_0^{t_0} \|g(\pi_x(t))\|^2dt
\label{eq::t1}
\end{aligned}
\end{equation}
because $\pi_x'(t) = g(\pi_x(t))$ by construction.
By the definition of $A_n$, $\lambda -p(x)\leq \delta_\infty$ so we conclude
\begin{equation}
\delta_\infty \geq\int_0^{t_0} \|g(\pi_x(t))\|^2dt.
\label{eq::t2}
\end{equation}
Now we partition the interval $[0,t_0]$ into $T_A$ and $T_B$
where $T_A = \{t: \pi_x(t)\in A_n\}$ and $T_B = \{t: \pi_x(t)\in B_n\}$. 
Denote $t_A = \int_{T_A} dt$ and $t_B = \int_{T_B}dt$ as the length of $T_A$ and $T_B$. 
Note that $t_0 = t_A+t_B$.
Using $T_A$ and $T_B$,
we rewrite
equation \eqref{eq::t2} as
\begin{equation}
\delta_\infty \geq \int_0^{t_0} \|g(\pi_x(t))\|^2dt  =\int_{T_A}\|g(\pi_x(t))\|^2dt + \int_{T_B} \|g(\pi_x(t))\|^2dt. 
\label{eq::t2_1}
\end{equation}

In the case of $T_A$, by the first assertion of Lemma~\ref{lem::critical}, 
we have a lower bound on the gradient: $\|g(\pi_x(t))\|\geq C_{\min} R_n$
for some constant $C_{\min}$.
Thus, 
\begin{equation}
\int_{T_A}\|g(\pi_x(t))\|^2dt  \geq \int_{T_A}\left(C_{\min} R_n\right)^2dt = C^2_{\min} R^2_n t_A.
\label{eq::t3_a}
\end{equation}

To bound $T_B$, for simplicity we assume $T_B=[t_*, t_{**}]$ 
(if $T_B$ contains multiple intervals, we can do the same thing for each of them). 
At the beginning of $T_B$ ($t=t_*$), $\pi_x(t_*) $ is on the boundary of the ball $B(c, R_n)$
so $g(\pi_x(t_*))\geq C_{\min} R_n$ by Lemma \ref{lem::critical}.
For $t\in [t_*, t_{**}]$, we have
\begin{equation}
\begin{aligned}
\|g(\pi_x(t))\|^2 &= \|g(\pi_x(t_*))\|^2 + \int_{s=t_*}^{s=t} \left(\frac{d \|g(\pi_x(s))\|^2}{ds}\right)ds\\
&\geq C^2_{\min} R_n^2 + \int_{s=t_*}^{s=t}2\pi'_x(s)^T H(\pi_x(s)) g(\pi_x(s)) ds\\
&=C^2_{\min} R^2_n + 2\int_{s=t_*}^{s=t}g^T(\pi_x(s)) H(\pi_x(s)) g(\pi_x(s)) ds.
\end{aligned}
\label{eq::bn}
\end{equation}
Note that in the last equality we use the property of a gradient flow: $\pi'_x(s) = g(\pi_x(s))$. 
Because $\pi_x(s) \in B(c,R_n)$ when $s\in [t_*,t_{**}]$, 
the second assertion of Lemma \ref{lem::critical} implies $\|g(\pi_x(s))\|\leq C_{\max} R_n$.
Let $\|M\|_{*}$ denotes the spectral norm of the matrix $M$.
Then
$$
g^T(\pi_x(s)) H(\pi_x(s)) g(\pi_x(s)) \geq -\sup_{y}\|H(y)\|_{*} \sup_{y\in B_n}\|g(y)\|^2 \geq - C'_{\max} R_n^2,
$$
for some constant $C'_{\max}>0$.
Note that the spectral norm is bounded because of (P1) -- the second derivative of $p$ is uniformly bounded.  
Thus, equation \eqref{eq::bn} implies
$$
\|g(\pi_x(t))\|^2 \geq C^2_{\min} R^2_n - C'_{\max} R_n^2\int_{s=t_*}^{s=t}ds = C^2_{\min} R^2_n - C'_{\max} R_n^2(t-t_*).
$$
This implies the following bound
\begin{equation}
\begin{aligned}
\int_{T_B}\|g(\pi_x(t))\|^2dt  &\geq \int_{t_*}^{t_{**}} \left(C^2_{\min} R^2_n - C'_{\max} R_n^2(t-t_*)\right)dt\\
& = C^2_{\min} R^2_n t_B - 0.5C'_{\max} R_n^2t_B^2.
\end{aligned}
\label{eq::t3_b}
\end{equation}

Putting equations \eqref{eq::t3_a} and \eqref{eq::t3_b} into equation \eqref{eq::t2_1} and using the fact that 
$t_0 = t_A+t_B$,
we obtain
\begin{align*}
\delta_\infty &\geq\int_{T_A}\|g(\pi_x(t))\|^2dt + \int_{T_B} \|g(\pi_x(t))\|^2dt\\
&\geq C^2_{\min} R^2_n t_A+ C^2_{\min} R^2_n t_B - 0.5C'_{\max} R_n^2t_B^2\\
& = C^2_{\min} R^2_n t_0 - 0.5C'_{\max} R_n^2t_B^2\\
&\geq C^2_{\min} R^2_n t_0 - 0.5C'_{\max} R_n^2t_0^2\\
&\geq C'_{\min} R^2_n t_0
\end{align*}
for some positive constant $C'_{\min}$
when $t_0\rightarrow0$ (later we will show that this is true).
This implies
\begin{equation}
t_0  = O(\delta_\infty/R_n^2).
\label{eq::t3}
\end{equation}



Let $x_\lambda = \pi_x(t_0)\in L_\lambda$ be the point where $\pi_x(t)$ intersect $L_\lambda$.
By {\bf Fact 1}, 
\begin{align*}
\|x_\lambda - x\| &\leq \int_0^{t_0} \|\pi'_x(t)\|dt\\
&\leq \sqrt{\int_0^{t_0} \|\pi'_x(t)\|^2dt \int_0^{t_0} 1^2dt} \quad \mbox{(Cauchy-Schwarz inequality)}\\
&= \sqrt{\int_0^{t_0} \|g(\pi_x(t))\|^2dt \int_0^{t_0} 1^2dt} \quad \mbox{($\pi'_x(t) = g(\pi_x(t))$)}\\
& \leq \sqrt{ \delta_\infty t_0} \quad \mbox{(equation \eqref{eq::t2})}\\
& = O(\delta_\infty/R_n)\quad \mbox{(equation \eqref{eq::t3})}.
\end{align*}
Thus, we conclude that $A_n\cap \{y: p(y)< \lambda\} \subset (\partial L_\lambda \oplus O(\delta_\infty/R_n))$.
Similarly, we can prove that $A_n\cap \{y: p(y)> \lambda\} \subset (\partial L_\lambda \oplus O(\delta_\infty/R_n))$
by considering a gradient descent flow.
So we conclude
$$
A_n \subset (\partial L_\lambda \oplus O(\delta_\infty/R_n)).
$$
Now using {\bf Fact 2} (because $p$ is a Morse function and has bounded 4 times differentiations, 
$\partial L_\lambda$ has a bounded surface area), we conclude
$$
\mu(A_n) = O(\delta_\infty/R_n)
$$
when $\delta_\infty/R_n\rightarrow 0$.

Because $B_n = D_{\lambda,\infty} \bigcap B(c,R_n)\subset B(c,R_n)$,
$$
\mu(B_n) \leq \mu(B(c,R_n)) = O(R_n^d). 
$$
Therefore, putting it altogether,
we conclude
$$
\mu\left(\hat{L}_\lambda\triangle L_\lambda\right) \leq \mu(D_{\lambda,\infty}) = \mu(A_n)+\mu(B_n) \leq O(\delta_\infty/R_n) + O(R_n^d).
$$
Choosing $R_n = \delta_\infty^{\frac{1}{d+1}}$, we obtain
$$
\mu\left(\hat{L}_\lambda\triangle L_\lambda\right) \leq O\left(\delta_\infty^{\frac{d}{d+1}}\right),
$$
which is the desired rate. 
Note that this choice of $R_n$ does satisfy the requirement in equation \eqref{eq::t3} that $t_0 = \delta_\infty/R^2_n =\delta_\infty^{\frac{d-1}{d+1}} \rightarrow 0$
when $d\geq 2$
so we have completed the proof.

\end{proof}

Before proving Theorem~\ref{thm::alpha_non}, we first note the following lemma from \cite{chazal2014robust}. 
\begin{lem}[Lemma 16 in \cite{chazal2014robust}]
Let $p$ be a density with compact support $S$. Assume that
$S$ is a $d$-dimensional compact submanifold of $\mathbb{R}^d$ with boundary. Assume $p$ is a Morse function with
finitely many, distinct, critical values with corresponding critical points $c_1,\cdots, c_k$. Also assume that $p$
is twice continuously differentiable on the interior of $S$, continuous and differentiable with non vanishing gradient on the
boundary of $S$.
Then
there exists $\epsilon_0>0$ such that for all $0<\epsilon<\epsilon_0$ the following is true: 
\begin{itemize}
\item[] for some positive
constant $c$, there exists $\eta\geq c\epsilon$ such that, for any density $q$ with support $S$ satisfying
$$
\|p-q\|_\infty < \eta, \sup_{x} \|\nabla p(x) - \nabla q(x)\|_{\max} <\eta, \sup_x \|\nabla\nabla p(x) - \nabla \nabla p(x)\|_{\max}<\eta,
$$
$q$ is a Morse function with exactly $k$ critical points $c'_1,\cdots, c'_k$ such that 
after relabeling, 
$$
\|c_\ell - c'_\ell\|<\epsilon
$$
and the number of positive eigenvalues of $\nabla\nabla p(c_\ell)$ and $\nabla \nabla q(c'_\ell)$
are the same.
\end{itemize}
\label{lem::chazal}
\end{lem}
Essentially, Lemma~\ref{lem::chazal}
shows that when the density, gradient, and Hessian of two functions
are sufficiently close, 
there exists a unique $1-1$ correspondence between critical points in the two functions. 
Thus, any critical point of $p$ corresponds
to a unique critical point of $q$. 
In the proof of Theorem~\ref{thm::alpha_non}, we will replace $q$ by the KDE $\hat{p}_n$.

\begin{proof}[Proof of Theorem~\ref{thm::alpha_non}]
Recall that $\hat{\alpha}_n(x) = \hat{P}_n\left(\hat{L}_{\hat{p}_n(x)}\right)$
and $\alpha(x) = P\left(L_{p(x)}\right)$.
The main idea for the proof is the following decomposition:
\begin{align*}
\hat{\alpha}_n(x) - \alpha(x) &= 
\underbrace{\hat{P}_n\left(\hat{L}_{\hat{p}_n(x)}\right) -P\left(\hat{L}_{\hat{p}_n(x)}\right)}_\text{(A)}\\
&\quad+\underbrace{P\left(\hat{L}_{\hat{p}_n(x)}\right) - P\left(\hat{L}_{p(x)}\right)}_\text{(B)}
+\underbrace{P\left(\hat{L}_{p(x)}\right) - P\left(L_{p(x)}\right).}_\text{(C)}
\end{align*}

We first bound part (B) and part (C).
We will use the results in these two parts to bound the case of part (A).

{\bf Part (B).}
Because
\begin{equation}
|P\left(\hat{L}_{\hat{p}_n(x)}\right) - P\left(\hat{L}_{p(x)}\right)|  \leq P\left(\hat{L}_{\hat{p}_n(x)}\triangle \hat{L}_{p(x)}\right),
\label{eq::alpha::pf::0}
\end{equation}
we first investigate a more general bound on the right hand side.
Let $\epsilon>0$ be a small number.
For the estimated level sets $\hat{L}_\lambda$ and $\hat{L}_{\lambda+\epsilon}$,
we want to control the quantity
$P(\hat{L}_\lambda\triangle \hat{L}_{\lambda+\epsilon})$
when $\epsilon\rightarrow 0$.
To obtain the bound in equation \eqref{eq::alpha::pf::0},
we choose $\lambda = p(x)$ and $\epsilon = \hat{p}_n(x)-p(x)$.

An interesting result is that $P(\hat{L}_\lambda\triangle \hat{L}_{\lambda+\epsilon})$ differs 
if $\lambda$ is a critical value (i.e. density value of a critical point) or not.
When $\lambda$ is not a critical value,
the gradient $g(x)$ has a nonzero lower bound on $\partial \hat{L}_\lambda$ and
we have the local approximation
$$
d(x, \hat{L}_{\lambda+\epsilon}) = \frac{\epsilon}{\|g(x)\|} + o(\epsilon)
$$
for each $x\in \partial \hat{L}_\lambda.$
Thus, this implies $\hat{L}_{\lambda+\epsilon}\subset \hat{L}_\lambda\oplus O(\epsilon)$
and 
the set $\hat{L}_{\lambda+\epsilon}\backslash \hat{L}_{\lambda}$ 
has Lebesgue measure $O(\epsilon)$ because of the normal compatibility property and the fact that $\partial \hat{L}_{\lambda}$ has 
finite $d-1$ volume. 
As a result,
\begin{equation}
P(\hat{L}_\lambda\triangle \hat{L}_{\lambda+\epsilon}) = P(\hat{L}_{\lambda+\epsilon}\backslash\hat{L}_\lambda) = O(\epsilon).
\label{eq::alpha::pf::1-1}
\end{equation}

When $\lambda$ coincides with a critical value, we need to split the region 
$\hat{L}_\lambda\triangle \hat{L}_{\lambda+\epsilon}$ into two subregions:
\begin{align*}
A_n &= (\hat{L}_\lambda\triangle \hat{L}_{\lambda+\epsilon}) \bigcap  B^C(c, R_\epsilon),\\
B_n &= (\hat{L}_\lambda\triangle \hat{L}_{\lambda+\epsilon}) \bigcap B(c,R_\epsilon),
\end{align*}
where $R_0$ is a nonzero constant and $R_\epsilon$ is a constant converging to $0$ when $\epsilon\rightarrow 0$.

This idea is similar to the proof of Theorem~\ref{thm::critical}.
Using the same calculation as Part $A_n$ in the proof of Theorem~\ref{thm::critical}, we obtain the rate
$$
P(A_n) = O_P\left(\frac{\epsilon}{R_\epsilon}\right).
$$
The other part $B_n$ controls the area around critical points. 
This part contributes $P(B_n) = O_P(\epsilon^d)$.
Thus, putting it all together, we have
\begin{equation}
\begin{aligned}
P(\hat{L}_\lambda\triangle \hat{L}_{\lambda+\epsilon}) &= P(A_n) + P(B_n) \\
&= O_P\left(\frac{\epsilon}{R_\epsilon}\right)+O_P(R_\epsilon^d)\\
&= O_P\left(\epsilon^{\frac{d}{d+1}}\right)
\end{aligned}
\label{eq::alpha::pf::1}
\end{equation}
when we choose $R_\epsilon = O(\epsilon^{\frac{1}{d+1}})$.

Note that the area around the critical point is at rate $\sqrt{\epsilon}$.
This is because the density at a critical point behaves quadratically, so a difference in density of $\epsilon$
results in a difference in distance of $\sqrt{\epsilon}$.
In addition, the choice $R_n = O(\epsilon^{\frac{1}{d+1}})$ is obviously slower than $\sqrt{\epsilon}$, so equation \eqref{eq::alpha::pf::1}
is valid.

Substituting equations \eqref{eq::alpha::pf::1-1} and \eqref{eq::alpha::pf::1} into equation \eqref{eq::alpha::pf::0}, 
we conclude that 
\begin{itemize}
\item if the density at $x$, $p(x)$, is not a critical value of $\hat{p},$
$$
\left|P\left(\hat{L}_{\hat{p}_n(x)}\right) - P\left(\hat{L}_{p(x)}\right)\right| = O_P\left(|\hat{p}_n(x)-p(x)|\right);
$$
\item if the density at $x$, $p(x)$, is a critical value of $\hat{p},$
$$
\left|P\left(\hat{L}_{\hat{p}_n(x)}\right) - P\left(\hat{L}_{p(x)}\right)\right| = O_P\left(|\hat{p}_n(x)-p(x)|^{\frac{d}{d+1}}\right).
$$
\end{itemize}

Finally, 
for a point $x$, 
the case where its density $p(x)$ might be a critical value of $\hat{p}_n$
occurs only when 
\begin{equation}
|p(x)-p(c)|\leq \|\hat{p}_n-p\|_\infty+O(\|\hat{g}_n-g\|_\infty^2) = O_P(\|\hat{p}_n-p\|_\infty)
\label{eq::crtical_bound}
\end{equation}
for some critical point $c$ of $p$.
To see equation \eqref{eq::crtical_bound},
first note for each critical point $c$ of $p$, there exists a critical point $\hat{c}$ of $\hat{p}_n$
that is close to $c$ due to Lemma~\ref{lem::chazal}
(this is because the gradient and Hessian of $\hat{p}_n$ are close to the gradient and Hessian of $p$, respectively). 
Then by Taylor expansion
$$
p(\hat c) - p(c) = (\hat{c}-c)^T\underbrace{g(c)}_{=0} + (\hat{c}-c)^T H(x) (\hat{c}-c) + O_P(\|\hat{c}-c\|^3) = O_P(\|\hat{c}-c\|^2).
$$
To bound the rate $\|\hat{c}-c\|$, 
note that $\hat{g}_n(\hat{c})=g(c)=0$
so
$$
\hat{g}_n(c)-g(c) = \hat{g}_n(c) - \hat{g}_n(\hat{c}) =  \hat{H}_n(c)(c-\hat{c}) + O_P(\|\hat{c}-c\|^2).
$$
Multiplying both sides by $\hat{H}^{-1}_n(c) $ from the left and take the norm, we obtain
\begin{align*}
\|\hat{c}-c\| &= \hat{H}_n^{-1}(c) (\hat{g}_n(c)-g(c))(1+o_P(1))\\ 
&= O_P\left(\|\hat{g}_n(c)-g(c)\|\right) \\
&= O_P\left(\|\hat{g}_n-g\|_\infty\right).
\end{align*}
Moreover, 
because $\|\hat{g}_n-g\|_\infty^2 = O(h^4) + O_P\left(\frac{\log n}{nh^{d+2}}\right)$
and $\|\hat{p}_n-p\|_\infty = O(h^2) + O_P\left(\sqrt{\frac{\log n}{nh^{d}}}\right)$,
the requirement on $h$: $h\rightarrow0, \frac{\log n}{nh^{d+4}}\rightarrow0$, implies $\|\hat{g}_n-g\|_\infty^2 = o_P(\|\hat{p}_n-p\|_\infty)$.
Putting it altogether, 
\begin{align*}
|\hat{p}_n(\hat{c})-p(c)|&\leq  |\hat{p}_n(\hat{c})-p(\hat{c})| + |p(\hat{c})-p(c)|\\
&\leq \|\hat{p}_n-p\|_\infty + O_P(\|\hat{c}-c\|^2)\\
& = \|\hat{p}_n-p\|_\infty + O_P(\|\hat{g}_n-g\|_\infty^2)\\
& = O_P\left(\|\hat{p}_n-p\|_\infty\right),
\end{align*}
which proves equation \eqref{eq::crtical_bound}.


Recall that $a_n$ is a sequence of $n$ such that $\|\hat{p}_n-p\|_\infty= o(a_n)$.
Then the above bound can be rewritten as
$$
\left|P\left(\hat{L}_{\hat{p}_n(x)}\right) - P\left(\hat{L}_{p(x)}\right)\right| = 
\begin{cases}
O_P\left(|\hat{p}_n(x)-p(x)|\right) \quad&\mbox{, if}\quad |p(x)-p(c)|>a_n\mbox{ for all $c\in\mathcal{C}$.}\\
O_P\left(|\hat{p}_n(x)-p(x)|^{\frac{d}{d+1}}\right) \quad&\mbox{, otherwise}.
\end{cases}
$$
This bound is uniform for all $x$ because the sequence $a_n$ does not depend on $x$.

{\bf Part (C).}
This part is simply applying Theorem~\ref{thm::critical}, which shows that
\begin{itemize}
\item if the density at $x$, $p(x)$, is not a critical value of $p$,
$$
\left|P\left(\hat{L}_{p(x)}\right) - P\left(L_{p(x)}\right)\right| = O_P\left(\|\hat{p}_n-p\|_\infty\right);
$$
\item if the density at $x$, $p(x)$, is a critical value of $p$,
$$
\left|P\left(\hat{L}_{p(x)}\right) - P\left(L_{p(x)}\right)\right| = O_P\left(\|\hat{p}_n-p\|_\infty^{\frac{d}{d+1}}\right).
$$
\end{itemize}

{\bf Part (A).}
Let $\delta_{\infty} = \|\hat{p}_n-p\|_\infty$. 
It is easy to see that 
\begin{equation}
L_{\hat{p}_n(x)+\delta_{\infty}} \subset \hat{L}_{\hat{p}_n(x)} \subset L_{\hat{p}_n(x)-\delta_{\infty}}.
\label{eq::pf::thm4-1}
\end{equation}
This is
because any point $z\in L_{\hat{p}_n(x)+\delta_{\infty}}$ satisfies $p(z)\geq\hat{p}_n(x)+\delta_{\infty}$
and thus, $\hat{p}_n(z) \geq p(z)-\delta_{\infty} \geq\hat{p}_n(x)$, which implies $z\in \hat{L}_{\hat{p}_n(x)}$.
The other case can be derived similarly. 
Equation \eqref{eq::pf::thm4-1} implies 
\begin{align*}
P(L_{\hat{p}_n(x)+\delta_{\infty}}) &\leq P(\hat{L}_{\hat{p}_n(x)} )\leq P(L_{\hat{p}_n(x)-\delta_{\infty}})\\
\hat{P}_n(L_{\hat{p}_n(x)+\delta_{\infty}}) &\leq \hat{P}_n(\hat{L}_{\hat{p}_n(x)} )\leq \hat{P}_n(L_{\hat{p}_n(x)-\delta_{\infty}})
\end{align*}
Now because the collection of set $\mathcal{L} =\{L_{\lambda}: \lambda>0\}$ has VC dimension $1$,
by VC theory (see, e.g., Theorem 2.43 in \citealt{wasserman2006all}),
$$
\sup_{\lambda} |\hat{P}_n(L_\lambda)-P(L_\lambda)| = O_P\left(\sqrt{\frac{\log n}{n}}\right).
$$
Therefore,
$$
P(L_{\hat{p}_n(x)+\delta_{\infty}})+O_P\left(\sqrt{\frac{\log n}{n}}\right) \leq \hat{P}_n(\hat{L}_{\hat{p}_n(x)} )\leq P(L_{\hat{p}_n(x)-\delta_{\infty}})+O_P\left(\sqrt{\frac{\log n}{n}}\right).
$$
This implies that 
$$
|\hat{P}_n(\hat{L}_{\hat{p}_n(x)} )-P(\hat{L}_{\hat{p}_n(x)} )| \leq P(L_{\hat{p}_n(x)-\delta_{\infty}})-P(L_{\hat{p}_n(x)+\delta_{\infty}}) + O_P\left(\sqrt{\frac{\log n}{n}}\right).
$$
By a similar derivation as part (B), we will obtain a similar bound as equations \eqref{eq::alpha::pf::1-1} and \eqref{eq::alpha::pf::1}
\begin{align*}
P(L_{\hat{p}_n(x)-\delta_{\infty}})&-P(L_{\hat{p}_n(x)+\delta_{\infty}}) 
\\
&=\begin{cases}
O_P\left(\delta_\infty\right)+O_P\left(\sqrt{\frac{\log n}{n}}\right) \,\,&\mbox{, if}\quad |\hat{p}(x)-p(c)|>\delta_\infty\mbox{ for all $c\in\mathcal{C}$.}\\
O_P\left(\delta_\infty^{\frac{d}{d+1}}\right)+O_P\left(\sqrt{\frac{\log n}{n}}\right) \,\,&\mbox{, otherwise}.
\end{cases}
\end{align*}
Note that the term $O_P\left(\sqrt{\frac{\log n}{n}}\right)$ in the above bound is of a smaller order so we can ignore them.
As a result, part (A) contributes
\begin{align*}
|\hat{P}_n(\hat{L}_{\hat{p}_n(x)} )&-P(\hat{L}_{\hat{p}_n(x)} )|\\
&=
\begin{cases}
O_P\left(\|\hat{p}_n-p\|_\infty\right) \quad&\mbox{, if}\,\, |\hat{p}(x)-p(c)|>\delta_\infty\mbox{ for all $c\in\mathcal{C}$.}\\
O_P\left(\|\hat{p}_n-p\|_\infty^{\frac{d}{d+1}}\right) \,\,&\mbox{, otherwise}.
\end{cases}
\end{align*}



Thus, putting it all together and using the fact that $\delta_\infty=\|\hat{p}_n-p\|_\infty= o_P(a_n)$, we conclude that uniformly for all $x$,
$$
\hat{\alpha}_n(x) - \alpha(x) = 
\begin{cases}
O_P\left(\|\hat{p}_n-p\|_\infty\right) \quad&\mbox{, if}\quad |p(x)-p(c)|>a_n\mbox{ for all $c\in\mathcal{C}$,}\\
O_P\left(\|\hat{p}_n-p\|_\infty^{\frac{d}{d+1}}\right) \quad&\mbox{, otherwise}.
\end{cases}
$$

\end{proof}

\begin{proof}[Proof of Lemma~\ref{lem::smooth::bias}]
{\bf Part 1: pointwise bias.}
Without loss of generality, we assume $x\in \mathring{\K}_s$. 
We first consider the case $m(x)>0$. 
In this case, there is a higher-dimensional support $\K_{s+m(x)}$ such that
$x\in \overline{\K}_{s+m(x)}$.
Thus, for any $r>0$, the ball $B(x,r)\bigcap \K_{s+m(x)} \neq \phi$.

Because the kernel function $K$ is supported on $[0,1]$,
$$
p_h(x) = \mathbb{E}\left(\hat{p}_n(x)\right) = \int\frac{1}{h^d}K\left(\frac{\|x-y\|}{h}\right)dP(y)
=\sum_{\ell=0}^d \int_{\K_\ell}\frac{1}{h^d}K\left(\frac{\|x-y\|}{h}\right)dP(y).
$$
When $h$ is sufficiently small, $B(x,h)\bigcap\K_{\ell}=\phi$ for any $\ell<s$
so the above expression can be rewritten as
$$
p_h(x) = \sum_{\ell=s}^d \int_{\K_\ell}\frac{1}{h^d}K\left(\frac{\|x-y\|}{h}\right)dP(y).
$$
Now by the definition of $m(x)$, $B(x,r)\bigcap\K_{\ell} =\phi$ for every $\ell>s$ and $\ell<s+m(x)$.
Thus, we can again rewrite $p_h(x)$ as
\begin{equation}
\begin{aligned}
p_h(x) &=  \int_{\K_s}\frac{1}{h^d}K\left(\frac{\|x-y\|}{h}\right)dP(y)\\
&\quad+ \sum_{\ell\geq s+m(x)}^d \int_{\K_\ell}\frac{1}{h^d}K\left(\frac{\|x-y\|}{h}\right)dP(y)\\
&=  \underbrace{\int_{\K_s\bigcap B(x,h)}\frac{1}{h^d}K\left(\frac{\|x-y\|}{h}\right)dP(y)}_\text{(I)}\\
&\quad+ \underbrace{\sum_{\ell\geq s+m(x)}^d \int_{\K_\ell\bigcap B(x,h)}\frac{1}{h^d}K\left(\frac{\|x-y\|}{h}\right)dP(y)}_\text{(II)}.
\end{aligned}
\label{eq::s::p_h1}
\end{equation}

Using the generalized density on the $s$-dimensional support, 
the first term equals
\begin{equation}
\begin{aligned}
(I) &= \int_{\K_s\bigcap B(x,h)}\frac{1}{h^d}K\left(\frac{\|x-y\|}{h}\right)dP(y)\\ 
&= \int_{\K_s\bigcap B(x,h)}\frac{1}{h^d}K\left(\frac{\|x-y\|}{h}\right)\rho(y)dy,
\end{aligned}
\label{eq::s::p_h2}
\end{equation}
where $dy$ is integrating with respect to $s$-dimensional areas (strictly speaking, it should be written
as a differential $s$-form but here we write $dy$ for simplicity).
Because we assume that $\rho(x)$ is at least three-times bounded differentiable, 
for any $y\in\K_s$ that is close to $x$, we have
$$
\rho(y) = \rho(x) + (y-x)^Tg_s(x) + (y-x)^T H_s(x) (y-x) + o(\|x-y\|^2),
$$
where $g_s$ and $H_s$ are the generalized gradient and Hessian matrix on $\overline{\K_s}$, respectively (see Section~2.3 of the main paper).
By substituting this into equation \eqref{eq::s::p_h2} and using the fact that $K(\|x\|)$ is symmetric and $\overline{\K_s}$ is an $s$-dimensional manifold,
we obtain
\begin{align*}
(I) &=\int_{\K_s\bigcap B(x,h)}\frac{1}{h^d}K\left(\frac{\|x-y\|}{h}\right)\rho(y)dy \\
&= \int_{x+uh\in \K_s\bigcap B(x,h)}\frac{1}{h^d}K\left(\|u\|\right)\rho(x)du\cdot h^s\\ 
&\quad + \int_{x+uh\in \K_s\bigcap B(x,h)}\frac{1}{h^d}K\left(\|u\|\right)u^T H_s(x)u du\cdot h^{s+2} 
+ o\left(\frac{h^{s+2}}{h^d}\right)\\
& = \rho(x)\cdot \int_{x+uh\in \K_s\bigcap B(x,h)}\frac{1}{h^d}K\left(\|u\|\right)du\cdot h^s + O\left(\frac{h^{s+2}}{h^d}\right).
\end{align*}
Note that again in the above integration, $du$ is integrating with respect to $s$-dimensional area (again,
it is actually a differential $s$-form).

Because $\K_s$ has positive reach, $\K_s$ around $x$ changes smoothly (if the reach is $\chi$, we can move
a ball with radius $\chi$ smoothly over every point in $K_s$). This further implies that when $\epsilon\rightarrow0$, 
the ($s$-dimensional) area of $\K_s\bigcap B(x,\epsilon)$ and the area of $B_s(0,\epsilon)$ 
are similar in the sense that 
$$
\frac{{\sf Vol}_s(\K_s\bigcap B(x,\epsilon))}{{\sf Vol}_s(B_s(0,\epsilon))} = 1+O(\epsilon^2),
$$
where ${\sf Vol}_s(A)$ is the $s$-dimensional volume of the set $A$.
Thus, we have
$$
\int_{x+uh\in \K_s\bigcap B(x,h)}K\left(\|u\|\right)du = \int_{B_s(0,1)}K\left(\|u\|\right)du(1+ O(h^2))
 = \frac{1}{C^\dagger_s}(1+O(h^2)).
$$
Using this and the fact that $\tau(x) = s$, the quantity (I) equals
$$
(I) = \frac{1}{C^\dagger_s} h^{s-d}\rho(x) + O(h^{s-d+2}) = \frac{1}{C^\dagger_{\tau(x)}} h^{\tau(x)-d}\rho(x) + O(h^{\tau(x)-d+2}).
$$

Now we bound the second quantity (II).
Because the set $\K_\ell\bigcap B(x,h)$ has $\ell$-dimensional volume $O(h^\ell)$,
$$
\int_{\K_\ell\bigcap B(x,h)}\frac{1}{h^d}K\left(\frac{\|x-y\|}{h}\right)dP(y) \leq \frac{1}{h^d}\int_{\K_\ell\bigcap B(x,h)}K(0)\rho_{\max} dy
= O\left(h^{\ell-d}\right).
$$
In addition, the smallest possible $\ell$ is $\ell= s+m(x)$.
Using this and the bound on (I), we get
\begin{align*}
p_h(x) &= \frac{1}{C^\dagger_s} h^{s-d}\rho(x) + O(h^{s-d+2}) + O(h^{s-d+m(x)})\\ 
&= \frac{1}{C^\dagger_{\tau(x)}} h^{\tau(x)-d}\rho(x) + O(h^{\tau(x)-d+2}) + O(h^{\tau(x)-d+m(x)}).
\end{align*}
Therefore, by multiplying both sides by $C^\dagger_s h^{d-\tau(x)}$, we obtain
$$
C^\dagger_{\tau(x)} h^{d-\tau(x)}\cdot p_h(x) = \rho(x) + O(h^2) + O(h^{m(x)}),
$$
which proves the first assertion.
Note that if $m(x)=0$, we will not have the second term (II), so there is no dimensional bias $O(h^{m(x)})$.

{\bf Part 2: Failure of uniform convergence of the bias.}
Without loss of generality, let $(s,\ell)$ be the two lower dimensional supports such that
$\overline{\K_\ell} \bigcap \K_s\neq \phi$ and $s<\ell$.

Let $x\in \overline{\K_\ell} \bigcap\K_s$ be a point on $\K_s$.
Then by the first assertion, we have
$$
C^\dagger_{s} h^{d-s}\cdot p_h(x) = \rho(x) + O(h^2) + O(h^{m(x)}).
$$

Now, consider a sequence of points where $h\rightarrow0$: $\{x_h\in \K_\ell: \|x_h-x\|\leq h^2\}$.
We can always find such a sequence because $\overline{\K_\ell} \bigcap\K_s\neq \phi$.
For such a sequence, the set $B(x_h,h)\bigcap K_s$ converges to the set $B(x,h)\bigcap K_s$
in the sense that 
$$
\frac{P\left((B(x_h,h)\bigcap K_s)\triangle (B(x,h)\bigcap K_s)\right)}{P\left(B(x,h)\bigcap K_s\right)} \rightarrow 0
$$
when $h\rightarrow0$.
This is because the distance between the centers of the two balls shrinks at rate $h^2$ but the radius of the ball shrinks at rate $h$.

Thus, $\frac{p_h(x_h)}{p_h(x)}\rightarrow 1$ when $h\rightarrow0$.
This implies that $C^\dagger_{s} h^{d-s} \cdot p_h(x_h)\rightarrow \rho(x)$ so
$$
C^\dagger_{\tau(x_h)} h^{d-\tau(x_h)}  \cdot p_h(x_h) = C^\dagger_{\ell} h^{d-\ell}  \cdot p_h(x_h) =  \frac{C^\dagger_{\ell}}{C^\dagger_{s}} h^{s-\ell}\cdot 
\underbrace{C^\dagger_{s} h^{d-s} \cdot p_h(x_h)}_{\rightarrow \rho(x)}
$$
diverges (since $\ell>s$ so $h^{d-\ell}$ diverges).
Thus, we do not have uniform convergence for the bias.

\end{proof}

Before we proceed to the proof of Theorem~\ref{thm::KDE}, 
we first derive a useful lemma about the variance of the KDE.

\begin{lem}[Pointwise variance]
Assume (S, P2, K1--2).
Then for $x\in \K(h)$,
\begin{align*}
{\sf Var}(\hat{p}_n(x)) &= O\left(\frac{1}{nh^{2d-\tau(x)}}\right),\\
{\sf Var}(C_{\tau(x)}h^{d-\tau(x)}\cdot \hat{p}_n(x))&= O\left(\frac{1}{nh^{\tau(x)}}\right).
\end{align*}
\label{lem::variance}
\end{lem}
\begin{proof}
Without loss of generality, let $x\in\K_s(h)$.
This implies that $B(x,h)\cap \K_\ell =\phi$ for all $\ell<s$.
By definition,
\begin{align*}
{\sf Var}(\hat{p}_n(x)) &= \mathbb{E}(\hat{p}_n(x)-\mathbb{E}(\hat{p}_n(x)))^2\\
&=\mathbb{E}\left(\frac{1}{nh^d}\sum_{i=1}^n\left(K\left(\frac{\|X_i-x\|}{h}\right) - \mathbb{E}\left(K\left(\frac{\|X_i-x\|}{h}\right)\right)\right)\right)^2\\
& = \frac{1}{n^2h^{2d}}\mathbb{E}\left(\sum_{i=1}^n\left(K^2\left(\frac{\|X_i-x\|}{h}\right) - n\left(\mathbb{E}^2\left(K\left(\frac{\|X_i-x\|}{h}\right)\right)\right)\right)\right)\\
&\leq \frac{1}{nh^{2d}} \mathbb{E}\left(K^2\left(\frac{\|X_1-x\|}{h}\right) \right)\\
& = \frac{1}{nh^{2d}} \int_{B(x,h)}K^2\left(\frac{\|y-x\|}{h}\right) dP(y)\\
& = \frac{1}{nh^{2d}}\sum_{\ell\leq s} \int_{B(x,h)\cap\K_\ell}K^2\left(\frac{\|y-x\|}{h}\right) \rho(y)dy\\
& \leq \frac{1}{nh^{2d}}\sum_{\ell\leq s} \int_{x+uh\in B(x,h)\cap\K_\ell}K^2(\|u\|)\rho_{\max} du \cdot h^{\ell}\\
& = O\left(\frac{h^s}{nh^{2d}}\right) = O\left(\frac{1}{nh^{2d-s}}\right).
\end{align*}
Note that $dy$ is integrating with respect to $s$-dimensional area and
in the last inequality, we use the transform $y=x+uh$.
Because $x+uh$ has to be on $\K_\ell$,
the change of variable gives $dy = du \cdot h^s$ (the reason why we have $h^s$ is because
both $dy$ and $du$ are differential $s$-form).
The above derivation is for the case $x\in \K_s(h)$ so $\tau(x)=s$.
This proves the first assertion and
the second assertion follows trivially from the first.

\end{proof}

\begin{proof}[Proof of Theorem~\ref{thm::KDE}]
Without loss of generality, we pick a point $x\in\K_s(h)$.
The difference
has the following decomposition:
\begin{equation}
C^{\dagger}_s h^{d-s} \cdot \hat{p}_n(x) - \rho(x) = C^{\dagger}_s h^{d-s} \cdot (\hat{p}_n(x) - p_h(x)) +C^{\dagger}_s h^{d-s} \cdot p_h(x)- \rho(x).
\end{equation}
The former part is the stochastic variation and the latter part is the bias.
The bias is uniformly controlled by Lemma~\ref{lem::smooth::bias}:
$$
C^{\dagger}_s h^{d-s} \cdot p_h(x)- \rho(x) = O(h^{2\bigwedge m_{\min}}).
$$
Thus, all we need is to control the stochastic variation. 

It is well known that the quantity $|\hat{p}_n(x)-p_h(x)|$ 
can be written as an empirical process \citep{Einmahl2005,Gine2002}
and to control the supremum of the empirical process $\sup_{x\in\K_s} C_{s}h^{d-s}\cdot|\hat{p}_n(x)-p_h(x)|$, 
we need to uniformly bound the variance.
By Lemma~\ref{lem::variance}, 
the variance is uniformly bounded at rate $O\left(\frac{1}{nh^{s}}\right)$.
Therefore, by the assumption (K2) and applying Theorem 2.3 in \cite{Gine2002}, 
we have 
$$
\sup_{x\in\K_s(h)}\left|C^{\dagger}_s h^{d-s} \cdot  (\hat{p}_n(x)-p_h(x))\right| =O_P\left(\sqrt{\frac{\log n}{nh^s}}\right),
$$
which together with the bias term proves the desired result for density estimation.

The case of the gradient and the Hessian can be proved in 
a similar way as the density estimation case so we ignore the proof.
The only difference is that the stochastic part has variance $\frac{1}{nh^{s+2}}$ and $\frac{1}{nh^{s+4}}$. 
The extra $+2$  and $+4$
in the power of $h$ come from taking the derivatives.

\end{proof}

\begin{proof}[Proof of Theorem~\ref{thm::alpha::point}]
Recall that
$$
\hat{\alpha}_n(x) = 1-\hat{P}_n(\{y: \hat{p}_n(y)\geq \hat{p}_n(x)\}).
$$
Now we consider a modified version 
$$
\tilde{\alpha}_n(x) = 1-P(\{y: \hat{p}_n(y)\geq \hat{p}_n(x)\})= 1-P(\hat{\Omega}_n(x)),
$$
where $\hat{\Omega}_n(x) = \{y: \hat{p}_n(y)\geq \hat{p}_n(x)\} = \hat{L}_{\hat{p}_n(x)}$.
The idea of the proof is to bound $|\hat{\alpha}_n(x)-\tilde{\alpha}_n(x)|$
first and then bound $|\tilde{\alpha}_n(x)-\alpha(x)|$.

{\bf Part 1: Bounding $|\hat{\alpha}_n(x)-\tilde{\alpha}_n(x)|$.}
This part can be bounded by 
a similar derivation as part (A) in the proof of Theorem \ref{thm::alpha_non}.
Here we just highlight the difference.
First, instead of using 
$L_{\hat p_n(x)+\delta_\infty}\subset \hat{\Omega}_n(x)=\hat{L}_{\hat{p}_n(x)}\subset L_{\hat p_n(x)-\delta_{\infty}}$,
we use
\begin{equation}
\mathbb{A}_{\hat \alpha_n(x)+\delta_{n,h,s}}\subset \hat{\Omega}_n(x)=\hat{L}_{\hat{p}_n(x)}\subset \mathbb{A}_{\hat\alpha_n(x)-\delta_{n,h,s}}\cup \K^C(h).
\label{eq::inc}
\end{equation}
The first inclusion ($\mathbb{A}_{\hat \alpha_n(x)+\delta_{n,h,s}}\subset \hat{L}_{\hat{p}_n}(x)$) in equation \eqref{eq::inc} can be derived as follows.
For any $y\in \mathbb{A}_{\hat \alpha_n(x)+\delta_{n,h,s}}$, by its definition $\alpha(y)\geq \hat\alpha_n(x)+\delta_{n,h,s}$. 
Thus, 
$$
\hat{\alpha}_n(y) \geq \alpha(y) - \delta_{n,h,s} \geq \hat{\alpha}_n(x).
$$
Thus, $y\in \hat{\mathbb{A}}_{\hat{\alpha}_n(x)} = \hat{L}_{\hat{p}_n(x)}$, which proves the first inclusion.
To prove 
the second inclusion ($\hat{L}_{\hat{p}_n}\subset \mathbb{A}_{\hat\alpha_n(x)-\delta_{n,h,s}}\cup \K^C(h)$),
note that if $y\in \hat{L}_{\hat{p}_n(x)}\cap \K^C(h)$, then this is automatically true so 
we focus on $y\in \hat{L}_{\hat{p}_n(x)}\backslash \K(h)$ (namely, $y\in \hat{L}_{\hat{p}_n}$ and is in good regions).
$y\in \hat{L}_{\hat{p}_n(x)}$ implies $\hat{p}_n(y)\geq \hat{p}_n(x)$, which further implies $\hat{\alpha}_n(y)\geq \hat{\alpha}_n(x)$.
Moreover, because $y\in \K^C(h) $ (no contribution from a lower dimension manifold), Theorem~\ref{thm::alpha::point} implies
$$
\alpha(y)\underbrace{\geq}_{y\in \K^C(h)} \hat{\alpha}_n(y)-\delta_{n,h,s}\underbrace{\geq}_{y\in \hat{L}_{\hat{p}_n(x)}} \hat{\alpha}_n(x)-\delta_{n,h,s},
$$
which further implies $y\in \mathbb{A}_{\hat{\alpha}_n(x)-\delta_{n,h,s}}$.
This proves the second inclusion and equation \eqref{eq::inc}.

Because the set $\mathcal{A}_1 =\{\mathbb{A}_{\varpi}:\varpi\in[0,1]\} $ and $\mathcal{A}_2 =\{\mathbb{A}_{\varpi}\cup \K^C(h):\varpi\in[0,1]\} $
both have VC dimension $1$, again we apply the VC theory (see, e.g., Theorem 2.43 in \citealt{wasserman2006all}) to bound 
\begin{align*}
|\hat{P}_n(\mathbb{A}_{\hat\alpha_n(x)+\delta_{n,h,s}})-P(\mathbb{A}_{\hat\alpha_n(x)+\delta_{n,h,s}})| &= O_P\left(\sqrt{\frac{\log n}{n}}\right)\\
|\hat{P}_n(\mathbb{A}_{\hat\alpha_n(x)-\delta_{n,h,s}}\cup \K^C(h))-P(\mathbb{A}_{\hat\alpha_n(x)-\delta_{n,h,s}}\cup \K^C(h))| &= O_P\left(\sqrt{\frac{\log n}{n}}\right).
\end{align*}
Finally, Lemma \ref{lem::bad} implies $P(\K^C(h)) = O(h^{2\bigwedge m_{\min}})$ and thus,
\begin{align*}
P(\mathbb{A}_{\hat\alpha_n(x)-\delta_{n,h,s}}&\cup \K^C(h)) - P(\mathbb{A}_{\hat\alpha_n(x)+\delta_{n,h,s}}) \\
&\leq P(\mathbb{A}_{\hat\alpha_n(x)-\delta_{n,h,s}}) - P(\mathbb{A}_{\hat\alpha_n(x)+\delta_{n,h,s}})  +P(\K^C(h))\\
&=P(\mathbb{A}_{\hat\alpha_n(x)-\delta_{n,h,s}}) - P(\mathbb{A}_{\hat\alpha_n(x)+\delta_{n,h,s}})  +O(h^{2\bigwedge m_{\min}}).
\end{align*}
By definition of the $\alpha$-level set $\mathbb{A}_{\varpi} = \{y: \alpha(y)\geq \varpi\}$, $P(\mathbb{A}_{\varpi}) = 1-\varpi$ (except for the case that 
$x\in\K_0$ is located at a point mass; but this will be a trivial case that the result holds).
Therefore,
$$
P(\mathbb{A}_{\hat\alpha_n(x)-\delta_{n,h,s}}) - P(\mathbb{A}_{\hat\alpha_n(x)+\delta_{n,h,s}}) = 2 \delta_{n,h,s}.
$$
Putting all these elements into the part (A) of the proof of Theorem \ref{thm::alpha_non}, 
we conclude that
$$
|\hat{P}_n(\hat{\Omega}_n(x))-P(\hat{\Omega}_n(x))| = 2 \delta_{n,h,s} + O_P\left(\sqrt{\frac{\log n}{n}}\right)+O(h^{2\bigwedge m_{\min}})
=O(\delta_{n,h,s}).
$$


{\bf Part 2: Bounding $|\tilde{\alpha}_n(x)-\alpha(x)|$.}
By definition of $\alpha(x)$, 
\begin{equation}
\begin{aligned}
1&-\alpha(x)\\ 
&= P(\{y: \alpha(y)\geq \alpha(x)\})\\ 
&= P\left(\{y: \tau(y)< \tau(x)\}\bigcup\{y: \tau(y)=\tau(x),\,\, \rho(y)\geq \rho(x)\}\right)\\
& = P\left(\{y: \tau(y)< \tau(x)\}\right) + P\left(\{y: \tau(y)=\tau(x),\,\, \rho(y)\geq \rho(x)\}\right)\\
& = P(\Omega(x)) + P(D(x)),
\end{aligned}
\end{equation}
where $\Omega(x) = \{y: \tau(y)< \tau(x)\}$ and $D(x)=\{y: \tau(y)=\tau(x),\,\, \rho(y)\geq \rho(x)\}$.
We define $E(x) = \{y: \tau(y)=\tau(x),\,\, \rho(y)< \rho(x)\}$ and $\Phi(x) = \{y: \tau(y)> \tau(x)\}$,
and then $\Omega(x), D(x), E(x), \Phi(x)$ form a partition of $\K$.

Using the fact that $\Omega(x), D(x), E(x), \Phi(x)$ is a partition of $\K$, we bound the difference
\begin{equation}
\begin{aligned}
|\alpha(x) - \tilde{\alpha}_n(x)| &= P\left(\hat{\Omega}_n(x)\triangle (\Omega(x)\cup D(x))\right)\\
& = P\left((\Omega(x)\cup D(x))\backslash \hat{\Omega}_n(x)\right) + P\left(\hat{\Omega}_n(x)\backslash (\Omega(x)\cup D(x))\right)\\
& = P\left(\Omega(x)\backslash \hat{\Omega}_n(x)\right) \\
&\quad+ P\left(D(x)\backslash \hat{\Omega}_n(x)\right) + P\left(\hat{\Omega}_n(x)\cap (E(x)\cup \Phi(x))\right)\\
& = P\left(\Omega(x)\backslash \hat{\Omega}_n(x)\right) + P\left(D(x)\backslash \hat{\Omega}_n(x)\right)\\ 
&\quad+ P\left(\hat{\Omega}_n(x)\cap E(x)\right)+ P\left(\hat{\Omega}_n(x)\cap \Phi(x)\right)\\
& \leq \underbrace{P\left(\left(\Omega(x)\backslash \hat{\Omega}_n(x)\right)\cap \K(h)\right)}_\text{(I)} 
+\underbrace{P\left(\left(D(x)\backslash \hat{\Omega}_n(x)\right)\cap \K(h)\right)}_\text{(II)}\\
&+\underbrace{P\left(\hat{\Omega}_n(x)\cap E(x)\cap \K(h)\right)}_\text{(III)}
+ \underbrace{P\left(\hat{\Omega}_n(x)\cap \Phi(x)\cap \K(h)\right)}_\text{(IV)} + \underbrace{P(\K^C(h))}_\text{(V)}.
\end{aligned}
\label{eq::AP::0}
\end{equation}

Our approach is to first control (I) and (IV) and then control (II) and (III).
Note that Lemma~\ref{lem::bad} controls the quantity (V):
\begin{equation}
(V) \leq O(h^{2\bigwedge m_{\min}}).
\label{eq::AP::V}
\end{equation}

{\bf Bounding (I) and (IV).}
Without loss of generality, we consider $x\in\K_s(h)$. 
Given a point $x\in \K_s(h)$, 
the contribution (I) is from a lower dimensional support, say $y\in\K_\ell(h)$ with $\ell<s$, 
such that $\alpha(y)>\alpha(x)$ but $\hat{p}_n(x)<\hat{p}_n(y)$. 
Namely, we wrongly estimate their ordering. 
The contribution (IV) is in a similar manner but it is for points in a higher dimensional support.
The idea of bound these two components is that 
we will show that when $h\rightarrow0, \frac{\log n}{nh^{d+2}}\rightarrow0$, 
these two components do not contribute at all with an overwhelming probability (a probability $>1-A_1 e^{-A_2\cdot nh^d}\rightarrow 1$ for some $A_1,A_2>0$).



For any point inside $\K_s(h)$, the quantity
\begin{equation}
\delta_{n,h,s}=\sup_{x\in\K_s(h)}|C^\dagger_s h^{d-s}\cdot\hat{p}_n(x)  - \rho(x)| = O(h^{2\bigwedge m_{\min}}) + O_P\left(\sqrt{\frac{\log n}{nh^s}}\right).
\label{eq::AP::2}
\end{equation}
Therefore,
$$
\hat{p}_n(x) \leq \frac{\rho(x)+\delta_{n,h,s}}{C^\dagger_s h^{d-s}}
$$
uniformly for every $x\in\K_s(h)$.

Similarly, 
for any $y\in \K_\ell(h)$ with $\ell<s$,
we obtain
$$
\hat{p}_n(y) \geq \frac{\rho(y)-\delta_{n,h,\ell}}{C^\dagger_\ell h^{d-\ell}}.
$$
Therefore, 
\begin{equation}
\sup_{y\in\K_\ell(h)}\sup_{x\in\K_s(h)}\frac{\hat{p}_n(x)}{\hat{p}_n(y)} \leq h^{s-\ell} \cdot\frac{C^\dagger_\ell}{C^\dagger_s}\cdot\frac{\rho_{\max}+\delta_{n,h,s}}{\rho_{\min}-\delta_{n,h,\ell}},
\label{eq::AP::2-2}
\end{equation}
where $\rho_{\max} $ and $\rho_{\min}$ is the maximum and minimum Hausdorff density (across various support) from assumption (P2).
Note that if the right-hand-side \eqref{eq::AP::2-2} is less than $1$,
then $\hat{p}_n(y)>\hat{p}_n(x)$ for every $y\in\K_\ell(h)$ and $x\in\K_s(h)$. 
Namely, we will be able to separate points in different dimensional supports
if these points are in the good region $\K(h)$ and equation \eqref{eq::AP::2-2} is less than $1$.

Thus, when 
\begin{equation}
\max_{\ell<s}h^{s-\ell} \cdot\frac{C^\dagger_\ell}{C^\dagger_s}\cdot\frac{\rho_{\max}+\delta_{n,h,s}}{\rho_{\min}-\delta_{n,h,\ell}}<1,
\label{eq::AP::2-3}
\end{equation}
we have 
$$
\hat{p}_n(y)>\hat{p}_n(x),\quad \forall y\in\K_\ell(h),x\in\K_s(h),\ell<s.
$$
Because 
$\Omega(x)\cap \K(h)$ is the regions within $\K(h)$
that is inside a lower dimensional support compared to the point $x$, 
equation \eqref{eq::AP::2-3} implies
$$
\left(\Omega(x)\backslash \hat{\Omega}_n(x)\right)\cap \K(h)= \emptyset
$$
and $P\left(\left(\Omega(x)\backslash \hat{\Omega}_n(x)\right)\cap \K(h)\right) = 0$.
A good news is that when $h$ is sufficiently small, equation \eqref{eq::AP::2-3}
holds whenever 
$$
\max_{\ell\leq s}\delta_{n,h,\ell} <\frac{1}{2}\rho_{\min}
$$
(note that the upper bound $\frac{1}{2}\rho_{\min}$ is just a convenient choice).
Since $\delta_{n,h,\ell}$ is just the supremum deviation of the KDE at the $\ell$-dimensional support,
by Assumption (K2) and Talagrand's inequality \citep{Gine2002,Einmahl2005}, when $h\rightarrow 0$,
$$
P(\delta_{n,h,\ell}<\rho_{\min})\geq1-A_1 e^{-A_2\cdot nh^\ell}
$$
for some constants $A_1,A_2>0$ (possibly depending on $\ell$). 
Thus, 
$$
P\left(\max_{\ell\leq s}\delta_{n,h,\ell} <\frac{1}{2}\rho_{\min}\right) \geq 1-A_3 e^{-A_4\cdot nh^s}
$$
for some constants $A_3,A_4$ (possibly depending on $s$). 
Therefore, when $nh^{s}\rightarrow \infty$ (our assumption on $h$ implies this), 
$$
P\left(\left(\Omega(x)\backslash \hat{\Omega}_n(x)\right)\cap \K(h)\right) = 0
$$
with a probability greater than or equal to $1-A_3 e^{-A_4\cdot nh^s}\rightarrow 1$.
Namely, quantity (I)$=0$ with an overwhelming probability.
Thus, we can ignore the contribution from (I). 

The similar analysis also holds for the case of (IV) since (IV) is also a contribution from wrongly estimating
the probability from a different dimensional support. 
Note that in this case, (IV)$=0$ when $\max_{s\leq \ell}\delta_{n,h,\ell} <\frac{1}{2}\rho_{\min}$, which has a probability
$$
P\left(\max_{s\leq \ell}\delta_{n,h,\ell} <\frac{1}{2}\rho_{\min}\right) \geq 1-A_5 e^{-A_6\cdot nh^d},
$$
for some constant $A_5,A_6>0$. 
Thus, the requirement on $h$ also implies that such a probability bound converges toward $1$ 
extremely fast. 
Therefore, the contribution from both (I) and (IV) are 0 with an overwhelming probability so we can ignore them
(their contribution will be of the order $O_P(1-A_5 e^{-A_6\cdot nh^d})$, which is much smaller than the other components).


 


{\bf Bounding (II) and (III).}
We consider cases of (II) and (III) together.
Again, we consider $x\in\K_s$.
Recall that bounds (II) and (III) are the probabilities within the regions
$$
\left(D(x)\backslash \hat{\Omega}_n(x)\right)\cap \K(h),\quad\hat{\Omega}_n(x)\cap E(x)\cap \K(h).
$$
Now define the region $\Psi(x;h) = \{y: \tau(y)=\tau(x)\}\cap \K(h)$.
Thus, we have
\begin{equation}
\begin{aligned}
(II)+(III) &\leq P\left(\left(\hat{\Omega}_n(x)\triangle \Omega(x)\right)\cap \Psi(x;h)\right)\\
& = P\left((\hat{\Omega}_n(x)\cap \Psi(x;h)) \triangle (\Omega(x)\cap \Psi(x;h))\right).
\end{aligned}
\end{equation}
The event 
\begin{align*}
\hat{\Omega}_n(x)\cap \Psi(x;h) &= \left\{y\in\K(h): \hat{p}_n(y)\geq \hat{p}_n(x), \tau(y)=\tau(x), x\in \K(h)\right\}\\
& = \big\{y\in\K(h): C^\dagger_sh^{d-s}\cdot\hat{p}_n(y)\geq C^\dagger_sh^{d-s}\cdot\hat{p}_n(x),\\ 
&\qquad \qquad\qquad \qquad \qquad \quad\tau(y)=\tau(x)=s, x\in \K(h)\big\},
\end{align*}
which can be viewed as the estimated density upper level set at level $\hat{p}_n(x)$ of the support $\K_{s}(h)$ (because $\tau(x)=s$).
Similarly, $\Omega(x)\cap \Psi(x;h)$ is just the upper level set at level $\rho(x)$ of the support $\K_s(h)$.
Therefore, the difference can be bounded by Theorem~\ref{thm::alpha_non}:
$$
(II)+(III) \leq 
\begin{cases}
O\left(\delta_{n,h,s}\right), \quad&\mbox{if}\quad \inf_{c\in \mathcal{C}_s}|p(x)-p(c)|>r_{n,h,s}\\
O\left((\delta_{n,h,s})^{\frac{s}{s+1}}\right), \quad&\mbox{otherwise}
\end{cases},
$$
where $r_{n,h,s}$ is a deterministic quantity such that $\delta_{n,h,s} = o_P(r_{n,h,s})$.

Thus, putting the above bound and equation \eqref{eq::AP::V} 
into equation \eqref{eq::AP::0}, we have
\begin{align*}
|\alpha(x) - \tilde{\alpha}_n(x)| &\leq (I)+(II) +(III) + (IV) +(V)\\
&\leq 
O(h^{2\bigwedge m_{\min}}) \\
&\quad+ 
\begin{cases}
O\left(\delta_{n,h,s}\right), \quad&\mbox{if}\quad \inf_{c\in \mathcal{C}_s}|p(x)-p(c)|>r_{n,h,s}\\
O\left((\delta_{n,h,s})^{\frac{s}{s+1}}\right), \quad&\mbox{otherwise}
\end{cases}\\
& = \begin{cases}
O\left(\delta_{n,h,s}\right), \quad&\mbox{if}\quad \inf_{c\in \mathcal{C}_s}|p(x)-p(c)|>r_{n,h,s}\\
O\left((\delta_{n,h,s})^{\frac{s}{s+1}}\right), \quad&\mbox{otherwise}
\end{cases},
\end{align*}
which, along with the fact that $|\hat{\alpha}_n(x)-\tilde{\alpha}_n(x)| = O\left(\delta_{n,h,s}\right)$, proves the desired result.

\end{proof}

Before we move on to the proof of Theorem~\ref{thm::alpha::Pro},
we first give a lemma that quantifies the size of good regions.

\begin{lem}[Size of good region]
Assume (S, P2, K1--2).
Define $m_{\min}$ from equation (9) of the main paper 
and $$
m^*_{\min} = d - \max\{s<d: \K_s \cap \overline{\K_d}\neq \emptyset\}.
$$
Let $\mu$ be the Lebesgue measure.
Then 
\begin{align*}
\mu(\K^C(h)) &= O(h^{ m^*_{\min}}),\\
P(\K^C(h)) &= O(h^{ m_{\min}}).
\end{align*}
\label{lem::bad}
\end{lem}
\begin{proof}

{\bf Case of the Lebesgue measure.}
By assumption (S), all supports $\K_s$ have Lebesgue measure $\mu(\K_s) = 0$ except $\K_d$.
Thus, 
$$
\mu(\K^C(h)) = \mu(\K_d^C(h)).
$$

By the definition of $m^*_{\min},$
the quantity $d-m^*_{\min} = \max\{s<d:  \K_s \cap \overline{\K_d}\neq \emptyset\}$ 
denotes the support with the largest dimension that intersects the closure of $\K_d$.

Note that for any two compact sets $A,B$ with dimensions $d(a)$ and $d(b)$ such that 
$d(b)>d(a)$ and $\overline{A}\cap\overline{B} \neq \phi$, 
then the $d(b)$-dimensional Lebesgue measure 
\begin{equation}
B\cap(A\oplus r) = O(r^{d(b)-d(a)})
\label{eq::leb}
\end{equation}
when $r\rightarrow 0$.
Thus, the set $\mu(\K_d \cap (\K_{d-m^*_{\min}}\oplus h))$ shrinks at rate $O(h^{m^*_{\min}})$.

Recall that $\K_d(h)= \K_d\backslash(\bigcup_{\ell<d}\K_\ell\oplus h)$
and $\K^C_d(h) = (\bigcup_{\ell<d}\K_\ell\oplus h)$. 
Because only sets in $\K_d$ have nonzero Lebesgue measure,
we have 
\begin{align*}
\mu(\K_d^C(h)) &= \mu(\K_d^C(h)\cap \K_d)\\
& = \mu\left(\K_d \cap (\bigcup_{\ell<d}\K_\ell\oplus h)\right)\\
& = \mu\left(\K_d \cap (\bigcup_{\ell\leq d-m^*_{\min}}\K_\ell\oplus h)\right)\\
& \leq \sum_{\ell=0}^{d-m^*_{\min}} \mu\left(\K_d \cap (\K_\ell\oplus h)\right)\\
& = O(h^{m^*_{\min}}) +o(h^{m^*_{\min}}),
\end{align*}
which proves the first assertion.


{\bf Case of the probability measure.}
The case of the probability measure is very similar to that of Lebesgue measure,
but now we also need to consider lower-dimensional supports because of the singular probability measure.

First, we expand the probability of bad regions by the following:
\begin{equation}
\begin{aligned}
P(\K^C(h)) &= 1- P(\K(h))\\
& = 1- P(\bigcup_{\ell\leq d}\K_\ell(h))\\
&=  1- \sum_{\ell=0}^d P(\K_\ell(h))\\
& = 1- \sum_{\ell=0}^d P\left(\K_\ell\backslash \left(\bigcup_{s<\ell}\K_s \oplus h\right)\right)\\
& = \sum_{\ell=0}^dP(\K_\ell)-  P\left(\K_\ell\backslash \left(\bigcup_{s<\ell}\K_s \oplus h\right)\right)\\
& = \sum_{\ell=0}^d P\left(\K_\ell\bigcap \left(\bigcup_{s<\ell}\K_s \oplus h\right)\right).
\end{aligned}
\label{eq::P::0}
\end{equation}
Note that we use the fact that $P(A)-P(A\backslash B) = P(A\cap B)$ in the last equality.
We will show that 
\begin{equation}
P\left(\K_\ell\bigcap \left(\bigcup_{s<\ell}\K_s \oplus h\right)\right) = O(h^{m_{\min}}).
\label{eq::P::1}
\end{equation}
for every $\ell$.
For simplicity, we define $\K^C_\ell(h) = \K_\ell\bigcap \left(\bigcup_{s<\ell}\K_s \oplus h\right)$.


Without loss of generality, we consider the support $\K_\ell$ and $\K^C_\ell(h)$. 
For simplicity, 
by the definition of $m_{\min}$, the largest lower-dimensional support that intersects the closure $\overline{\K_\ell}$
has a dimension lower than or equal to $\K_{\ell-m_{\min}}$.
Note that if $\ell< m_{\min}$, it is easy to see that there will be no other lower dimensional support intersecting $\overline{\K_\ell}$,
so $\K_\ell(h) = \K_\ell$ and there is nothing to prove.
Thus, the set $\K^C_\ell(h)$ can be rewritten as
$$
\K^C_\ell(h) = \K_\ell\bigcap \left(\bigcup_{s\leq\ell-m_{\min}}\K_s \oplus h\right)
$$
By equation \eqref{eq::leb}, the $\ell$-dimensional Lebesgue measure $\mu_{\ell}$ on the set $\K^C_\ell(h)$
is at rate
$$
\mu_\ell\left(\K^C_\ell(h)\right) = O(h^{m_{\min}}) +o(h^{m_{\min}}).
$$
By assumption (P2), the $\ell$-dimensional Lebesgue measure on $\K_\ell$ implies that bound on the probability measure, so we have
$$
P\left(\K^C_\ell(h)\right) = O(h^{m_{\min}}). 
$$
Because this works for every $\ell$, by equations \eqref{eq::P::0} and \eqref{eq::P::1},
we have
$$
P(\K^C(h)) = O(h^{m_{\min}}),
$$
which proves the result.


\end{proof}

\begin{proof}[Proof of Theorem~\ref{thm::alpha::Pro}]
We first note that it is easy to see that
$m^*_{\min} \geq m_{\min}$, so the rate of the integrated error is bounded by that of the probability error.
Thus, we only prove the case for the probability error here.

Because $\K_0,\cdots, \K_d$ form a partition of $\K$,
we separately analyze the probability error at each $\K_\ell$ and then join them to get the final bound.

For a support $\K_\ell$,
we partition it into three subregions $A,B$, and $C$, where 
\begin{equation}
\begin{aligned}
A &= \K_\ell^C(h) = \K_\ell\backslash \K_\ell(h)\\
B &= \K_\ell(h)\cap \{x: \min_{c\in\mathcal{C}_s}|\rho(x)-\rho(c)|\leq r_{n,h,s}\}\\
C &= \K_\ell(h)\cap \{x: \min_{c\in\mathcal{C}_s}|\rho(x)-\rho(c)|> r_{n,h,s}\},
\end{aligned}
\end{equation}
where $r_{n,h,s} = \frac{h^{2\bigwedge m_{\min}} + \frac{\log n}{nh^{s}}}{\log n}$ 
satisfies the requirement $\frac{\delta_{n,h,s}}{r_{n,h,s}} = o_P(1)$.

{\bf Case A.}
By Lemma~\ref{lem::bad}, $P(A) = O(h^{m_{\min}})$ and $|\hat{\alpha}(x)-\alpha(x)|\leq 1$.
Thus, 
\begin{equation}
\int_{A} |\hat{\alpha}_n(x)-\alpha(x)| dP(x) = O(h^{m_{\min}}).
\label{eq::alpha::A}
\end{equation}

{\bf Case B.}
For set $B$, note that the generalized density $\rho(x)$ behaves quadratically around critical points.
So a difference in density level at rate $\delta$ results in a change in the difference in distance at rate $\sqrt{\delta}$. 
Thus, 
the $\ell$-dimensional Lebesgue measure $\mu_\ell(B) =O(\sqrt{r_{n,h,s}})$,
which by assumption (P2) implies $P(B) = O(\sqrt{r_{n,h,s}})$. 
By Theorem~\ref{thm::alpha::point}, 
$|\hat{\alpha}_n(x) -\alpha(x)| = \delta_{n,h,s}^{\frac{s}{s+1}}$ uniformly for all $x\in B$.
Thus, the error is
\begin{equation}
\int_{B} |\hat{\alpha}_n(x)-\alpha(x)| dP(x) = O\left(\sqrt{r_{n,h,s}}\cdot \delta_{n,h,s}^{\frac{s}{s+1}}\right).
\label{eq::alpha::B}
\end{equation}

{\bf Case C.}
For points in this region, directly applying Theorem~\ref{thm::alpha::point} yields
\begin{equation}
\int_{C} |\hat{\alpha}_n(x)-\alpha(x)| dP(x) = O\left(\delta_{n,h,s}\right).
\label{eq::alpha::C}
\end{equation}

By adding up equations \eqref{eq::alpha::A}, \eqref{eq::alpha::B}, and \eqref{eq::alpha::C}
and using the fact that $\sqrt{r_{n,h,s}}\cdot \delta_{n,h,s}^{\frac{s}{s+1}} = o(\delta_{n,h,s})$ and $O(h^{m_{\min}})$ are
part of $\delta_{n,h,s}$, 
we obtain
\begin{equation}
\int_{\K_\ell} |\hat{\alpha}_n(x)-\alpha(x)| dP(x) = O\left(\delta_{n,h,s}\right).
\end{equation}
This works for every $\K_\ell$, which proves the desired bound.

\end{proof}

\begin{proof}[Proof of Lemma~\ref{lem::DCP::p}]

{\bf First assertion.}
By assumption (B) and (P2), the first assertion is trivially true because when we move down the level $\alpha$,
the only situation that creates a new connected component is when $\alpha$ passes through the $\alpha$-level of a local mode.

{\bf Second assertion.}
For level sets, there are only two situations where a change in topology may occur: creation of a new connected component, and
merging of two (or more) connected components.
By the first assertion, we only need to focus on the merging case.

Assume $\varpi\in\mathcal{A}$ be a level where only merging of connected components occurs.
In addition, recall that $\xi(\varpi)$ is the integer such that $\K_s\subset \mathbb{A}_\varpi$ for all $s\leq \xi(\alpha)$
and $\K_{\xi(\varpi)+1}\not\subset \mathbb{A}_\varpi$.
For a sufficiently small $\epsilon$, the difference between $\mathbb{A}_\varpi$ and $\mathbb{A}_{\varpi+\epsilon}$ 
is in $\K_{\xi(\varpi)+1}$.
Thus, when we move $\epsilon$ down, only the connected components in $\K_{\xi(\alpha)+1}$
expand. 
When merging occurs, there are only two cases:
two connected components in $\K_{\xi(\varpi)+1}$ meet each other
or a connected component in $\K_{\xi(\varpi)+1}$ hits a lower-dimensional support.
Note that there is no higher-dimensional supports in $\mathbb{A}_\varpi$ because their ordering is less than any point in $\K_{\xi(\varpi)}$.
The first case corresponds to a saddle point, which is an element in $\mathcal{C}$. 
The second case corresponds to a DCP.
Thus, $\varpi$ must be either an $\alpha$-level of a critical point or a DCP.

\end{proof}

\begin{proof}[Proof of Lemma~\ref{lem::g_critical}]

{\bf Critical points.}
Without loss of generality, let $c$ be a critical point of $\rho(x)$ on $\K_s$.
Recall that $p_h = \E(\hat{p}_n)$ is the smoothed density function and $g_h = \nabla p_h$
and $H_h = \nabla \nabla p_h$ are the corresponding gradient and Hessian.

The idea of the proof is to argue that when $\delta_{n,h,s}^{(2)}\rightarrow 0$, 
we can always find a critical point $\hat{c}$ of $\hat{p}_n$
and a critical point $c_h$ of $p_h$
such that 
$$
\|\hat{c}-c_h\|=O_P\left(\frac{1}{nh^{s+2}}\right),\quad \|c_h-c\| = O(h).
$$
Then we use the triangle inequality to bound $\|\hat{c}-c\|$.

Before we go to the details, we first define a useful set:
$$
\tilde{\K}_s(h) = \{x: x\in\K_s\oplus h,x\notin \K_\ell\oplus h, \ell<s\} = (\K_s\oplus h) \backslash \bigcup_{\ell<s} (\K_\ell\oplus h).
$$
The set $\tilde{\K}_s$ is the regions around $\K_s$ but away from a lower dimensional support
($\tilde{\K}_s$ is thicker than the `good region' $\overline{\K}_s$; $\overline{\K}_s$ does not include regions around $\K_s$).
Assumption (C) implies that $c\in \tilde{\K}_s(h)$ when $h$ is sufficiently small.
Note that any point $x\in \tilde{\K}_s$ has a distance to $\K_s$ less than $h$, i.e.,
$\sup_{x\in \tilde{\K}_s}d(x,\K_s)\leq h$.
Using the same bias analysis of $p_h$ in Lemma~\ref{lem::smooth::bias},
one can prove that when $h\rightarrow 0$
\begin{equation}
\sup_{x\in\tilde{\K}_s(h)}h^{d-s}\max\left\{p_h(x),\|g_h(x)\|_{\max}, \|H_h(x)\|_{\max}\right\} \leq U_0
\label{eq::lem12::fact1}
\end{equation}
for some fix constant $U_0$.
An intuitive explanation for Equation \eqref{eq::lem12::fact1} is that
within the set $\tilde{\K}_s(h)$, the support $\K_s$ is the dominating structure (lowest dimensional support)
so
the smoothed version of these quantities will be diverging at the rate $O(h^{s-d})$.
Rescaling them by $O(h^{d-s})$ will uniformly bound all of them (this is the same idea as Lemma~\ref{lem::smooth::bias}).

In addition to equation \eqref{eq::lem12::fact1},
there is another useful fact about the region $\tilde{\K}_s(h)$:
when $\frac{nh^{s+4}}{\log n}\rightarrow\infty$ and $h\rightarrow 0$, 
\begin{equation}
\begin{aligned}
\sup_{x\in\tilde{\K}_s(h)}h^{d-s} |\hat{p}_n(x)-p_h(x)| &= O_P\left(\sqrt{\frac{\log n}{nh^{s}}}\right)\\
\sup_{x\in\tilde{\K}_s(h)}h^{d-s} \|\hat{g}_n(x)-g_h(x)\|_{\max} &= O_P\left(\sqrt{\frac{\log n}{nh^{s+2}}}\right)\\
\sup_{x\in\tilde{\K}_s(h)}h^{d-s} \|\hat{H}_n(x)-H_h(x)\|_{\max} &= O_P\left(\sqrt{\frac{\log n}{nh^{s+4}}}\right).
\end{aligned}
\label{eq::lem12::fact2}
\end{equation}
The proof of these quantity follows the same way as the proof of Lemma~\ref{lem::variance} and Theorem~\ref{thm::KDE}.
Moreover, the $\sqrt{\log n}$ factor will disappear in equation \eqref{eq::lem12::fact2}
if the bound is taken for a given point (or a fixed sequence of points).


We first show that $\|c_h-c\| = O(h)$.
When $h\rightarrow 0$, we can always find a critical point $c_h$ of $p_h$
such that $c_h\rightarrow c$ 
(because $p_h$ is essentially a smoothed version of the singular measure).
Note that such a point $c_h$ serves as an `approximation' of $c$ in the function $p_h$.
Moreover, $d(c_h,\K_s) < h$ because regions outside $h$ distance of $\K_s$
will not be affected by $\K_s$.
For any point $x\in\K$, let $\pi_s(x)\in\K_s$ be the projection from $x$ onto $\K_s$.
Because $\K_s$ has positive reach, the projection $\pi_s(c_h)$ will be unique when $h$ 
is less than the reach of $\K_s$.
Note that the fact that $d(c_h,\K_s) < h$ implies $d(c_h,\pi_s(c_h))=d(c_h,\K_s) < h$.
By assumption (C), there is a fixed distance between $c$ and any lower dimensional support
so $\pi_s(c_h)\in\K_s(h)$ when $h$ is sufficiently small (namely, the projected
point will be in the good region).

Because $d(c_h,\pi_s(c_h))<h$ and the (scaled) Hessian $C^\dagger_sh^{d-s}H_h$ is uniformly bounded
within $\tilde{K}_s(h)$ (equation \eqref{eq::lem12::fact1}),
Taylor theory implies
$$
C^\dagger_sh^{d-s}\|\underbrace{ g_h(c_h)}_{=0} - g_h(\pi_s(c_h))\|_{\max}  = O(\|c_h-\pi_s(c_h)\|) = O(h).
$$ 
Note that $g_h(x) = \nabla p_h(x)$ so $\|g_h(x)\|_{\max} \geq\|\nabla_{T_s(x)}p_h(x)\|_{\max}$
(recall that $\nabla_{T_s(x)}$ is the gradient with respect to the tangent space so it has fewer degree of freedom than $\nabla$).
Thus, 
\begin{align*}
C^\dagger_sh^{d-s}\|g_h(\pi_s(c_h))\|_{\max} &=O(h)\\ 
&\geq \|C^\dagger_sh^{d-s}\nabla_{T_s(\pi_s(c_h))}g_h(\pi_s(c_h))\|_{\max}\\
&=\|\nabla_{T_s(\pi_s(c_h))} \rho(\pi_s(c_h))\|_{\max} + O(h^{2\bigwedge m_{\min}}) .
\end{align*}
Note that the last equality is from Theorem~\ref{thm::KDE}.
Because $m_{\min}\geq 1$, the above inequality further implies 
\begin{equation}
\|\nabla_{T_s(\pi_s(c_h))} \rho(\pi_s(c_h))\|_{\max} = O(h).
\label{eq::lem12::1}
\end{equation}
When $h\rightarrow 0$, region on $\K_s$ with a low gradient of $\rho(x)$ must be around a critical point 
since $\rho$ on $\K_s$ is a Morse function (assumption (P2)).
Because the eigenvalues of the Hessian of $\rho(x)$ 
at $c$ is uniformly bounded away from $0$ (an natural outcome of a Morse function),
the gradient changes linearly around $c$.
As a result, equation \eqref{eq::lem12::1} implies $\|\pi_s(c_h)-c\|=O(h)$.
This, together with the fact that $\|\pi_s(c_h)-c_h\| = O(h)$,
further implies that 
$\|c_h-c\| = O(h)$.


To prove that $\|\hat{c}-c_h\|=O_P\left(\sqrt{\frac{1}{nh^{s+2}}}\right)$,
by equation \eqref{eq::lem12::fact2}
the gradient of any point $x\in\tilde{\K}_s(h)$ is
$$
h^{d-s}\left(\hat{g}_n(x) - g_h(x)\right) = O_P\left(\sqrt{\frac{1}{nh^{s+2}}}\right).
$$
Plugging-in $x=c_h$ into the above equality and using equation \eqref{eq::lem12::fact2},
we obtain
\begin{align*}
h^{d-s}\left(\hat{g}_n(c_h) - \underbrace{g_h(c_h)}_{=0}\right) & = O_P\left(\sqrt{\frac{1}{nh^{s+2}}}\right)\\
& = h^{d-s}\left(\hat{g}_n(c_h) - \underbrace{\hat{g}_n(\hat{c})}_{=0}\right)\\
& = h^{d-s} \hat{H}_h(c_h) (c_h-\hat{c}) + O_P(\|c_h-\hat{c}\|^2)\\
& = h^{d-s} \underbrace{\left(H_h(c_h)+O_P\left(\sqrt{\frac{\log n}{nh^{s+4}}}\right)\right)}_{\mbox{by \eqref{eq::lem12::fact2}}} (c_h-\hat{c})  + O_P(\|c_h-\hat{c}\|^2)\\
& = h^{d-s} H_h(c_h) (c_h-\hat{c}) + o_P(\|c_h-\hat{c}\|).
\end{align*}
Since all eigenvalues of $h^{d-s}H_h(c_h)$ are uniformly bounded away from $0$ when $h\rightarrow 0$
(this follows from the fact that $\rho(x)$ is a Morse function and $K^{(2)}(0)<0$), its inverse exists
so $\|c_h-\hat{c}\| = O_P\left(\sqrt{\frac{1}{nh^{s+2}}}\right)$.
Thus, putting it altogether we obtain
$\|\hat{c}-c\| \leq \|c_h-\hat{c}\| + \|c_h-c\| = O(h)+ O_P\left(\sqrt{\frac{1}{nh^{s+2}}}\right)$,
which proves the stability of critical points.

{\bf Critical levels.}
To derive the rate of $|\hat{\alpha}_n(\hat{c})-\alpha(c)|$, 
we decompose it into 
\begin{equation*}
|\hat{\alpha}_n(\hat{c})-\alpha(c)| \leq |\hat{\alpha}_n(\hat{c})-\hat{\alpha}_n(c)| + |\hat{\alpha}_n(c)-\alpha(c)|.
\end{equation*}
Because the KDE around $\hat{c}$ behaves quadratically, i.e., 
for any $\mu\in\R^d$ such that $\|\mu\|$ is small, $\hat{p}_n(\hat{c}+\mu)-\hat{p}_n(\hat{c})= O_P(\|\mu^2\|)$ and thus 
$$
|\hat{\alpha}_n(\hat{c})-\hat{\alpha}_n(c)| = O\left(\|\hat{c}-c\|^2\right)=O(h^2)+O_P\left(\frac{1}{nh^{s+2}}\right),
$$
which is dominated by the second term $|\hat{\alpha}_n(c)-\alpha(c)|= O\left(\delta^{\frac{s}{s+1}}_{n,h,s}\right)$
(by Theorem~\ref{thm::alpha::point}).
Thus, the convergence rate is $|\hat{\alpha}_n(\hat{c})-\alpha(c)| = O\left(\delta^{\frac{s}{s+1}}_{n,h,s}\right)$.

{\bf Eigenvalues.}
For the eigenvalues, 
it is easy to see that when we are moving away from $c$ along a direction that is normal to $\K_s$,
the estimated density is going down.
Thus, these $(d-s)$ directions must have negative eigenvalues (dimension of the normal subspace is $d-s$).
In addition, the original generalized Hessian matrix at $c$ has $n(c)$ negative eigenvalues.
So the total number of negative eigenvalues of the Hessian matrix of $\hat{p}_n$ at $c$ is $n(s)+d-s$.
Because $\hat{c}$ is converging to $c$, the sign of negative eigenvalues also converges, which proves the lemma.

\end{proof}

\begin{proof}[Proof of Lemma~\ref{lem::DCP}]
{\bf First assertion: Location of DCPs.}
By Theorem~\ref{thm::KDE},
the scaled KDE is uniformly consistent in both density estimation and gradient estimation
in the good region $\K(h)$.

The estimated DCPs are points satisfying $\nabla \hat{p}_n(x)= 0$.
Therefore, when $\delta^{(2)}_{n,h,d}\overset{P}{\rightarrow} 0$, 
the only areas in $\K(h)$ such that $\nabla \hat{p}_n(x)=0$ will be regions where $\nabla_{T_{\tau(x)}(x)} \rho(x)$ is small. 
This can only be the regions around the generalized critical points.
As a result, we cannot have any DCPs within $\K(h)$ when $\delta^{(2)}_{n,h,d}\overset{P}{\rightarrow} 0$.

{\bf Second assertion: Number of estimated critical points.}
This follows directly from Lemma~\ref{lem::DCP_level} and Assumption (A) and from the fact that $\delta^{(2)}_{n,h,d}\overset{P}{\rightarrow} 0$.

{\bf Third assertion: No local modes.}
By assumption (B), the only case where the creation of a connected component occurs is a (generalized) local mode.
These population local modes correspond to elements in $\hat{\mathcal{C}}$.
Thus, any estimated local mode in $\hat{\mathcal{D}}$ does not have a population target so they are away from the population local modes
and by first assertion, it has to be in the bad region $\K^C(h)$.
It is easy to see that we cannot have any estimated local mode under such a constraint when
we have the gradient and Hessian consistency of the scaled KDE
($\delta^{(2)}_{n,h,d}\overset{P}{\rightarrow} 0$). 

\end{proof}

\begin{proof}[Proof of Lemma~\ref{lem::DCP_level}]
Let $c$ be a DCP and $\varpi_0(c)$ be the corresponding $\alpha$-level of merging.
By assumption (A), $c$ is a unique point for $\alpha_0(c)\in \mathcal{D}$ and
for any sufficiently small $\epsilon$,
there is a unique connected component $C_\epsilon\in \mathbb{A}_{\varpi_0(c)+\epsilon}$
and a support $\K_\ell$ with $\ell\leq\xi(\varpi_0(c))$ 
such that $x\in\K_\ell$, $c\notin C_\epsilon$ and $d(c,C_\epsilon)\rightarrow 0$ when $\epsilon\rightarrow 0$.

The idea of the proof is to find $\hat{\varpi}^+$ and $\hat{\varpi}^-$ such that
the merging has not yet happened in the set $\mathbb{A}_{\hat{\varpi}^+}$
but the merging has happened in the set $\mathbb{A}_{\hat{\varpi}^-}$.
Then we know the actual value $\hat{\varpi}_n(\hat{c})$ lies within the interval 
$[\mathbb{A}_{\hat{\varpi}^-}, \mathbb{A}_{\hat{\varpi}^+}]$.

{\bf Case: Lower bound.}
By Theorem~\ref{thm::alpha::point}, any point $y\in K(h)$ satisfies $|\hat{\alpha}_n(y)-\alpha(y)| = \delta_{n,h,\tau(y)}$.
Thus, for any $C_\epsilon$ defined as the connected component within $\mathbb{A}_{\varpi_0(x)+\epsilon}$
that is about to merge with $\K_\ell$, 
$$
\inf_{y\in C_\epsilon}\hat{\alpha}_n(y)> \varpi_0(x) - \delta_{n,h,\xi(\varpi_0(c))+1}.
$$
This is because $\inf_{y\in C_\epsilon}\alpha(y)\geq \varpi_0(c)$ by definition and for $C_\epsilon\cap \K(h)$,
we have the uniform bound from Theorem~\ref{thm::alpha::point}. 
For $y\in C_\epsilon\backslash \K(h)$, the estimated $\hat{\alpha}_n(y)$ will be influenced by $\K_\ell$, which is a lower dimensional support.
So the estimated $\alpha$ value will be more than $\varpi_0(c) - \delta_{n,h,\xi(\varpi_0(c))+1}$. 
Thus, we can pick the lower bound $\mathbb{A}_{\hat{a}^-} = \varpi_0(c) - \delta_{n,h,\xi(\varpi_0(c))+1}$.

{\bf Case: Upper bound.}
To prove the upper bound, the idea is very simple. 
We show that when the $\alpha$-level is high enough, $\hat{\mathbb{A}}_{\varpi'}$ will not 
contain the set $\tilde{C}_\ell\oplus h$ where $c\in \tilde{C}_\ell$ and $\tilde{C}_\ell$ is a connected component of $\mathbb{A}_{\varpi_0}$.

Let $\tilde{C}_\ell$ be defined as above. 
Because $\tilde{C}_\ell$ is a subset of $\bigcup_{s\leq \ell}\K_s$,
the set $C_\epsilon \cap (\tilde{C}_\ell\oplus h)$ is always within the bad region $\K^C(h)$. 
Thus, $P(C_\epsilon \cap (\tilde{C}_\ell\oplus h))\leq P(\K^C(h)) = O(m^{2\bigwedge m_{\min}})$ by Lemma~\ref{lem::bad}.

Now we consider the boundary $\partial (\tilde{C}_\ell\oplus h) = \{x: d(x,\tilde{C}_\ell) = h\}$.
Because $P(C_\epsilon \cap (\tilde{C}_\ell\oplus h))\leq O(m^{2\bigwedge m_{\min}})$,
$$
\sup_{x\in \partial (\tilde{C}_\ell\oplus h)} \alpha(x) \leq \varpi_0(c) + O(m^{2\bigwedge m_{\min}}).
$$
Moreover, outside the boundary $\partial (\tilde{C}_\ell\oplus h)$ we can apply Theorem~\ref{thm::alpha::point}
to bound $\hat{\alpha}_n(x)-\alpha(x)$.
Thus, 
$$
\sup_{x\in \partial (\tilde{C}_\ell\oplus h)} \hat{\alpha}_n(x) \leq \varpi_0(c) + O(m^{2\bigwedge m_{\min}}) + \delta_{n,h,\xi(\varpi_0(c))+1}.
$$
This suggests that 
$\mathbb{A}_{\hat{\varpi}^+} = a_0 + O(m^{2\bigwedge m_{\min}}) + \delta_{n,h,\xi(\varpi_0(c))+1} = a_0 + \delta_{n,h,\xi(\varpi_0(c))+1}$
because $O(m^{2\bigwedge m_{\min}})$ is part of the term $\delta_{n,h,\xi(\varpi_0(c))+1}$.

Thus, the quantity $\hat{\alpha}_n(\hat{c})$ must lie within $a_0 \pm \delta_{n,h,\xi(\varpi_0(c))+1}$, which is the desired bound.
Because there is a topological change in the upper level set for $\hat{p}_n$ at such $\hat{\alpha}_n(\hat{c})$, 
$\hat{c}$ must be a critical point of $\hat{p}_n$. This completes the proof.



\end{proof}

\begin{proof}[Proof of Theorem~\ref{thm::top::p}]

The main idea is to show that there exists a constant $a_0>0$ such that
\begin{equation}
\max_{j=0,1,2}\max_{s=0,\cdots, d}|\delta^{(j)}_{n,h,s}|< a_0 \Longrightarrow T_{\hat{\alpha}_n}\overset{T}{\approx} T_\alpha.
\label{eq::top::1}
\end{equation}
Note that $\delta^{(0)}_{n,h,s} = \delta_{n,h,s}$.


By assumption (A), there exists a constant $a_1 = \min\{|\alpha_1-\alpha_2|: \alpha_1,\alpha_2 \in\mathcal{A}, \alpha_1\neq \alpha_2 \}>0$.
Without loss of generality, let the elements in $\mathcal{A}$ be $\alpha_1>\alpha_2>\cdots>\alpha_m$ (assume $\mathcal{A}$ has $m$ elements)
and $c(\alpha_j)$ be the corresponding critical point or DCP for level $\alpha_j$. 
By Lemma~\ref{lem::g_critical} and \ref{lem::DCP_level}, 
we can find a sequence of points $\hat{c}_1,\cdots,\hat{c}_m$ such that
each $\hat{c}_j$ is the estimator to $c(\alpha_j)$.
Again, by Lemma~\ref{lem::g_critical} and \ref{lem::DCP_level}, when 
\begin{equation}
\max_{s=0,\cdots, d}|\delta^{(0)}_{n,h,s}|< \frac{a_1}{2},
\label{eq::top::2}
\end{equation}
we have
$$
\hat{\alpha}_n(\hat{c}_1)>\cdots>\hat{\alpha}_n(\hat{c}_m).
$$
Namely, the ordering will be the same when $\max_{s=0,\cdots, d}|\delta^{(0)}_{n,h,s}|$ is sufficiently small.
Thus, we need to prove that (1) there is no other connected component of $\hat{\alpha}_n$ and
(2) there is no other merging point to get the topological equivalent.

By Lemma~\ref{lem::DCP}, there will be no estimated local mode in $\hat{\mathcal{D}}$ so there will be no
extra connected components. 
In the proof of Lemma~\ref{lem::DCP_level}, we showed that
each estimator of a DCP corresponds to a merging in the estimated level sets,
and similarly, if connected components are merged at a saddle point or a local minimum, 
the corresponding estimator will be a merging point.
To use the conclusion of Lemma~\ref{lem::g_critical}, we need uniform consistency in both gradient and Hessian estimation.
Namely, we need
$\delta^{(1)}_{n,h,s}, \delta^{(2)}_{n,h,s}$ to be sufficiently small.
The gradient consistency regularizes the positions of estimators of generalized critical points and
the bound on the Hessian matrix guarantees that the eigenvalues retain the same sign. 
Thus, to apply Lemma~\ref{lem::g_critical}, there is some $a_2>0$ such that the conclusion of Lemma~\ref{lem::g_critical}
holds whenever 
\begin{equation}
\max_{j=1,2}\max_{s=0,\cdots, d}|\delta^{(j)}_{n,h,s}|<a_2.
\label{eq::top::3}
\end{equation}
To use Lemma~\ref{lem::DCP_level}, we need $\delta^{(1)}_{n,h,s}$ to be sufficiently small.
This comes from Theorem~\ref{thm::alpha::point}. 

Combining equations \eqref{eq::top::2} and \eqref{eq::top::3}, we obtain equation \eqref{eq::top::1}.
For the quantity $\delta^{j}_{n,h,s}$, the slowest rate occurs at $j=2$ and $s=d$.
Thus, 
there are some constants $c_1,c_2>0$ such that 
\begin{equation}
\max_{j=0,1,2}\max_{s=0,\cdots, d}|\delta^{(j)}_{n,h,s}| < c_1 h^{2\bigwedge m(x)} + c_2 Z_d,
\label{eq::top::p3}
\end{equation}
where $Z_d = O_P\left(\sqrt{\frac{\log n}{nh^{d+4}}}\right)$ is the stochastic variation.
Recall that in the proof of Theorem~\ref{thm::KDE}, $Z_d$ is 
$$
Z_d = \sup_{x\in \K}\|\nabla\nabla\hat{p}_n(x)-\nabla\nabla p_h(x)\|_{\max}.
$$
By Assumption (K2) and Talagrand's inequality \citep{Gine2002,Einmahl2005}, 
there are constants $c_3, c_4>0$ such that when $\frac{nh^{d+4}}{\log n}\rightarrow \infty$
\begin{equation}
P(Z_d>t) <c_3\cdot e^{-c_4\cdot nh^{d+4}\cdot t^2}.
\label{eq::top::p4}
\end{equation}
Thus,  
using equations \eqref{eq::top::p3} and \eqref{eq::top::p4},
when $h^{2\bigwedge m(x)}< \frac{a_0}{c_1}$,
\begin{align*}
P(T_{\hat{\alpha}_n}\overset{T}{\approx} T_\alpha) 
&\geq P(\max_{s=0,\cdots, d}|\delta^{(2)}_{n,h,s}| < a_0) \\
&\geq1-P\left(Z_d>\frac{a_0}{c_2}\right)\\
&\geq 1-c_3 \cdot e^{-c_4\cdot nh^{d+4}\cdot (\frac{a_0}{c_2})^2} \\
&= 1-c_3 \cdot e^{-c_5\cdot nh^{d+4}},
\end{align*}
where $c_5 = c_4\cdot (\frac{a_0}{c_2})^2>0$.
This proves the desired result.
\end{proof}


\bibliographystyle{abbrvnat}
\bibliography{GTree.bib}

\end{document}